\documentclass[12pt,twoside,reqno,centertags,draft,tikz,border=3mm]{amsart}

\usepackage{amsmath,amsthm,amsfonts,amssymb,amsrefs,
ascmac, 
mathrsfs, 
}
\usepackage{enumerate,color}

\usepackage{etoolbox} 

\makeatletter

\@addtoreset{equation}{section}
\makeatother

\theoremstyle{plain} 
\newtheorem{thm}{Theorem}[section]
\newtheorem{lem}[thm]{Lemma}
\newtheorem{prop}[thm]{Proposition}

\theoremstyle{definition} 
\newtheorem{definition}[thm]{Definition}
\newtheorem{rem}[thm]{Remark}

\newcommand{\R}{\mathbb{R}}

\newcommand{\supp}{\rm{supp}}

\def\supp{\, \hbox{\rm supp}\,  }

\let\tilde=\widetilde

\newcommand{\eqsp}[1]{{\begin{equation}\begin{aligned}#1\end{aligned}\end{
equation}}}

\title[Global dynamics for the Hardy-Sobolev equation]{Well-posedness and global dynamics for the critical Hardy-Sobolev parabolic equation}
\date\today

\author[N. Chikami, M. Ikeda and K. Taniguchi]{Noboru Chikami, Masahiro Ikeda and Koichi Taniguchi}
\address[N. Chikami]
{
Graduate School of Engineering, 
Nagoya Institute of Technology, 
Gokiso-cho, Showa-ku, Nagoya 
466-8555, Japan.}
\email{chikami.noboru@nitech.ac.jp}
\address[M. Ikeda]
{Faculty of Science and Technology,
Keio University, 
3-14-1 Hiyoshi, Kohoku-ku, Yokohama, 223-8522, Japan/ Center for Advanced Intelligence Project
RIKEN, Japan.}
\email{masahiro.ikeda@keio.jp/masahiro.ikeda@riken.jp}
\address[K. Taniguchi]
{Advanced Institute for Materials Research,
Tohoku University,
2-1-1 Katahira, Aoba-ku, Sendai, 980-8577, Japan.}
\email{koichi.taniguchi.b7@tohoku.ac.jp}

\begin{document}

\maketitle

\begin{abstract}
We study the Cauchy problem for the semilinear heat equation with 
the singular potential, called the Hardy-Sobolev parabolic equation, in the energy space. 
The aim of this paper is to determine a necessary and sufficient condition on initial data below or at the ground state, under which 
the behavior of solutions is completely dichotomized. More precisely, the solution exists globally in time and its energy decays to zero in time, or it blows up 
in finite or infinite time. 
The result on the dichotomy for the corresponding Dirichlet problem is also shown as a by-product via comparison principle.
\end{abstract} 


\section{Introduction}\label{sec:1}
\subsection{Introduction and setting}
We consider the Cauchy problem of the critical Hardy-Sobolev parabolic equation
\begin{equation}\label{crtHS}
	\begin{cases}
		\partial_t u - \Delta u = |x|^{-\gamma} |u|^{2^*(\gamma)-2} u,&(t,x)\in (0,T)\times\R^d, \\
		u(0) = u_0
	\end{cases}
\end{equation}
in spatial dimensions $d\ge3$ 
with initial data $u_0$ in the energy space $\dot{H}^1(\R^d)$,  defined by 
\[
	\dot{H}^1(\R^d)
	:= 
	\left\{ f \in L^{q_c} (\R^d)\ ;\ \|f\|_{\dot H^1} = \left(\int_{\R^d}|\nabla f(x)|^2\, dx \right)^\frac12 < \infty \right\},
	\quad q_c := \frac{2d}{d-2},
\]
where $T>0$, $\gamma\in [0,2)$, and $2^*(\gamma)$ is the critical Hardy-Sobolev exponent, i.e.,
\[
	2^*(\gamma):= \frac{2(d-\gamma)}{d-2}. 
\]
Here, $\partial_t:=\partial/\partial t$ is the time derivative, 
$\nabla := (\partial/\partial x_1, \ldots, \partial/\partial x_d)$ is the vector differential operator, 
$\Delta:=\sum_{j=1}^d\partial^2/\partial x_j^2$ is the Laplace operator on $\R^d$, 
$u=u(t,x)$ is an unknown complex-valued function on $(0,T)\times \mathbb R^d$, and $u_0=u_0(x)$ is a prescribed complex-valued function on $\mathbb R^d$. The equation \eqref{crtHS} with $\gamma>0$ is known as a {\it Hardy parabolic equation} or a {\it Hardy-Sobolev parabolic equation}, while that with $\gamma<0$ is known as a {\it H\'enon parabolic equation}. 
The case $\gamma=0$ corresponds to 
a heat equation with a standard power-type nonlinearity, often called the {\it Fujita equation}, 
which has been extensively studied in various directions. 
In the case $\gamma\ne0$, the equation \eqref{crtHS} 
is not invariant under the translation with respect to space variables,
owing to the existence of the space-dependent potential.
Furthermore, in the case $\gamma>0$, the equation \eqref{crtHS} has no non-trivial classical solution, as the potential has a singularity at the origin. 
The elliptic part of \eqref{crtHS}, that is,
\begin{equation}\label{statio}
    -\Delta \phi=|x|^{-\gamma}|\phi|^{2^*(\gamma)-2}\phi,\ \ \ x\in \R^d,
\end{equation}
was proposed by H\'enon as a model to study the rotating stellar systems (see \cite{H-1973}), and has been extensively studied in the mathematical context, especially in the field of nonlinear analysis and variational methods (see \cite{GhoMor2013} for example). 

In this study, we address the equation \eqref{crtHS} with $\gamma>0$ in the energy space $\dot H^1(\R^d)$.
The {\it total energy} (or simply {\it energy}) functional $E_{\gamma}$ is defined by 
\[
	E_{\gamma}(f) 
	:= \frac{1}{2} \|f\|_{\dot H^1}^2 - \frac{1}{2^*(\gamma)} 
	\int_{\R^d} \frac{ |f(x)|^{2^*(\gamma)} }{|x|^\gamma} dx,\quad f \in \dot H^1(\R^d),
\]
where the first and second terms correspond to the {\it kinetic} and {\it potential energies}, respectively. 
The energy of solution is (formally) dissipated: 
\begin{equation}\label{energy-diss}
	\frac{d}{dt}E_{\gamma}(u(t)) = -\int_{\mathbb R^d} |\partial_t u(t,x)|^2\, dx \le 0.
\end{equation}
Moreover, the equation \eqref{crtHS}, and the total energies, kinetic energies, and potential energies of its solutions are invariant under the scaling transformation $u\mapsto u_{\lambda}$ for $\lambda>0$, which is defined by
\begin{equation}\label{crt.HSscale}
	u_{\lambda}(t,x):=\lambda^{\frac{2-\gamma}{2^*(\gamma)-2}}u(\lambda ^2t,\lambda x) = \lambda^{\frac{d-2}{2}}u(\lambda ^2t,\lambda x). 
\end{equation}
Thus, the problem \eqref{crtHS} is called {\it energy critical}, and the space $\dot H^1(\R^d)$ (as well as $L^{q_c}(\R^d)$) is often called a {\it scaling critical} space.
We say that the problem is {\it energy subcritical} ({\it energy supercritical} resp.) if the power $p$ of the nonlinearity $|x|^{-\gamma}|u|^{p-2}u$ is strictly less than (strictly greater than resp.) the critical exponent $2^*(\gamma)$. 

Our interest is a problem on global behavior in time of solutions to \eqref{crtHS}. 
The equation \eqref{crtHS} has a nonlinearity that works as a source term. 
In general, there are various behaviors of solutions to partial differential equations 
with the source term, depending on the choice of the initial data. 
Thus, it requires great effort to completely classify the behavior of solutions based on the initial data. 
However, below the energy of {\it ground state} (or the {\it mountain pass energy}), it is often possible to obtain a necessary and sufficient condition on the initial data, under which the behavior of solution is completely dichotomized into {\it dissipative} and {\it blow-up}. 
Here, the ground state is a minimal-energy non-trivial solution to the corresponding stationary problem, and 
the mountain pass energy 
coincides with that of the ground state. 
In the energy-subcritical case, this type of dichotomy has been extensively studied for various partial differential equations 
(see \cites{GW-2005,Gig86,IT-arxiv,IS-1996,Ish2007,PaySat1975,Sat1968,T-1972} and references therein). 
However, the problem in the energy-critical case is more delicate due to the lack of compactness. 
Proof of this dichotomy was advanced by 
Kenig and Merle \cites{KM-2006,KM-2008} for focusing semilinear Schr\"odinger equations and wave equations on $\mathbb R^d$ with $d=3,4,5$, and 
by Ishiwata \cite{Ish2008} for nonlinear parabolic equations involving the $p$-Laplacian on $\mathbb R^d$ with $d\ge2$, 
including the equation \eqref{crtHS} with $\gamma=0$ and $d\ge3$ as a typical case, based on the argument of concentration compactness. 
Recently, Roxanas also obtained the dichotomy for the harmonic map heat flow and the four-spatial-dimensional energy-critical heat equation 
(see \cite{Rox2017} and also \cite{GR-2018}). 

Recently, it has been also studied for several partial differential equations with the same type of nonlinear term as in the equation \eqref{crtHS}. 
In particular, several results on the dichotomy for semilinear Schr\"odinger equations 
\begin{equation}\label{schro}
i\partial_t u +\Delta u = -  |x|^{-\gamma} |u|^{p-2} u,\quad (t,x)\in (0,T)\times\R^d
\end{equation}
have been obtained (see \cites{CFGM_arxiv,CL_arxiv} and references therein). 
For example, Cho and Lee proved a scattering result for the energy-critical semilinear Schr\"odinger equation \eqref{schro} on $\R^3$ with $p=2^{*}(\gamma)$, 
where they dealt exclusively with symmetric solutions (see \cite{CL_arxiv}). 
In the non-radial case, the latest result have been obtained by Cardoso, Farah, Guzm\'an, and Murphy for the energy-subcritical Schr\"odinger equations \eqref{schro} with $p<2^{*}(\gamma)$ (see \cite{CFGM_arxiv}).
However, to our knowledge, similar research on the Hardy-Sobolev parabolic equation does not exist. 
Therefore, this study aims to provide the necessary and sufficient condition on initial data with energy less than or equal to that of the ground state for \eqref{crtHS}.

\subsection{Statement of the result}\label{sub:1.2}
To state our result, let us introduce the notion of solution in this paper. 
We study the problem \eqref{crtHS} via the integral form 
\begin{equation}\label{integral-equation}
	u(t,x) = (e^{t\Delta} u_0)(x) + \int_{0}^t e^{(t-\tau)\Delta} \{|\cdot|^{-\gamma} |u(\tau,\cdot)|^{2^*(\gamma)-2} u(\tau,\cdot)\}(x)\,d\tau,
\end{equation}
where 
$\{e^{t\Delta}\}_{t>0}$ is the linear heat semigroup, defined by 
\[
	(e^{t\Delta} f)(x) := (G(t, \cdot) * f)(x) = \int_{\R^d} G(t,x-y)f(y)\, dy,\quad t>0,\ x\in \R^d,
\]
and $G$ is the heat kernel, i.e., 
\begin{equation}\label{gaussian-kernel}
	G(t,x) := (4 \pi t)^{-\frac{d}{2}} e^{-\frac{|x|^2}{4t}},\quad t>0,\  x\in \mathbb R^d.
\end{equation}
We say that a function $u=u(t,x)$ is a (mild) solution to \eqref{crtHS} on $[0,T) \times \mathbb R^d$ with initial data $u_0\in\dot H^1(\R^d)$ if 
$u \in C([0,T'] ; \dot H^1(\R^d))$ satisfies the integral equation \eqref{integral-equation} for any $T'\in (0,T)$,
where $T \in (0,\infty]$. When $T<\infty$, the solution $u$ is called local in time. 
We denote by $T_m=T_m(u_0)$ the maximal existence time of solution with initial data $u_0$. 
We say that $u$ is global in time if $T_m = + \infty$ and that $u$ blows up in finite time otherwise. 
Moreover, we say that $u$ is dissipative if $T_m = + \infty$ and 
\[
	\lim_{t\to\infty} \|u(t)\|_{\dot H^1} = 0,
\]
that $u$ grows up at infinite time if 
$T_m = + \infty$ and 
\[
	\limsup_{t\to\infty} \|u(t)\|_{\dot H^1} = + \infty,
\]
and that $u$ is stationary if $u(t,x) = \phi(x)$ on $[0,\infty)\times\R^d$, where $\phi$ is a solution of the elliptic equation \eqref{statio}. 

The problem on local well-posedness, i.e., existence of local in time solution, uniqueness, and continuous dependence on initial data, for \eqref{crtHS} 
has been studied in the space $L^{q}(\R^d)$ and the space of continuous bounded functions on $\R^d$ (see \cites{BenTayWei2017,Chi2019,Tay2020,Wan1993}). 
Particularly, Slimene, Tayachi, and Weissler proved local well-posedness, except the uniqueness, for \eqref{crtHS} in the scaling critical space $L^{q_c}(\R^d)$ (see \cite{BenTayWei2017}). Recently, the unconditional uniqueness for \eqref{crtHS} in $C([0,T]; L^{q_c}(\R^d))$ has been proven by Tayachi \cite{Tay2020}, and 
local well-posedness has been studied in scaling critical Besov spaces by Chikami \cite{Chi2019}. 
Similarly to these works, we can obtain local well-posedness for \eqref{crtHS} in the energy space $\dot H^1(\R^d)$, but a more detailed argument is required to justify the energy identity \eqref{energy-diss}. The details are given in Subsections~\ref{sub:2.3} and~\ref{sub:2.4} below (see also Appendix \ref{app:A}).

In addition, we provide some notations and definitions. 
The function 
\begin{equation}\label{Talenti}
	W_\gamma(x) := \left((d-\gamma) (d-2) \right)^{\frac{d-2}{2(2-\gamma)}} (1+|x|^{2-\gamma})^{-\frac{d-2}{2-\gamma}}
\end{equation}
is a ground state of \eqref{crtHS}. By invariance of \eqref{crtHS}, its scaling and rotation
\[
e^{i\theta_0} \lambda_0^{\frac{d-2}{2}}W(\lambda_0 x),\quad \lambda_0>0,\ \theta_0 \in \mathbb R
\]
 is also a ground state.
We introduce the Nehari functional $J_\gamma$ and the Nehari manifold $\mathcal N_\gamma$ by 
\begin{equation}\label{Nehari}
	J_\gamma(\phi) := \frac{d}{d\lambda} E(\lambda \phi) \Big|_{\lambda=1} 
	= \|\phi\|_{\dot{H}^1}^2 
	- \int_{\R^d} \frac{ |\phi(x)|^{2^*(\gamma)} }{|x|^\gamma} dx,
\end{equation}
\[
	\mathcal N_{\gamma} := \{\dot{H}^1(\mathbb R^d) \setminus \{0\}\ ;\, J_\gamma(\phi) = 0\},
\]
respectively. Then, the mountain pass energy $l_{HS}$ is given by
\begin{equation}\label{mini}
	l_{HS} 
	:= \inf_{\phi\in \dot{H}^1(\mathbb R^d) \setminus \{0\}} 
	\max_{\lambda\ge0} E(\lambda \phi)
	= \inf_{\phi \in \mathcal N_{\gamma}} E_\gamma(\phi).
\end{equation}
This $l_{HS}$ coincides with the energy $E_\gamma(W_\gamma)$ of the ground state, as mentioned above (see Remark \ref{rem:ground} below). 

\medskip

Our main result is the following.

\begin{thm}
\label{thm:GD}
Let $d\ge 3$, $0< \gamma<2$, and $u=u(t)$ be a solution to \eqref{crtHS} with initial data $u_0\in \dot{H}^1(\mathbb R^d)$. 
Assume $E_\gamma(u_0) \le l_{HS}$. Then, the following statements hold{\rm :}
\begin{itemize}
	\item[(i)] If $J_\gamma(u_0) > 0$, then $u$ is dissipative. 

	\item[(ii)] If $J_\gamma(u_0) < 0$, then $u$ blows up in finite time or grows up at infinite time. Furthermore, if $u_0\in L^2(\R^d)$ is also satisfied, then $u$ blows up in finite time.
\end{itemize}
\end{thm}

\begin{rem}
There is no function $u_0 \in \dot{H}^1(\R^d) \setminus \{0\}$ such that $E_\gamma(u_0) < l_{HS}$ and $J_\gamma(u_0) = 0$.
When $E_\gamma(u_0) = l_{HS}$ and $J_\gamma(u_0) = 0$, the solution $u$ is always a ground state (see Remark \ref{rem:ground} below).
\end{rem}

\begin{rem} 
As a corollary of (ii) in Theorem \ref{thm:GD}, it can be immediately obtained that all solutions to \eqref{crtHS} with negative energy initial data in the inhomogeneous Sobolev space $H^1(\R^d)$ blow up in finite time. 
In fact, it is readily seen that $u_0\in \dot H^1(\R^d)\setminus \{0\}$ with $E_\gamma(u_0)\le 0$ implies $J_\gamma (u_0)<0$, as 
\[
E_\gamma (u_0) = \frac12 J_\gamma (u_0) + \frac{2-\gamma}{2(d-\gamma)} 
\int_{\R^d} \frac{ |u_0(x)|^{2^*(\gamma)} }{|x|^\gamma} dx.
\]
\end{rem}

\subsection{Outline of the proof and contributions}
Herein, the outline of proof of Theorem \ref{thm:GD} is outlined and contributions of this paper are described.

\medskip

\noindent{\bf Statement (i)}. The statement (i) is the main 
contribution of this paper, and 
in its proof, some new difficulties arise from the existence of the space-dependence singular potential.
The strategy of proof of (i) is based on the argument of concentration compactness and rigidity by \cite{KM-2006}. 
More precisely, under an assumption of proof by contradiction, 
we construct 
a minimal-energy blow-up solution $v^c$ to \eqref{crtHS} with $J_\gamma(v^c)>0$ by using  
the perturbation result in Proposition \ref{prop:perturbation} and 
the linear profile decomposition in Proposition \ref{prop:profile}. 
In the construction of $v^c$, a new difficulty peculiar to the case wherein $\gamma>0$ 
arises because of the translation symmetry breaking of \eqref{crtHS} with respect to space variables.
Most existing literature on such equations have dealt 
exclusively with radially symmetric solutions, as it works effectively and eliminates the difficulty. 
However, we confront the non-radial case and use Lemma~\ref{lem:key1} 
to solve the difficulty. 
After constructing $v^c$, we show that it must be identically zero using the rigidity argument, which derives a contradiction and concludes the statement (i). 
Here, we emphasize that our rigidity argument might be more simple and flexible than 
that in existing literature and 
does not require the backward uniqueness often used in such research on parabolic equations (see \cite{GR-2018} and references therein).

\medskip

\noindent{\bf Statement (ii)}. 
The former part of (ii) is also a main contribution of this paper, because it is unknown whether it holds in the case $\gamma=0$, or not.
The proof of (ii) is based on Levine's concavity method, which reduces to an argument for an ordinary differential inequality.
In particular, in our proof of the former part, we employ the concavity method for the spatially localized solution. 
Lemma~\ref{lem:key-blowup} is essential to realize this method. 
By taking advantage of the decay effect of $|x|^{-\gamma}$ at infinity, we can prove this lemma and obtain a better result than for the case $\gamma=0$. 
As for the latter part, we simply apply the concavity method for the solution itself (see, e.g., Theorem~1.2 in \cite{GR-2018}).

\medskip

\noindent{\bf Energy identity}. 
In the proofs of (i) and (ii), we often use the energy identity \eqref{energy-diss}. 
This is formally obtained by multiplying \eqref{crtHS} by $\overline{\partial_t u}$ and integrating it over $\mathbb R^d$; however, its validity is non trivial. 
To prove the validity, we need to know the integrability of solutions to \eqref{crtHS} in more detail, which will be discussed in Subsection~\ref{sub:2.4}. This is also our contribution.

\subsection{The absorbing case and Dirichlet problem}
In this paper, we also study two related problems. 
The first one is the Cauchy problem in 
the absorbing case: 
\begin{equation}\label{crtHS-a-intro}
	\begin{cases}
		\partial_t u - \Delta u = -|x|^{-\gamma} |u|^{2^*(\gamma)-2} u,&(t,x)\in (0,T)\times\R^d, \\
		u(0) = u_0.
	\end{cases}
\end{equation}
Here, the nonlinearity in \eqref{crtHS-a-intro} works as an absorbing term. 
We will show that all solutions to \eqref{crtHS-a-intro} are dissipative in the scaling critical space $L^{q_c}(\R^d)$ by using almost the same argument as that in the proof of (i) in Theorem~\ref{thm:GD} (see Subsection~\ref{sub:4.1}). 

The second one is the Dirichlet problem of the energy-critical Hardy-Sobolev parabolic equation
\begin{equation}\label{crtHS-D-intro}
	\begin{cases}
		\partial_t u - \Delta u = |x|^{-\gamma} |u|^{2^*(\gamma)-2} u,&(t,x)\in (0,T)\times\Omega, \\
		u|_{\partial \Omega} = 0,\\
		u(0) = u_0,
\end{cases}
\end{equation}
where $\Omega$ is a domain of $\R^d$ that contains the origin. 
Then, we will extend Theorem~\ref{thm:GD} to the result on dichotomy for the Dirichlet problem \eqref{crtHS-D-intro} 
through the comparison principle in Lemma~\ref{lem:comparison} (see Subsection \ref{sub:4.2}),
where we require no geometrical assumption on $\Omega$, such as boundedness, smoothness, and convexity.

\section{Preliminaries}\label{sec:2}

\subsection{The Hardy-Sobolev inequality}\label{sub:2.1}
The Hardy-Sobolev inequality plays a fundamental role throughout this paper. 
In this subsection, we summarize the basic results on this inequality and its minimization problem.

\begin{lem}[\cite{Lie1983}, Theorems 15.1.1 and 15.2.2 in \cite{GhoMor2013}]
\label{t:HrdySob2}
Let $d\ge3$, $0\le \gamma\le 2$, and $\Omega$ be a domain in $\mathbb R^d$. 
Then, the inequality 
\begin{equation}\label{HS-ineq}
	\left(\int_{\Omega} \frac{ |f(x)|^{2^*(\gamma)} }{|x|^\gamma} dx \right)^{\frac{1}{ 2^*(\gamma)} }
	\le C_{HS} 
	\left(\int_{\Omega} |\nabla f(x)|^{2} dx \right)^{\frac{1}{ 2}}
\end{equation}
holds for any $f\in H^{1}_0(\Omega)$, 
where $C_{HS}=C_{HS}(d,\gamma,\Omega)$ is the best constant. 
Furthermore, $C_{HS}$ is independent of $\Omega$ whenever $0\in \Omega$, and is attained only if $\Omega=\mathbb R^d$ with the extremal $W_\gamma$ given in \eqref{Talenti}. 
\end{lem}

\begin{rem}\label{rem:ground}
In the case $\Omega=\mathbb R^d$, 
the inequality \eqref{HS-ineq} holds for any $f\in \dot H^1(\R^d)$, and the equality in \eqref{HS-ineq} holds if and only if $f$ is a ground state of 
the elliptic equation \eqref{statio}. 
Moreover, the mountain pass energy $l_{HS}$ defined by \eqref{mini} coincides with
the energy $E_\gamma(W_\gamma)$ of $W_\gamma$, and is represented by the best constant $C_{HS}$ as follows:
\[
	l_{HS}=E_\gamma(W_\gamma)=\frac{2-\gamma}{2(d-\gamma)} C_{HS}^{\frac{2(d-\gamma)}{2-\gamma}}.
\]
\end{rem}

{See Chapter 15 in \cite{GhoMor2013} for more details on the minimization problem for \eqref{HS-ineq} and the stationary problem \eqref{statio}.

\subsection{Smoothing and decay estimates for heat semigroup}\label{sub:2.2}
Let us recall the definition of the linear heat semigroup $\{e^{t\Delta}\}_{t>0}$:
\[
	e^{t\Delta} f := G(t, \cdot) * f,\quad t>0,
\]
where $G$ is the heat kernel given by $G(t,x) := (4 \pi t)^{-d/2} e^{-|x|^2/(4t)}$ for $t>0$ and $x\in \mathbb R^d$.
We prepare smoothing and decay estimates for $\{e^{t\Delta}\}_{t>0}$. 
\begin{lem}
\label{l:lin.est.HS}
Let $d\ge1$ and $\alpha = (\alpha_1, \ldots, \alpha_d)$ be a multi-index with $|\alpha|=0,1$. 
Then, the following statements hold{\rm :}
\begin{itemize}
	\item[(i)]
	Let $1\le p_1 \le p_2 \le \infty.$ 
	Then, there exists a constant $C=C(d,\alpha,p_1,p_2)>0$ such that 
	\begin{equation}\label{decayest1}
		\| \partial_x^\alpha e^{t\Delta} f\|_{L^{p_2}} 
		\le C  t^{-\frac{d}{2}(\frac{1}{p_1}-\frac{1}{p_2})-\frac{|\alpha|}{2}} 
		\| f\|_{L^{p_1}}
	\end{equation}
	for any $t>0$ and $f \in L^{p_1}(\R^d)$, where $\partial_x^\alpha = \partial_{x_1}^{\alpha_1}\cdots \partial_{x_d}^{\alpha_d}$.
	
	\item[(ii)]
	Let $p_1,p_2,\gamma$ be such that 
	\[
		0 < \gamma < d,\quad 0 \le \frac{1}{p_2} < \frac{\gamma}{d} + \frac{1}{p_1} < 1.
	\]
	Then $\partial_x^\alpha e^{t\Delta}|\cdot|^{-\gamma} : L^{p_1}(\R^d) \to L^{p_2}(\R^d)$ 
	is a bounded map {\rm (}replace $L^{p_2}(\R^d)$ by $C_0(\R^d)$ if $p_2=\infty${\rm )} and 
	there exists a constant $C = C(d, \gamma,\alpha, p_1, p_2)>0$ such that 
	\begin{equation}\label{decayest2}
		\|\partial_x^\alpha e^{t\Delta} (|\cdot|^{-\gamma} f) \|_{L^{p_2}} 
		\le C t^{-\frac{d}{2}(\frac{1}{p_1}-\frac{1}{p_2})-\frac{|\alpha|+\gamma}{2}} \|f\|_{L^{p_1}}
	\end{equation}
	for any $t>0$ and $f \in L^{p_1}(\R^d).$
\end{itemize}
\end{lem}
The statement (i) is well known, and the statement (ii) is a combination of (i) and H\"older's inequality. 
The proof of (ii) can be found in Proposition 2.1 from \cite{BenTayWei2017}.

\subsection{Local well-posedness}\label{sub:2.3}
We summarize the results on local well-posedness, small-data global existence, and dissipation of global solutions 
for \eqref{crtHS} in $\dot{H}^1(\R^d)$. 
For this purpose, let us introduce an auxiliary space as the following type.

\begin{definition}\label{def:Kato}
Let $T \in (0,\infty]$, $q\in [1,\infty]$, and $\alpha\in \R$. The space $\mathcal{K}^{q,\alpha}(T)$ is defined by 
\[
   \mathcal{K}^{q,\alpha}(T):=\left\{u\in \mathscr{D}'([0,T)\times\mathbb R^d)\ ;\ \|u\|_{\mathcal{K}^{q,\alpha}(T')}
   <\infty\ \text{for any }T' \in (0,T)\right\}
\]
endowed with 
\[
\|u\|_{\mathcal K^{q,\alpha}(T)}
	:=\sup_{0\le t\le T}t^{\frac{d}{2}(\frac{1}{q_c}-\frac{1}{q})+\alpha}\|u\|_{L^q},
\]
where $ \mathscr{D}'([0,T)\times\mathbb R^d)$ is the space of distributions on $[0,T)\times\mathbb R^d$. 
We simply write $\mathcal{K}^{q}(T)=\mathcal{K}^{q,0}(T)$ when $\alpha=0$, and $\mathcal{K}^{q,\alpha}=\mathcal{K}^{q,\alpha}(\infty)$ and $\mathcal{K}^{q}=\mathcal{K}^{q}(\infty)$
when $T=\infty$ if they do not cause a confusion. 
\end{definition}

\begin{prop}\label{prop:wellposed1}
Let $d\ge3$ and $0\le \gamma< 2$. Assume that $q\in (1,\infty)$ satisfies 
\begin{equation}\label{l:crtHS.nonlin.est:c2}
	\frac{1}{q_c} - \frac{1}{d(2^*(\gamma)-1)}
	< \frac{1}{q} < \frac{1}{q_c}.
\end{equation} 
Then, the following statements hold{\rm :}
\begin{itemize}
\item[(i)] {\rm (}Existence{\rm )} 
For any $u_0 \in \dot{H}^1(\R^d)$, there exists a maximal existence time $T_m=T_m(u_0)\in (0,\infty]$ such that 
there exists a unique mild solution 
\[
	u\in C([0,T_m); \dot{H}^1(\R^d))\cap \mathcal K^q(T_m)
\]
to \eqref{crtHS} with $u(0)=u_0$. Moreover, the solution $u$ also satisfies 
\[
\|u\|_{\mathcal K^{\tilde{q}}_{\tilde{r}}(T,\Omega)} := 
	\left(
		\int_0^T (t^{\kappa} \|u(t)\|_{L^{\tilde{q}}(\Omega)})^{\tilde{r}} \, dt
	\right)^\frac{1}{\tilde{r}} < \infty
\]
for any $T \in (0,T_m)$ and for any $\tilde{q}, \tilde{r} \in [1,\infty]$ satisfying \eqref{l:crtHS.nonlin.est:c2} and 
\begin{equation}\label{condi-new}
	0\le  \frac1{\tilde r}  < \frac{d}2 \left(\frac1{q_c} - \frac1{\tilde q} \right),
\end{equation}
respectively, 
where $\kappa$ is given by 
\[
\kappa = \kappa(\tilde{q},\tilde{r}) := \frac{d}{2} \left( \frac{1}{q_c} - \frac{1}{\tilde q}\right) - \frac{1}{\tilde r}.
\]

\item[(ii)] 
{\rm (}Uniqueness in $\mathcal{K}^{q}(T)${\rm )} 
Let $T>0.$ If $u_1, u_2 \in \mathcal{K}^{q}(T)$ 
satisfy the integral equation \eqref{integral-equation} 
with $u_1(0)=u_2(0)=u_0$, then $u_1=u_2$ on $[0,T].$ 

\item[(iii)] {\rm (}Continuous dependence on initial data{\rm )} 
The map $T_m : \dot{H}^1(\R^d) \to (0,\infty]$ is 
lower semicontinuous. Furthermore, for any $u_0, v_0 \in \dot{H}^1(\R^d)$ and for any $T < \min\{T_m(u_0),T_m(v_0)\}$, there exists a constant $C>0$, depending on $\|u_0\|_{\dot H^1}$, $\|v_0\|_{\dot H^1}$, and $T$, such that
\[
	\sup_{t \in [0,T]} \|u(t) - v(t)\|_{\dot H^1} 
	+
	\|u-v\|_{\mathcal K^q(T)} \le C \|u_0 - v_0\|_{\dot H^1}.
\]

\item[(iv)] {\rm (}Blow-up criterion{\rm )} 
If $T_m < + \infty,$ then $\|u\|_{\mathcal{K}^q(T_m)} = \infty.$ 

\item[(v)] {\rm (}Small-data global existence and dissipation{\rm )} 
There exists $\rho >0$ such that 
if $u_0 \in \dot{H}^1(\R^d)$ satisfies
$\|e^{t\Delta} u_0\|_{\mathcal{K}^{q}}\le \rho$, 
then $T_m=+\infty$ and 
\[
\|u\|_{\mathcal{K}^{q}} \le 2\rho \quad \text{and}\quad \lim_{t\to\infty}\|u(t)\|_{\dot H^1} = 0.
\]

\item[(vi)] {\rm (}Dissipation of global solutions{\rm )} 
The following statements are equivalent{\rm :}
\begin{itemize}
	\item[(a)] $T_m=+\infty$ and $\|u\|_{\mathcal K^q} < \infty$.
	\item[(b)] $\lim_{t\to T_m} \|u(t)\|_{\dot H^1}=0$.
	\item[(c)] $\lim_{t\to T_m} t^{\frac{d}{2} (\frac{1}{q_c}-\frac{1}{q})} \|u(t)\|_{L^{q}}=0$.
\end{itemize}

\item[(vii)] Let $d=3$. Suppose that $q$ satisfies the additional assumption
\begin{equation}\label{l:crtHS.nonlin.est:c3}
	\frac{1}{q_c} - \frac{1}{12(2-\gamma)}
	< \frac{1}{q}. 
\end{equation} 
Then, for any $u_0 \in \dot H^1(\R^3)$, there exists a maximal existence time $T_m=T_m(u_0)\in (0,\infty]$ such that 
there exists a unique mild solution 
\[
	u\in C([0,T_m); \dot{H}^1(\R^d))\cap \mathcal K^q(T_m)\quad \text{and}\quad \partial_t u \in \mathcal K^{3,1}(T_m)
\]
to \eqref{crtHS} with $u(0)=u_0$. Furthermore, the solution $u$ satisfies 
\[
	\partial_t u \in \mathcal K^{2,1}(T_m).
\]
\end{itemize}
\end{prop}

The statements (i)--(vi) are known, but the last statement (vii) is a new ingredient, which is utilized to justify the energy identity in Subsection \ref{sub:2.4} below. 
The proof is given in Appendix~\ref{app:A}. 

\begin{rem}\label{rem:classical}
It is generally impossible to obtain classical solutions for \eqref{crtHS}. However, 
mild solutions $u$ to \eqref{crtHS} given in Proposition \ref{prop:wellposed1} are continuous and bounded on $\R^d$ for each $t \in (0,T_m)$, and 
belong to 
\[
	u \in C_{\mathrm{loc}}^{1,2}((0,T_m)\times (\R^d\setminus\{0\})) \cap C_{\mathrm{loc}}^{\frac{\alpha}{2},\alpha}((0,T_m)\times \R^d)
\]
for $\alpha \in (0,2-\gamma)$ by the regularity theory for parabolic equations. 
Here, $C_{\mathrm{loc}}^{\alpha,\beta}(I\times \Omega)$ is the space of functions that are locally H\"older continuous with exponent $\alpha\ge0$ in $t\in I$ and exponent $\beta\ge0$ in $x\in \Omega$ for an interval $I\subset (0,\infty)$ and a domain $\Omega\subset\mathbb R^d$. 
See Remark~1.1 and Proposition 3.2 in \cite{BenTayWei2017} (see also the remark after Definition 2.1 in \cite{Wan1993} on page 563).
\end{rem}

Moreover, we have the following stability result for \eqref{crtHS} (see Proposition \ref{prop:perturbation-A} in Appendix~\ref{app:A}).

\begin{prop}[Perturbation result]\label{prop:perturbation}
Let $d\ge3$ and $0\le \gamma <2$. and let $q,\tilde{q}\in [1,\infty]$ and $\tilde{r} \in [1,\infty)$ satisfy \eqref{l:crtHS.nonlin.est:c2} and \eqref{condi-new}, respectively. 
Assume that 
\begin{equation}\label{new-condition1}
\frac{1}{q_c} - \frac{\gamma}{d}
 - \frac{1}{q} < \frac{1}{\tilde q} <\frac{\gamma}{d(2^*(\gamma) - 2)}
\end{equation}
and 
\begin{equation}\label{new-condition2}
\frac{1}{\tilde r} < \min \left\{ \frac{d-2}{2} - \frac{d}{2 \tilde q}, \frac{1}{2^*(\gamma)-2} \right\}.
\end{equation}
Let 
$v$ satisfy 
\[
	\|v\|_{L^\infty([0,\infty); \dot H^1)} + 
	\|v\|_{\mathcal K^{\tilde q}_{\tilde r}} \le M,
\]
and the equation
\[
	\partial_t v - \Delta v = |x|^{-\gamma} |v|^{2^*(\gamma)-2}v + e
\]
with initial data $v(0) = v_0 \in \dot H^1(\R^d)$, where $e=e(t,x)$ is a function on $(0,\infty) \times \mathbb R^d$. 
Then there exist constants $\delta_0=\delta_0(M)>0$ and $C=C(M)>0$ such that 
the following assertion holds{\rm :}
If the error term $e$ and a function $u_0 \in \dot H^1(\mathbb R^d)$ satisfy
\[
	\delta := \| u_0 - v_0\|_{\dot H^1} + \left\|\int_0^t e^{(t-s)\Delta_\Omega}(e(s))\,ds\right\|_{\mathcal K^q} \le \delta_0,
\]
then there exists a unique solution $u$ to \eqref{crtHS} on $(0,\infty)\times \mathbb R^d$ with $u(0)=u_0$ satisfying 
\[
	\|u-v \|_{L^\infty([0,\infty); \dot H^1)\cap \mathcal K^q}\le C \delta.
\]
\end{prop}

\begin{rem}
We note that it is possible to take $(\tilde{q}, \tilde{r})$ satisfying \eqref{new-condition1} and \eqref{new-condition2}. 
In fact, the lower bound is always less than the upper bound in \eqref{new-condition1} if $q$ satisfies 
\[
\frac{1}{q_c} - \frac{\gamma}{d}
 - \frac{1}{q}
< \frac{\gamma}{d(2^*(\gamma) - 2)}\quad \text{i.e.}\quad 
\frac{1}{q_c} - \frac{\gamma}{d}\left( 1 + \frac{1}{2^*(\gamma) - 2}\right) < \frac{1}{q}.
\]
Moreover, the upper bound is always positive if 
\[
\frac{d-2}{2} - \frac{d}{2 \tilde q}>0\quad \text{i.e.}\quad \frac{1}{\tilde{q}} < \frac{1}{q_c}.
\]
\end{rem}

\subsection{Energy identity}\label{sub:2.4}
The energy identity \eqref{energy-diss}
plays an essential role in proving Theorem \ref{thm:GD}, 
and is formally obtained by multiplying the equation \eqref{crtHS} by $\overline{\partial_t u}$ and integrating it over $\mathbb R^d$. 
However, the validity of \eqref{energy-diss} is non trivial. 
In this subsection, we discuss this matter. 

\begin{prop}\label{prop:energy-id}
Let $u_0 \in \dot{H}^1(\R^d)$ and 
$t_0 \in (0,T_m)$. Then, the mild solution $u$ to \eqref{crtHS} with 
$u(0)=u_0$ satisfies the energy identity 
\begin{equation}\label{energy-id}
	E_\gamma(u(t))+ \int_{t_0}^t\int_{\R^d}|\partial_t u(\tau,x)|^2\, dx d\tau 
		= E_\gamma(u(t_0))
\end{equation}
for any $t \in [t_0, T_m)$.
Furthermore, the energy inequality 
\begin{equation}\label{energy-ineq}
	E_{\gamma}(u(t)) \le E_{\gamma}(u_0)
\end{equation}
holds for any $t \in [0, T_m)$.
\end{prop}

\begin{proof}
Let $u_0 \in \dot{H}^1(\R^d)$ and $u$ be a mild solution to \eqref{crtHS} with $u(0)=u_0$. 
To prove the validity of \eqref{energy-id}, we need to know the integrability of $\partial_t u$, $\Delta u$, and $|x|^{-\gamma} |u|^{2^*(\gamma)-2}u$. 
To begin with, we check 
the integrability of the nonlinear term. 
It is easily seen from Remark \ref{rem:classical} that 
\begin{equation}\label{proof:energy-id_2}
	|x|^{-\gamma} |u|^{2^*(\gamma)-2}u \in L^\infty_{\mathrm{loc}}((0,T_m);L^{\sigma_1}(|x|<1))
\end{equation}
for any $1\le \sigma_1<d/\gamma$. 
Since $u\in L^\infty(0,T; L^{q_c}(\R^d)),$ H\"older's inequality implies 
\begin{equation}\label{proof:energy-id_3}
	|x|^{-\gamma} |u|^{2^*(\gamma)-2}u\in L^\infty((0,T);L^{\sigma_2}(|x|\ge1))
\end{equation}
for any $\sigma_2 > 2d/(d+2)$. 
Let us divide the proof into two cases: 
\begin{itemize}
	\item[(a)] $d\ge4$ or $d=3$ and $0\le \gamma <3/2$; 
	\item[(b)] $d=3$ and $3/2 \le \gamma <2$.
\end{itemize}

\noindent{\bf Case (a)}: Let $t_0\in (0,T_m)$.
Then, we have $u(t_0) \in \dot{H}^1(\R^d)$ by Proposition \ref{prop:wellposed1},
and 
\begin{equation}\label{proof:energy-id_4}
	|x|^{-\gamma} |u|^{2^*(\gamma)-2}u \in L^2_{\mathrm{loc}}([t_0, T_m); L^2(\R^d))
\end{equation}
by \eqref{proof:energy-id_2} and \eqref{proof:energy-id_3} 
provided that $d\ge4$ or $d=3$ and $0\le \gamma <3/2$.
Hence, we can apply the maximal regularity for parabolic equations to obtain
\begin{equation}\label{proof:energy-id_5}
	\partial_t u, \Delta u \in L^2_{\mathrm{loc}}([t_0, T_m); L^2(\R^d)).
\end{equation}
Then, \eqref{proof:energy-id_4} and \eqref{proof:energy-id_5} ensure the energy identity \eqref{energy-id} for any $t \in [t_0, T_m)$.

\smallskip

\noindent{\bf Case (b)}: 
It follows from (vii) in Proposition \ref{prop:wellposed1} and \eqref{proof:energy-id_3} that
\[
	\partial_t u, \Delta u, |x|^{-\gamma} |u|^{2^*(\gamma)-2}u \in L^2_{\mathrm{loc}}((0,T_m) ; L^2(|x|\ge1)).
\]
Hence, multiplying the equation \eqref{crtHS} by $\overline{\partial_t u}$ and integrating it over $[t_0,t]\times \{|x|\ge1\}$
are justified. 
On the other hand, we see from (vii) in Proposition \ref{prop:wellposed1} and \eqref{proof:energy-id_2} that 
\[
	\partial_t u \in L^2_{\mathrm{loc}}((0,T_m) ; L^3(|x|<1)) \subset L^2_{\mathrm{loc}}((0,T_m) ; L^2(|x|<1)),
\]
\[
	|x|^{-\gamma} |u|^{2^*(\gamma)-2}u \in L^2_{\mathrm{loc}}((0,T_m) ; L^\frac32(|x|<1)).
\]
Then, we also have
\[
	\Delta u \in L^2_{\mathrm{loc}}((0,T_m) ; L^\frac32(|x|<1)),
\]
as $u$ satisfies the differential equation \eqref{crtHS} by Remark \ref{rem:classical}.
Hence, multiplying \eqref{crtHS} by $\overline{\partial_t u}$ and integrating it over $[t_0,t]\times \{|x|<1\}$
are also justified. 
The above argument ensures the energy identity \eqref{energy-id} for any $t \in [t_0, T_m)$.

Finally, it follows from \eqref{energy-id} that 
\begin{equation}\label{energy-ineq_t0}
	E_{\gamma}(u(t)) \le E_{\gamma}(u(t_0))
\end{equation}
for any $t_0 \in (0, T_m)$. Since the energy $E_{\gamma}(u(t))$ is continuous in $t\in [0,T_m)$, 
we have 
\[
	E_{\gamma}(u(t)) \le E_{\gamma}(u_0)
\]
by taking the limit of \eqref{energy-ineq_t0} as $t_0 \to 0$.
Thus, we conclude Proposition \ref{prop:energy-id}.
\end{proof}

\begin{rem}
The proof of case (a) cannot be applied to the case (b) 
as the nonlinear term does not necessarily satisfy 
\eqref{proof:energy-id_4} in the case (b).
For example, the ground state $W_\gamma$ given in \eqref{Talenti}, which is also a mild solution to \eqref{crtHS}, does not satisfy \eqref{proof:energy-id_4}, as 
$\Delta W_\gamma = |x|^{-\gamma} W_\gamma^{2^*(\gamma)-1} \in L^2(\R^d)$ if and only if $d\ge4$ or $d=3$ and $0\le \gamma<3/2$.
In contrast, we can perform the argument in the proof of case (b) only if $d=3$, as it relies on (vii) in Proposition \ref{prop:wellposed1} 
(see Remark \ref{rem:restriction-q} below). 
\end{rem}

\begin{rem}\label{rem:energy}
We do not know whether it is possible to take the limit as $t_0 \searrow 0$ of the integral
\[
\int_{t_0}^t\int_{\R^d}|\partial_t u(\tau,x)|^2\, dx d\tau.
\]
Hence, in discussing near $t=0$ in Subsection \ref{sub:2.5}, we will use the energy inequality \eqref{energy-ineq},
instead of the energy identity \eqref{energy-id}. 
\end{rem}

\subsection{Variational arguments}\label{sub:2.5}
We only consider the case $E_\gamma(u_0)<l_{HS}$, as the other case where $E_\gamma(u_0)=l_{HS}$ and $J_\gamma(u_0)\ne 0$ is reduced to this case. 
In fact, let $u_0 \in \dot H^1(\R^d)$ with $E_\gamma(u_0)=l_{HS}$ and $J_\gamma(u_0)\ne 0$. 
Suppose that there exists a time $t_1\in (0,T_m)$ such that $E_\gamma(u(t_1)) = E_\gamma(u_0)$. 
Then, by the energy inequality \eqref{energy-ineq}, we have $E_\gamma(u(t)) = E_\gamma(u_0)$ for any $t \in [0,t_1]$. 
Furthermore, for any $t_0 \in (0,t_1)$, the solution $u$ is stationary in the interval $[t_0,t_1]$ by the energy identity \eqref{energy-id}, 
and hence, $J_\gamma(u(t)) = 0$ for any $t\in [t_0,t_1]$. 
However this contradicts $J_\gamma(u_0)\ne 0$ and the continuity of $J_\gamma(u(t))$ in $t\in [0,t_1]$. 
Therefore, $E(u(t)) < E(u_0) = l_{HS}$ for any $t\in (0,T_m)$. 
Thus, we only have to consider the case $E_\gamma(u_0)<l_{HS}$. 

Let us define a stable set $\mathcal{M}^+$ and an unstable set $\mathcal{M}^-$ in the energy space $\dot{H}^1(\R^d)$ as
\begin{equation}\nonumber\begin{aligned}
	&\mathcal M^+ := \{ \phi \in \dot{H}^1(\mathbb R^d) \ ;\,  E_\gamma(\phi) < l_{HS}, J_\gamma(\phi) \ge0\},\\
	&\mathcal M^- := \{ \phi \in \dot{H}^1(\mathbb R^d) \ ;\,  E_\gamma(\phi) < l_{HS}, J_\gamma(\phi) <0\}, 
\end{aligned}\end{equation}
respectively. 
The following lemma means that the sets $\mathcal{M}^{\pm}$ are invariant under the semiflow associated with \eqref{crtHS} and 
coercive inequalities for $\mathcal{M}^{\pm}$.

\begin{lem}
\label{lem:invariant}
Let $u$ be a mild solution to \eqref{crtHS} with initial data $u_0\in \dot{H}^1(\R^d)$. 
Then, the following statements hold{\rm :}
\begin{itemize}
	\item[(i)] 
	If $u_0\in \mathcal{M}^{\pm}$, then $u(t)\in \mathcal{M}^{\pm}$ for any $t\in [0,T_m)$, where double-sign corresponds.

	\item[(ii)] 
	If $u_0 \in \mathcal M^{+}$, then there exists $\delta>0$ such that  
	\[
		J_\gamma(u(t)) \ge \delta \|u(t)\|_{\dot{H}^1(\mathbb R^d)}
	\]
	for any $t \in (0,T_{m})$.

	\item[(iii)] 
	If $u_0 \in \mathcal M^{-}$, then 
	\[
		J_\gamma(u(t)) < - 2^*(\gamma) \{l_{HS}- E_\gamma(u(t))\}
	\]
	for any $t \in (0,T_{m})$.
\end{itemize}
\end{lem}

The proof of this lemma is known (see, e.g., \cites{IT-arxiv,KM-2006}). 
However, to be self-contained, we give the proof. 

\begin{proof} 
First, we show the assertion (i). Let $u_0 \in \mathcal M^{+}$. Then $u(t)\in \mathcal M^{+} \cup \mathcal M^{-}$ for any $t \in [0,T_m)$ by the energy inequality \eqref{energy-ineq} in Remark \ref{rem:energy}. Suppose that there exists a positive time $t_0 \in (0,T_m)$ such that $u(t_0) \in \mathcal M^-$. Then, as $J_\gamma(u(\cdot))$ is continuous on $[0,T_m)$, there exists a time $t_1 \in [0,t_0)$ such that the identity $J_\gamma(u(t_1)) = 0$ holds.
From the definition \eqref{mini} of $l_{HS}$ and the energy inequality \eqref{energy-ineq}, we see that 
\[
	l_{HS}\le E_\gamma(u(t_1))\le E_\gamma(u_0),
\]
which contradicts the assumption $E_\gamma(u_0)<l_{HS}$. 
Thus $u(t)\in \mathcal M^{+}$ for any $t \in [0,T_m)$.
Similarly, in the case where $u_0\in \mathcal M^-$, we can prove that $u(t)\in \mathcal{M}^-$ for any $t\in [0,T_m)$, which completes the proof of (i).

Next, we show the statement (ii). 
Since $E_\gamma(u_0)<l_{HS}$, there exists $\delta_0>0$ such that the estimate
\[
     E_\gamma(u_0)\le (1-\delta_0)l_{HS}
\]
holds. We define a function $G:[0,\infty)\rightarrow\R$ given by
\[
    G(y):=\frac{1}{2}y-\frac{C_{HS}^{2^*(\gamma)}}{2^*(\gamma)}y^{\frac{2^*(\gamma)}{2}}.
\]
Then, we see that $G'(y)=0$ if and only if 
\[
	y=y_c:= C_{HS}^{-\frac{22^*(\gamma) }{2^*(\gamma)-2} },
\]
and hence, we have 
$G(y_c)= l_{HS}$ 
and $G''(y_c) < 0.$ From the Hardy-Sobolev inequality \eqref{HS-ineq} and energy inequality \eqref{energy-ineq}, we deduce that 
\begin{equation}\label{lem:vari:pr1}
	G\left(\|u(t)\|_{\dot H^1}^2 \right) 
	\le E_\gamma(u(t)) 
	\le E_\gamma(u_0) 
	\le(1-\delta_0) l_{HS} 
	= (1-\delta_0) G(y_c)
\end{equation}
for any $t\in [0,T_m)$. 
Since $\|u(t)\|_{\dot H^1}^2 < y_c$ and $F$ is non-negative and strictly monotone increasing on $(0,y_c),$ the inequality \eqref{lem:vari:pr1} implies that 
there exists some $\delta_1>0$ independent of $u(t)$ such that 
\[
	\|u(t)\|_{\dot H^1}^2 \le(1-\delta_1) y_c
\]
for any $t\in [0,T_m)$  
(In fact, it suffices to take $\delta_1 =\delta_1 (\delta_0, d, \gamma) 
		= 1 - \frac{G^{-1} ( (1-\delta_0) G(y_c))} {y_c}$).  
The convexity of $J$ implies that there exists a positive constant 
$C$ such that 
\begin{equation}\nonumber\begin{aligned}
	J_\gamma(u(t))  \ge C \min \left\{\|u(t)\|_{\dot H^1}^2, y_c-\|u(t)\|_{\dot H^1}^2\right\}
		\ge C \delta_1 \|u(t)\|_{\dot H^1}^2,
\end{aligned}\end{equation}
which completes the proof of (ii). 

Finally, we show (iii). Let $t\in [0,T_m)$ be fixed. By (i) in Lemma \ref{lem:invariant}, we have
$J_\gamma(u(t))<0$ for all $t\in [0,T_m).$
Setting 
$$K(\lambda) := E(e^{\lambda} u)
= \frac{e^{2\lambda} }{2} \|u\|_{\dot H^1}^2 
	- \frac{e^{2^*(\gamma)\lambda} }{2^*(\gamma)} 
	\|u\|_{L^{2^*(\gamma)}_\gamma}^{2^*(\gamma)},
	\quad \lambda\in\R,$$ 
we can readily check 
\begin{equation}\label{lem:vari:pr2}
	K''(\lambda) -  2^*(\gamma) K'(\lambda)
	= - (2^*(\gamma) -2)e^{2\lambda}  \|u\|_{\dot H^1}^2<0
\end{equation} 
for any $\lambda \in \R.$ 
Then, we see that $K'$ is continuous in $\lambda,$
$K'(0) = J_\gamma(u(t)) <0$ and $K'(\lambda)>0$ for $-1 <\lambda<0.$
Thus, the intermediate value theorem ensures that there exists $\lambda_0<0$ 
such that $K'(\lambda_0)=0,$ which implies that 
$e^{\lambda_0} u(t) \in \mathcal{N}_\gamma$ and $K(\lambda_0)\ge l_{HS}.$ 
Integrating \eqref{lem:vari:pr2} for the interval $(\lambda_0,0],$ we obtain 
\[
	J_\gamma(u(t)) = K'(0) < 2^*(\gamma)\{ K(0)-K(\lambda_0) \}
		\le 2^*(\gamma)\{E_\gamma(u(t))-l_{HS}\}.
\]
Thus, we conclude the lemma.
\end{proof}

Now, we state the following lemma.

\begin{lem}\label{lem:EJ}
Let $\{f_j\}_{j=1}^J \subset  \dot{H}^1(\mathbb R^d)$. 
Suppose that there exist $\varepsilon>0$ and $0<\delta<l_{HS}$ with $2 \varepsilon < \delta$ such that
\begin{equation}\label{assumE}
	E\bigg( \sum_{j=1}^J f_j\bigg) < l_{HS} - \delta,\quad E\bigg( \sum_{j=1}^J f_j\bigg) > \sum_{j=1}^J E( f_j) - \varepsilon,
\end{equation}
\begin{equation}\label{assumJ}
	J\bigg( \sum_{j=1}^J f_j\bigg) \ge -\varepsilon,\quad J\bigg( \sum_{j=1}^J f_j\bigg) \le \sum_{j=1}^J J( f_j) + \varepsilon.
\end{equation}
Then, 
\[
	0 \le E_\gamma (f_j) < l_{HS}\quad \text{and}\quad J_\gamma(f_j) \ge0
\]
for any $1 \le j \le J$.
\end{lem}

\begin{proof}
For $f \in \dot{H}^1(\mathbb R^d)$, we define 
\[
	I_\gamma(f) := E_\gamma(f) - \frac12 J_\gamma(f) =  \frac{2-\gamma}{2(d-\gamma)}\|f\|_{L^{2^*(\gamma)}_\gamma}^{2^*(\gamma)}. 
\]
The inequalities \eqref{assumE} give
\[
	E_\gamma(f_j) \le \sum_{j=1}^J E_\gamma(f_j) < E_\gamma\bigg( \sum_{j=1}^J f_j\bigg) + \varepsilon < l_{HS}- \delta + \varepsilon < l_{HS}
\]
for any $1 \le j \le J$. 
Moreover, if we can prove that 
\begin{equation}\label{J-posi}
	J_\gamma(f_j) \ge0 \quad \text{for any $1\le j \le J$},
\end{equation} 
we can obtain $E_\gamma (f_j) = I_\gamma(f_j) + \frac12 J_\gamma(f_j) \ge 0$ 
for any $1 \le j \le J$. 
We prove \eqref{J-posi} by contradiction. 
Suppose that there exists some $j_0 \in \{1,\ldots, J\}$ such that $J (f_{j_0})<0$. 
Then, there exists a real number $\lambda_0 \in (0,1)$ such that 
$J_\gamma (\lambda_0 f_{j_0})=0$. Hence,
\[
\begin{split}
	l_{HS} & \le E_\gamma(\lambda_0 f_{j_0}) 
	= I_\gamma(\lambda_0 f_{j_0}) < I_\gamma(f_{j_0}) 
	\le \sum_{j=1}^J I_\gamma(f_{j})
	= \sum_{j=1}^J E_\gamma(f_j) - \frac12 \sum_{j=1}^J J_\gamma(f_j) \\
	& < E_\gamma\bigg( \sum_{j=1}^J f_j\bigg) + \varepsilon 
	- \frac12 \bigg\{ J_\gamma\bigg( \sum_{j=1}^J f_j\bigg) - \varepsilon\bigg\}
	< l_{HS} - \delta + \varepsilon - \frac12 ( - \varepsilon - \varepsilon) 
	< l_{HS},
\end{split}
\]
which is a contradiction. Therefore, \eqref{J-posi} is proved. 
The proof of Lemma \ref{lem:EJ} is complete. 
\end{proof}

\subsection{Linear profile decomposition}\label{sub:2.6}
In this subsection, we state the following linear profile decomposition, which is a key tool to construct the minimal energy blow-up solution in the proof of (i) in Theorem~\ref{thm:GD}. 
For convenience, we use the notation $\|\cdot\|_{L^{2^*(\gamma)}_\gamma}$ for the weighted norm given by 
\[
\|f\|_{L^{2^*(\gamma)}_\gamma}
:= 
\left(\int_{\R^d} |f(x)|^{2^*(\gamma)} |x|^{-\gamma}\, dx \right)^{\frac{1}{ 2^*(\gamma)} }.
\]

\begin{prop}\label{prop:profile}
Let $\{\phi_n\}_{n=1}^\infty$ be a sequence of functions in $\dot{H}^1(\mathbb R^d)$. 
Then, after possibly passing to a subsequence {\rm (}in which case, we rename it $\phi_n${\rm )}, there exist $J^* \in \{1,2,\ldots,\infty\}$, $\{\psi^j\}_{j=1}^{J^*} \subset \dot{H}^1(\mathbb R^d)$, $\{\lambda_n^j\}_{j=1}^{J^*} \subset (0,\infty)$, and $\{x_n^j\}_{j=1}^{J^*} \subset \mathbb R^d$ such that for $1\le J \le J^*$
\begin{equation}\label{profile_0}
	\phi_n(x) = \sum_{j=1}^{J} \frac{1}{(\lambda_n^j)^{\frac{d-2}{2}}} \psi^j\left( \frac{x-x_n^j}{\lambda_n^j}\right) + w_n^J(x),
\end{equation}
where $w_n^J \in \dot{H}^1(\mathbb R^d)$ is such that 
\begin{equation}\label{profile_1}
	\limsup_{J\to J^*}
	\lim_{n\to\infty} \|e^{t\Delta}w_n^J\|_{\mathcal K^q}
	= 0,
\end{equation}
\begin{equation}\label{profile_2}
	(\lambda_n^j)^{\frac{d-2}{2}} w_n^J(\lambda_n^j x + x_n^j) \rightharpoonup 0 
	\quad \text{in } \dot{H}^1(\mathbb R^d) \text{ as }n\to\infty
\end{equation}
for any $1\le j \le J$, and 
\begin{equation}\label{profile_3}
	x_n^j \equiv 0 \quad \text{or} \quad |x_n^j|\to\infty \text{ and }\frac{|x_n^j|}{\lambda_n^j}\to\infty \text{ as $n\to\infty$}\quad \text{for }1\le j \le J^*.
\end{equation}
Moreover, the scaling and translation parameters are asymptotically orthogonal in the sense that 
\begin{equation}\label{profile_4}
	\frac{\lambda_n^j}{\lambda_n^i} + \frac{\lambda_n^i}{\lambda_n^j} + \frac{|x_n^i - x_n^j|^2}{\lambda_n^j\lambda_n^i}
	\to + \infty
\end{equation}
as $n\to \infty$ for any $i\not = j$. Furthermore, for any $1\le J \le J^*$, we have the following decoupling properties{\rm :} 
\begin{equation}\label{profile_5}
	\lim_{n\to\infty}\Big|\|\phi_n\|_{\dot{H}^1}^2 - \sum_{j=1}^J \|\psi^j\|_{\dot{H}^1}^2 - \|w_n^J\|_{\dot{H}^1}^2\Big|=0,
\end{equation}
\begin{equation}\label{profile_6}
	\lim_{n\to\infty}\Big|
	\|\phi_n\|_{L^{2^*(\gamma)}_\gamma}^{2^*(\gamma)} - \sum_{j=1}^J \|\psi^j\|_{L^{2^*(\gamma)}_\gamma}^{2^*(\gamma)} 
	- \|w_n^J\|_{L^{2^*(\gamma)}_\gamma}^{2^*(\gamma)}\Big| =0.
\end{equation}
Especially, 
\begin{equation}\label{decop-E-phi_n}
	\lim_{n\to\infty}\Big|
	E_\gamma(\phi_n) - \sum_{j=1}^J E_\gamma(\psi^j) - E_\gamma(w_n^J)
	\Big| = 0,
\end{equation}
\begin{equation}\label{decop-J-phi_n}
	\lim_{n\to\infty}\Big|
	J_\gamma(\phi_n) - \sum_{j=1}^J J_\gamma(\psi^j) - J_\gamma(w_n^J) 
	\Big| = 0
\end{equation}
for any $1\le J \le J^*$.
\end{prop}

The profile decomposition for $\dot{H}^1(\R^d)$ is known (see, e.g., Theorem~4.3 in \cite{KM-2006} and Theorem 4.7 in \cite{KV-2013}). More precisely, 
it is known that 
for a sequence $\{\phi_n\}_{n=1}^\infty \subset \dot{H}^1(\mathbb R^d)$, 
after possibly passing to a subsequence, there exist $J^* \in \{1,2,\ldots,\infty\}$, $\{\psi^j\}_{j=1}^{J^*} \subset \dot{H}^1(\mathbb R^d)$, $\{\lambda_n^j\}_{j=1}^{J^*} \subset (0,\infty)$, and $\{x_n^j\}_{j=1}^{J^*} \subset \mathbb R^d$ such that
\eqref{profile_0}, \eqref{profile_2}, \eqref{profile_4}, \eqref{profile_5}, and the following assertions hold:
\begin{equation}\label{profile_1'}
	\lim_{J\to J^*}\lim_{n\to\infty} \|w_n^J\|_{L^{2^*}} = 0,
\end{equation}
\[
	\lim_{n\to\infty}\Big|
	\|\phi_n\|_{L^{2^*}}^{2^*} - \sum_{j=1}^J \|\psi^j\|_{L^{2^*}}^{2^*} - \|w_n^J\|_{L^{2^*}}^{2^*} \Big| = 0.
\]
The property \eqref{profile_1} follows from \eqref{profile_1'} and the inequality 
\[
	\|e^{t\Delta}w_n^J\|_{\mathcal K^q}\le C\|w_n^J\|_{L^{2^*}}.
\]
Moreover, \eqref{decop-E-phi_n} and \eqref{decop-J-phi_n} are immediate consequences of \eqref{profile_6}, 
and finally, we can reset the profiles $\psi^j$ and the remainder terms $w^J_n$ such that parameters $x_n^j$ and $\lambda_n^j$ satisfy \eqref{profile_3} 
(see, e.g., Proposition 3.2 in \cite{MiaMurZhe2020}).
Hence, we only have to show \eqref{profile_6}. For this purpose, 
we use the Brezis-Lieb lemma. 

\begin{lem}[Theorem 1 in \cite{BreLie1983}]\label{lem:WBL}
Let $1\le q<\infty$ and $\mu$ be a measure on $\mathbb R^d$, and let $\{f_n\}_n$ be a bounded sequence in $L^q(\mathbb R^d, d\mu)$ such that $f_n \to f$ almost everywhere in $\mathbb R^d$. Then, $f\in L^q(\mathbb R^d,d\mu)$ and 
\begin{equation}\label{decomposition}
	\lim_{n\to \infty}
	\big\{ \|f_n\|_{L^q(d\mu)}^q - \|f- f_n\|_{L^q(d\mu)}^q
	\big\} = \|f\|_{L^q(d\mu)}^q.
\end{equation}
Here, $L^q(\mathbb R^d,d\mu)$ is defined by 
\[
	L^q(\mathbb R^d,d\mu) := \left\{ f \,;\, \|f\|_{L^q(d\mu)} 
	=\left(\int_{\mathbb R^d} |f(x)|^q \, d\mu(x)\right)^{\frac{1}{q}}<\infty \right\}.
\]
\end{lem}

\begin{proof}[Proof of \eqref{profile_6}]
We shall prove \eqref{profile_6} by induction. 
The decomposition \eqref{profile_0} with $J=1$ 
is written as 
\[
	(\lambda_n^1)^{\frac{d-2}{2}} \phi_n(\lambda_n^1x+x_n^1) = \psi^1(x) + (\lambda_n^1)^{\frac{d-2}{2}} w_n^1(\lambda_n^1x+x_n^1).
\]
Then, we apply Lemma \ref{lem:WBL} with a Radon measure $\mu$ defined by 
\[
	\mu(A) := \int_A |x|^{-\gamma}\, dx,\quad A\subset \R^d,
\]
and with $f = \psi^1$ and $f_n =(\lambda_n^1)^{\frac{d-2}{2}} \phi_n(\lambda_n^1x+x_n^1)$ 
to obtain
\[
	\lim_{n\to \infty}
	\big\{ \|(\lambda_n^1)^{\frac{d-2}{2}} \phi_n(\lambda_n^1x+x_n^1)\|_{L^{2^*(\gamma)}_\gamma}^{2^*(\gamma)} 
	- 
	\|(\lambda_n^1)^{\frac{d-2}{2}} w_n^1(\lambda_n^1x+x_n^1)\|_{L^{2^*(\gamma)}_\gamma}^{2^*(\gamma)} 
	\big\} = \|\psi^1\|_{L^{2^*(\gamma)}_\gamma}^{2^*(\gamma)}.
\]
This is equivalent to \eqref{profile_6} with $J=1$. Thus, the case $J=1$ is proved. 

Next, we suppose that \eqref{profile_6} holds up to $J$ ($\ge2$). 
By two decompositions \eqref{profile_0} with $J$ and $J+1$, we have
\[
	w_n^J(x) 
	= \frac{1}{(\lambda_n^{J+1})^{\frac{d-2}{2}}} \psi^{J+1}\left( \frac{x-x_n^{J+1}}{\lambda_n^{J+1}}\right) + w_n^{J+1}(x),
\]
which is written as 
\[
	(\lambda_n^{J+1})^{\frac{d-2}{2}} w_n^J (\lambda_n^{J+1}x+x_n^{J+1}) 
	= \psi^{J+1}(x) + (\lambda_n^{J+1})^{\frac{d-2}{2}} w_n^{J+1}(\lambda_n^{J+1}x+x_n^{J+1}). 
\]
Then, again applying Lemma \ref{lem:WBL} with $f=\psi^{J+1}$ and $f_n = (\lambda_n^{J+1})^{\frac{d-2}{2}} w_n^J (\lambda_n^{J+1}x+x_n^{J+1})$, we obtain
\[
\begin{split}
	\lim_{n\to \infty}
	\big\{ \| (\lambda_n^{J+1})^{\frac{d-2}{2}}  w_n^J (\lambda_n^{J+1}x&+x_n^{J+1}) \|_{L^{2^*(\gamma)}_\gamma}^{2^*(\gamma)} \\
	&- 
	\|w_n^{J+1}(\lambda_n^{J+1}x+x_n^{J+1})\|_{L^{2^*(\gamma)}_\gamma}^{2^*(\gamma)} 
	\big\} = \|\psi^{J+1}\|_{L^{2^*(\gamma)}_\gamma}^{2^*(\gamma)}.
\end{split}
\]
By combining this convergence and \eqref{profile_6} with $J$, we prove \eqref{profile_6} with $J+1$. 
By induction, we conclude \eqref{profile_6} for any $1\le J \le J^*$.
This completes the proof of Proposition~\ref{prop:profile}. 
\end{proof}

\section{Proof of Theorem \ref{thm:GD}}\label{sec:3}
In this section, we prove Theorem \ref{thm:GD}. 
We only have to consider the case $E_\gamma(u_0)<l_{HS}$, 
because the other case where $E_\gamma(u_0)=l_{HS}$ and $J_\gamma(u_0)\ne 0$ is reduced to 
this case 
by \eqref{energy-id} and \eqref{energy-ineq}. 
In addition, 
by Proposition \ref{prop:energy-id} and Lemma \ref{lem:invariant}, we may assume that the solution $u=u(t)$ to \eqref{crtHS} with $u(0)=u_0$ satisfies the energy identity
\begin{equation}\label{energy-id2}
	E_\gamma(u(t))+ \int_{0}^t\int_{\R^d}|\partial_t u(\tau,x)|^2\, dx d\tau = E_\gamma(u_0)
\end{equation}
for any $t \in [0,T_m)$ without loss of generality.

\subsection{Proof of the dissipation part (i)}\label{sub:3.1}
In this subsection, we give a proof of (i) in Theorem~\ref{thm:GD}. 
Let us introduce a subset $\mathcal{M}_E^+\subset \dot{H}^1(\mathbb R^d)$ defined by 
\[
	\mathcal M_E^+ := \left\{\phi \in \dot{H}^1(\mathbb R^d) \,;\, 
	E_\gamma(\phi) < E, J_\gamma(\phi)\ge0\right\}, \quad E\in \R, 
\]
and a critical energy $E^c$ given by 
\begin{equation}\label{def:Ec}
\begin{split}
E^c 
:= \sup\big\{ E \in \R \,;\, &\ \text{$T_m(u_0)=+\infty$ and $\|u\|_{\mathcal K^q}<\infty$}\\
& \text{ for any solution $u$ to \eqref{crtHS} with $u_0\in \mathcal M_E^+$}
\big\}.
\end{split}
\end{equation}
Note that all solutions to \eqref{crtHS} with initial data in $\mathcal M_E^+$ are dissipative if $E < E^c$ by (vi) in Proposition \ref{prop:wellposed1}.
It follows that (i) in Theorem \ref{thm:GD} is equivalent to $E^c \ge l_{HS}$. 
Hence, it suffices to prove that $E^c \ge l_{HS}$ by contradiction. 
To this end, we suppose that 
\begin{equation}\label{assum}
	E^c < l_{HS},
\end{equation}
and then aim at deducing $E^c=0$. 
This is a contradiction, as $E^c>0$ by the small-data global existence.

Let us concentrate on proving $E^c=0$ under the assumption \eqref{assum}. 
We take a sequence $\{\phi_n\}_{n=1}^\infty \subset \mathcal M^+$ to attain 
$E^c$ from above, such that 
\begin{equation}\label{cri-ele}
	\text{$E_\gamma(\phi_n) \searrow E^c$ as $n\to \infty$}
	\quad \text{and}\quad 
	\text{$\|u_n\|_{\mathcal K^q(T_{m}(\phi_n))}=\infty$ for $n=1,2,\cdots$},
\end{equation}
where $u_n$ is a solution to \eqref{crtHS} with $u_n(0)=\phi_n$. 
The following is the key decomposition of $\{\phi_n\}_{n=1}^\infty$ with 
a single profile $\psi$, based on the linear profile decomposition 
(see Proposition \ref{prop:profile}).

\begin{lem}\label{lem:singleprofile}
Suppose \eqref{assum}. Let $\{\phi_n\}_{n=1}^\infty$ be the above sequence. Then, 
\begin{equation}\label{decom:J=1}
\phi_n = \psi_n  + w_n,\quad \psi_n (x):=\frac{1}{(\lambda_n)^{\frac{d-2}{2}}}\psi\left( \frac{x}{\lambda_n}\right)
\end{equation}
with scale parameters $\{\lambda_n\}_{n=1}^\infty \subset (0,\infty)$, where $\psi_n, w_n \in \mathcal M^+$ and 
\begin{equation}\label{w_n}
\lim_{n\to \infty} E_\gamma(w_n) = \lim_{n\to \infty} \|w_n\|_{\dot{H}^1} = 0.
\end{equation}
\end{lem}

Supposing this lemma holds, we now complete the proof of $E^c=0$ 
by using this lemma. 
Let $v^c=v^c(t,x)$ be a solution to \eqref{crtHS} with $v^c(0) = \psi$, where $\psi$ is the profile in the decomposition \eqref{decom:J=1}.
Then, it follows from \eqref{w_n} and the energy inequality \eqref{energy-ineq} that 
\[
E^c = \lim_{n\to \infty}E_\gamma(\phi_n) = \lim_{n\to \infty} E_\gamma(\psi_n) 
	= E_\gamma(\psi) \ge E_\gamma(v^c(t)), 
	\quad t \in [0,T_{m}(\psi)).
\]
On the other hand, we see that there exists $t_0 \in (0, T_m(\psi))$ such that 
\begin{equation}\label{E^c-ineq}
	E^c \le E_\gamma(v^c(t_0))
\end{equation}
by contradiction.
In fact, we suppose that $E^c > E_\gamma(v^c(t))$ for any $t \in (0,T_{m}(\psi))$. 
Then, $T_m(\psi)=+\infty$ and $\|v^c\|_{\mathcal K^q}<\infty$ by the definition of $E^c$. 
We denote by $v_n^c$ the scaled function of $v^c$ such that 
\[
v_n^c (t,x) := \frac{1}{(\lambda_n)^{\frac{d-2}{2}}} v^c\left(\frac{t}{(\lambda_n)^{2}}, \frac{x}{\lambda_n}\right).
\]
Then, $v_n^c$ is a solution to \eqref{crtHS} with $v_n^c(0) = \psi_n$ and satisfies $\|v^c_n\|_{\mathcal K^q}<\infty$ for any $n\in\mathbb N$. 
Combining the perturbation result with \eqref{w_n}, we also have $T_{m}(\phi_n)=+\infty$ and $\|u_n\|_{\mathcal K^q}<\infty$ 
for a sufficiently large $n$. This contradicts \eqref{cri-ele}. 
Thus, there exists $t_0 \in (0, T_m(\psi))$ such that \eqref{E^c-ineq} holds.
Summarizing what has been obtained so far, we find that 
\[
	E^c = E_\gamma(v^c(t_0)).
\] 
This means that $v^c$ is a stationary solution by the energy identity. 
However, by \eqref{assum}, $v^c$ must be the zero solution. This proves $E^c = 0$.
Therefore, by contradiction, \eqref{assum} is negated. 
Thus we conclude that $E^c \ge l_{HS}$.\\

The rest of this subsection is devoted to the proof of Lemma \ref{lem:singleprofile}. 
For this purpose, we prepare the following two lemmas.
\begin{lem}\label{lem:key1}
Let $d\ge3$ and $0<\gamma<2$, and let $v_n$ be a solution to the linear heat equation with initial data 
\[
v_{n}(0) = \lambda_n^{-\frac{d-2}{2}} u_0\left(\frac{x-x_n}{\lambda_n} \right),\quad u_0 \in \dot H^1(\mathbb R^d),
\]
where $\lambda_n \in (0,\infty)$ and $x_n \in \mathbb R^d$. 
Assume that $|x_n| \to + \infty$ and $|x_n/\lambda_n|\to + \infty$ as $n\to\infty$. 
Then, 
\begin{equation}\label{error-linear}
\lim_{n\to \infty}\left\|\int_0^t e^{(t-\tau)\Delta}\{|x|^{-\gamma} |v_n(\tau)|^{2^*(\gamma)-2}v_n(\tau)\}\,d\tau\right\|_{\mathcal K^q} = 0
\end{equation}
for $q$ satisfying the assumptions of Proposition \ref{prop:wellposed1}. 
\end{lem}

\begin{rem}
Lemma \ref{lem:key1} and Proposition \ref{prop:perturbation} imply that the solution to \eqref{crtHS} with initial data $v_{n}(0)$ 
converges to the linear solution $v_n$ as $t\to \infty$ (i.e., it is dissipative), if $n$ is sufficiently large, although $\|v_{n}(0)\|_{\dot H^1} = \|u_0\|_{\dot H^1}$ for all $n\in\mathbb N$.
This phenomenon is unique to equations such as \eqref{crtHS} and is caused by $|x|^{-\gamma}$. 
In fact, this phenomenon does not occur in case $\gamma=0$, 
as \eqref{crtHS} with $\gamma=0$ is invariant under the translation with respect to $x$.
\end{rem}

\begin{lem}\label{lem:key2}
Let $d\ge3$ and $0<\gamma<2$, and let $u^j$ be a solution to the linear heat equation or \eqref{crtHS} with $u^j(0)=u^j_0 \in \dot H^1(\R^d)$ and $\|u^j\|_{\mathcal K^q} < \infty$ for $j=1,2$.
For $j=1,2$, define 
\begin{equation}\label{sc-tr}
u_n^j(t,x) 
:=
\frac{1}{(\lambda_n^j)^{\frac{d-2}{2}}} 
u^j\left( \frac{t}{(\lambda_n^j)^2}, \frac{x-x_n^j}{\lambda_n^j}\right)
\end{equation}
with parameters $\{\lambda_n^j\}_{n=1}^{\infty} \subset (0,\infty)$ and $\{x_n^j\}_{n=1}^{\infty} \subset \mathbb R^d$.
Assume that 
\begin{equation}\label{aym-orth}
\frac{\lambda_n^1}{\lambda_n^2} + \frac{\lambda_n^2}{\lambda_n^1} + \frac{|x_n^2 - x_n^1|^2}{\lambda_n^1\lambda_n^2}
\to + \infty
\end{equation}
as $n\to \infty$. 
Then,
\begin{equation}\label{eq.key}
\lim_{n\to \infty} \sup_{t\in (0,\infty)}t^{\frac{d(2^*(\gamma)-1)}{2}(\frac{1}{q_c}-\frac{1}{q})}
\big\||u_n^1(t)|^{2^*(\gamma)-2}|u_n^2(t)|\big\|_{L^{\frac{q}{2^*(\gamma)-1}}} = 0
\end{equation}
for $q$ satisfying the assumptions of Proposition \ref{prop:wellposed1}. 
\end{lem}

Let us now prove Lemma \ref{lem:singleprofile} by using Proposition \ref{prop:profile} and these lemmas.

\begin{proof}[Proof of Lemma \ref{lem:singleprofile}]
Let $\{\phi_n\}_{n=1}^\infty \subset \mathcal M^+$ be a sequence 
for attaining $E^c$ from above, such that \eqref{cri-ele} holds.
By Proposition \ref{prop:profile}, there exist $J^* \in \{1,2,\ldots,\infty\}$, 
$\{\psi^j\}_{j=1}^{J^*} \subset \dot{H}^1(\mathbb R^d)$, 
$\{\lambda_n^j\}_{j=1}^{J^*} \subset (0,\infty)$, and 
$\{x_n^j\}_{j=1}^{J^*} \subset \mathbb R^d$ such that for $1\le J\le J^*$
\begin{equation}\label{profile1}
	\phi_n(x) = \sum_{j=1}^{J} \psi_n^j(x) + w_n^J(x)
\end{equation}
and \eqref{profile_1}--\eqref{profile_6} hold, where $\psi_n^j$ is defined by 
\begin{equation}\label{psi^j_n}
\psi_n^j(x)
:=
\frac{1}{(\lambda_n^j)^{\frac{d-2}{2}}} \psi^j\left( \frac{x-x_n^j}{\lambda_n^j}\right).
\end{equation}
We discuss only the case $J^*=\infty$, as the case $J^*<\infty$ is similar.
Note that, from \eqref{profile_5}, 
\begin{equation}\label{series-psi^j}
	\sum_{j=1}^\infty\| \psi^{j} \|_{\dot{H}^1} 
	\le  \|\phi_n\|_{\dot{H}^1} + o_n(1) 
	\le C(1 + E^c)
\end{equation}
for a sufficiently large $n$. Moreover, owing to Lemma \ref{lem:EJ}, it follows from 
\eqref{decop-E-phi_n}, \eqref{decop-J-phi_n}, and \eqref{assum} that 
\begin{equation}\label{psi-w-M}
	\psi^j, w_n^J \in \mathcal M^+
	\quad \text{for any $1\le j\le J$ and $n\in\mathbb N$}.
\end{equation}
Taking \eqref{profile_3} into account, we define
\[
\mathcal J_1 :=\left\{ j \in\mathbb N \ ;\,  x_n^j = 0 \text{ for any $n\in\mathbb N$}\right\},
\]
\[
\mathcal J_2 := \left\{ j \in\mathbb N \ ;\,  |x_n^j|\to\infty \text{ and }\frac{|x_n^j|}{\lambda_n^j} \to +\infty \text{ as }n\to \infty\right\}.
\]
We consider the case where $\mathcal J_1$ and $\mathcal J_2$ are non empty. In the last of this proof, we will explain the other cases. 
Let us define an approximate solution $u_n^J$ of $u_n$ by 
\begin{equation}\label{u_n^J}
u_n^J(t) := \sum_{j=1}^{J} v^j_n(t) + e^{t\Delta}w_n^J,
\end{equation}
where 
$v^j_n$ is a solution to \eqref{crtHS} with initial data $v^j_n(0)=\psi^j_n$ if $j \in \mathcal J_1$, and 
to the linear equation $\partial_t v_n^j - \Delta v_n^j =0$ with $v_n^j(0) = \psi^j_n$ if $j \in \mathcal J_2$.
As to $\mathcal J_2$, we see that  
\[
\sum_{j\in \mathcal J_2}\|v_n^j\|_{\mathcal K^q}
\le 2
C \sum_{j=1}^\infty\| \psi^{j} \|_{\dot{H}^1} 
<\infty.
\]
As to $\mathcal J_1$, we write $v_n^j$ as
\[
v_n^j (t,x) = 
\frac{1}{(\lambda_n^j)^{\frac{d-2}{2}}} v^j\left(\frac{t}{(\lambda_n^j)^{2}}, \frac{x}{\lambda_n^j}\right),\quad j\in\mathcal J_1,
\]
where $v^j$ is a solution to \eqref{crtHS} with initial data $v^j(0)=\psi^j$. 
Since 
$\| \psi^{j} \|_{\dot{H}^1} \to 0$ as $j\to \infty$ by \eqref{series-psi^j}, 
there exists $J' \in \mathbb N$ such that 
\begin{equation}\label{sdge}
\| v^j \|_{\mathcal K^q} \le 2\| \psi^{j} \|_{\dot{H}^1} \quad \text{for any $j \ge J'$.}
\end{equation}
For a contradiction, we assume that 
\begin{equation}\label{contradiction}
\|v^j\|_{\mathcal K^q(T_{m}(\psi^j))} < \infty\quad \text{for any $j\in\mathcal J_1$ with $1 \le j \le J'-1$}.
\end{equation}
Then, $T_{m}(\psi^j)= + \infty$ and 
\begin{equation}\label{series}
\sum_{j\in \mathcal J_1}\|v_n^j\|_{\mathcal K^q} =\sum_{j\in \mathcal J_1} \|v^j\|_{\mathcal K^q} < \infty.
\end{equation}
Hence, we have 
\[
\lim_{J\to\infty}\lim_{n\to\infty}\|u_n^J\|_{\mathcal K^q} <\infty.
\]
Here, we note from (i) in Proposition \ref{prop:wellposed1} that 
\[
\|v_n^j\|_{\mathcal K^{\tilde{q}}_{\tilde{r}} } = \|v^j\|_{\mathcal K^{\tilde{q}}_{\tilde{r}}} < \infty
\quad \text{and}\quad 
\|u_n^J\|_{\mathcal K^{\tilde{q}}_{\tilde{r}} } < \infty
\]
for any $n \in \mathbb N$ and for any $(\tilde{q},\tilde{r})$ satisfying the conditions in Proposition \ref{prop:perturbation}.
For convenience, we use the notations $\tilde{\mathcal J}_1 = \mathcal J_1 \cap \{1,2,\ldots, J\}$ and $\tilde{\mathcal J}_2 = \mathcal J_2 \cap \{1,2,\ldots, J\}$.
Now, $u_n^J$ is a solution to the approximate equation
\[
\begin{cases}
\partial_t u_n^J - \Delta u_n^J = |x|^{-\gamma}|u^J_n|^{2^*(\gamma)-2}u^J_n + e_n^J,\\
u_n^J(0) = \phi_n,
\end{cases}
\]
where 
\[
e_n^J :=
|x|^{-\gamma} \left\{
\sum_{j\in \tilde{\mathcal J}_1} |v^j_n|^{2^*(\gamma)-2}v^j_n - |u^J_n|^{2^*(\gamma)-2}u^J_n \right\}.
\]
To use Proposition \ref{prop:perturbation}, 
we will prove that 
\begin{equation}\label{error-e_n^J}
\lim_{J\to\infty}\limsup_{n\to \infty}\left\|\int_0^t e^{(t-\tau)\Delta}(e_n^J(\tau))\,d\tau\right\|_{\mathcal K^q} = 0.
\end{equation}
We write 
\[
e_n^J = e_{n,1}^J  + e_{n,2}^J  + e_{n,3}^J,
\]
where $e_{n,1}^J$, $e_{n,2}^J$, and $e_{n,3}^J$ are given by 
\[
e_{n,1}^J := |x|^{-\gamma} \left\{\sum_{j=1}^J |v^j_n|^{2^*(\gamma)-2}v^j_n -  \bigg|\sum_{j=1}^{J} v^j_n\bigg|^{2^*(\gamma)-2}\bigg( \sum_{j=1}^{J}v^j_n\bigg)\right\},
\]
\[
e_{n,2}^J := |x|^{-\gamma} \left\{|u^J_n - e^{t\Delta}w_n^J|^{2^*(\gamma)-2}( u^J_n - e^{t\Delta}w_n^J) - |u^J_n|^{2^*(\gamma)-2}u^J_n\right\},
\]
\[
e_{n,3}^J := - |x|^{-\gamma} \sum_{j\in \tilde{\mathcal J}_2} |v^j_n|^{2^*(\gamma)-2}v^j_n.
\]
As to the term $e_{n,1}^{J}$, 
as
\[
|e_{n,1}^{J}|
\le C
|x|^{-\gamma} \sum_{\substack{ 1\le i,j \le J \\ i\not = j}}|v_n^i|^{2^*(\gamma)-2}|v_n^j|,
\]
it follows from Lemma \ref{lem:key2} that 
\[
\lim_{n\to \infty} \sup_{t\in (0,\infty)}t^{\frac{d(2^*(\gamma)-1)}{2}(\frac{1}{q_c}-\frac{1}{q})}
\||v_n^i(t)|^{2^*(\gamma)-2}|v_n^j(t)|\|_{L^{\frac{q}{2^*(\gamma)-1}}} = 0
\]
for any $1\le i,j \le J$ with $i\ne j$.
Hence, we obtain
\begin{equation}\label{error-e_n^J_1}
\limsup_{n\to \infty}\left\|\int_0^t e^{(t-\tau)\Delta}(e_{n,1}^J(\tau))\,d\tau\right\|_{\mathcal K^q} = 0.
\end{equation}
for each $J\ge1$. As to the term $e_{n,2}^{J}$, we estimate
\[
\begin{split}
&\left\|\int_0^t e^{(t-\tau)\Delta}(e_{n,2}^J(\tau))\,d\tau\right\|_{\mathcal K^q}\\
& \le C \big(\|u_n^J\|_{\mathcal K^q}^{2^*(\gamma)-2}\|e^{t\Delta}w_n^J\|_{\mathcal K^q}
+ \|e^{t\Delta}w_n^J\|_{\mathcal K^q}^{2^*(\gamma)-2}
\|u_n^J\|_{\mathcal K^q} + \|e^{t\Delta}w_n^J\|_{\mathcal K^q}^{2^*(\gamma)-1}\big),
\end{split}
\]
where we note that the above constant $C$ is independent of $J$.
Hence, by \eqref{profile_1}, 
we also obtain
\begin{equation}\label{error-e_n^J_2}
\lim_{J\to\infty}\limsup_{n\to \infty}\left\|\int_0^t e^{(t-\tau)\Delta}(e_{n,2}^J(\tau))\,d\tau\right\|_{\mathcal K^q} = 0.
\end{equation}
As to the term $e_{n,3}^{J}$, it follows from Lemma \ref{lem:key1} that 
\begin{equation}\label{error-e_n^J_3}
\limsup_{n\to \infty}\left\|\int_0^t e^{(t-\tau)\Delta}(e_{n,3}^J(\tau))\,d\tau\right\|_{\mathcal K^q} = 0
\end{equation}
for each $J\ge1$. 
Summarizing \eqref{error-e_n^J_1}--\eqref{error-e_n^J_3}, we obtain \eqref{error-e_n^J}. 
Hence, we can apply Proposition \ref{prop:perturbation} to $u_n$ and $u_n^J$, and then we obtain $\|u_n\|_{\mathcal K^q}<\infty$ for a sufficiently large $n$. 
This contradicts \eqref{cri-ele}. Therefore, \eqref{contradiction} is negated and there exists $j_0 \in \mathcal J_1$ such that 
\begin{equation}\label{step2}
\|v^j\|_{\mathcal K^q(T_{m}(\psi^{j_0}))} = \infty.
\end{equation}

Since $E(\psi^j)\ge 0$ and $E(w_n^J)\ge0$ by $\psi^j, w_n^J\in \mathcal M^+$, 
we see that  
\begin{equation}\label{decop-E-phi_n2}
E^c = \lim_{n\to\infty} E(\phi_n) = \sum_{j=1}^J E(\psi^j) + \lim_{n\to\infty} E(w_n^J)
\end{equation}
for any $J\ge1$. Then, $E^c \ge E(\psi^j)$ for any $j \ge1$. 
On the other hand, by \eqref{step2}, we have $E^c \le E(\psi^{j_0})$. 
Hence, 
\[
E^c = E(\psi^{j_0}).
\]
By \eqref{decop-E-phi_n2}, we have $E(\psi^j) = 0$ for any $j \not = j_0$ and  
\[
\lim_{n\to \infty} E(w_n^J) = 0\quad \text{for any $J\ge1$}.
\]
Moreover, as  
\[
\|\psi^j\|_{\dot{H}^1(\mathbb R^d)} \le C E(\psi^j),\quad \|w_n^J\|_{\dot{H}^1(\mathbb R^d)} \le C E(w_n^J)
\]
by $\psi, w_n^J\in \mathcal M^+$, we obtain 
$\psi^j = 0$ for any $j \not = j_0$ and 
\[
\displaystyle \lim_{n\to \infty} \|w_n^J\|_{\dot{H}^1(\mathbb R^d)} = 0\quad \text{for any $J\ge1$}.
\]
Therefore, Lemma \ref{lem:singleprofile} is proved 
in the case where $\mathcal J_1$ and $\mathcal J_2$ are non-empty. 

Finally, we consider the remaining cases. 
In the case where $\mathcal J_2$ is empty, we can perform the same argument as above and prove Lemma \ref{lem:singleprofile}. 
In contrast, the case where $\mathcal J_1$ is empty does not occur. In fact, 
in this case, the error term $e_n^J$ is written as 
\[
e_n^J = 
- |x|^{-\gamma} |u^J_n|^{2^*(\gamma)-2}u^J_n = 
- |x|^{-\gamma}
\left\{
\sum_{j=1}^{J} v^j_n(t) + e^{t\Delta}w_n^J\right\}.
\]
By Lemma \ref{lem:key1} and \eqref{profile_1}, we obtain \eqref{error-e_n^J}. 
Hence we apply Proposition \ref{prop:perturbation} to obtain $\|u_n\|_{\mathcal K^q}<\infty$ for a sufficiently large $n$,
which is a contradiction. Therefore, this case does not occur. Thus, we conclude Lemma \ref{lem:singleprofile}.
\end{proof}

It remains to prove Lemmas \ref{lem:key1} and \ref{lem:key2}. 

\begin{proof}[Proof of Lemma \ref{lem:key1}]
To prove Lemma \ref{lem:key1}, it suffices to show that 
\begin{equation}\label{error'}
\lim_{n\to \infty}\left\|\int_0^t e^{(t-\tau)\Delta}(e_n(\tau))\,d\tau\right\|_{\mathcal K^q(T)} = 0
\end{equation}
for each $T>0$, where 
\[
e_n(\tau) := |x|^{-\gamma} |v_n(\tau)|^{2^*(\gamma)-2}v_n(\tau).
\]
In fact, the function 
\begin{equation}\label{function-t}
t^{\frac{d}{2}(\frac{1}{q_c}-\frac{1}{q})}\left\|\int_0^t e^{(t-\tau)\Delta}(e_n(\tau))\,d\tau\right\|_{L^q}
\end{equation}
is continuous and bounded in $t>0$ and converges to $0$ as $t\to\infty$.
Hence, there exists a time $T_n \ge0$ such that
it attains the maximum of \eqref{function-t} over $t\in [0,\infty)$.
If the limit superior of $T_n$ diverges as $n\to\infty$, then \eqref{error-linear} holds, as the function \eqref{function-t} 
converges to $0$ as $t\to\infty$. Hence we only consider the case where $\sup_{n} T_n <\infty$. 
Therefore we only have to prove \eqref{error'} for each $T>0$. 

Next, we will construct approximation sequences $\{\tilde{v}_{n,k}\}_{k=1}^\infty$ of $v_n$ with compact support in $\mathbb R^d$ and 
\begin{equation}\label{appro-3.1}
\sup_n \left\|\int_0^t e^{(t-\tau)\Delta}(e_n(\tau) - \tilde{e}_{n,k}(\tau))\,d\tau\right\|_{\mathcal K^q(T)} 
\to 0\quad \text{as }k\to\infty,
\end{equation}
where 
\begin{equation}\label{error-tilde}
\tilde{e}_{n,k} := - |x|^{-\gamma} |\tilde{v}_{n,k}|^{2^*(\gamma)-2}\tilde{v}_{n,k}.
\end{equation}
For $u_0 \in \dot H^1(\R^d)$, we define $v(t) := e^{t\Delta} u_0$. 
Let us define $\{\tilde{v}_{k}\}_{k=1}^\infty$ by 
\begin{equation}\label{def:v_k}
\tilde{v}_{k}(t,x) := \chi_k(x)v(t,x),
\end{equation}
where $\chi_k \in C^\infty_0(\mathbb R^d)$ is such that $\chi_k(x) \to 1$ as  $k\to\infty$ for each $x\in\mathbb R^d$. 
By Lebesgue's dominated convergence theorem, we have
\begin{equation}\label{pointwise}
\|v (t)-\tilde{v}_k (t)\|_{L^q} \to 0\quad \text{as }k\to\infty \text{ for each }t>0.
\end{equation}
Similarly to \eqref{function-t}, 
the function 
\begin{equation}\label{function-t2}
t^{\frac{d}{2}(\frac{1}{q_c}-\frac{1}{q})}\|v(t) - \tilde{v}_{k}(t)\|_{L^q}
\end{equation}
is continuous and bounded in $t>0$ and converges to $0$ as $t\to\infty$. 
Then, after possibly passing to a subsequence, there are $\tilde{T}_k$ and $\tilde{T}$ such that 
$\tilde{T}_k\to\tilde{T}$ as $k\to\infty$ and 
$\tilde{T}_k$ attains the maximum of \eqref{function-t2} over $t\in [0,\infty)$. 
If $\tilde{T}=+\infty$, then 
\begin{equation}\label{sufficient1}
\|v - \tilde{v}_{k}\|_{\mathcal K^q(\infty)}\to 0\quad \text{as $k\to \infty$}. 
\end{equation}
On the other hand, if $\tilde{T}<+\infty$, 
then 
\[
\begin{split}
\|v - \tilde{v}_{k}\|_{\mathcal K^q} 
& \le 
(\tilde{T} + 1)^{\frac{d}{2}(\frac{1}{q_c}-\frac{1}{q})} \|v(\tilde{T}_{k}) - \tilde{v}_k(\tilde{T}_{k})\|_{L^q}\\
& \le C
\big(
\|v(\tilde{T}_{k}) - v(\tilde{T})\|_{L^q}
+
\|v(\tilde{T}) - \tilde{v}_{k}(\tilde{T})\|_{L^q}
+
\|\tilde{v}_{k}(\tilde{T}) - \tilde{v}_{k}(\tilde{T}_{k})\|_{L^q}
\big)\\
& \le C
\big(
\|v(\tilde{T}_{k}) - v(\tilde{T})\|_{L^q}
+
\|v(\tilde{T}) - \tilde{v}_{k}(\tilde{T})\|_{L^q}
\big)
\end{split}
\]
for a sufficiently large $k$, where the definition of $\tilde{v}_{k}$ is used in the last step. 
The first and second terms on the right-hand side converge to $0$ as $k\to\infty$ by uniform continuity of $v$ in $t\in [0,T]$ and the pointwise convergence \eqref{pointwise}, respectively. 
Hence, we also obtain \eqref{sufficient1} in this case. 
We define the sequence $\{\tilde{v}_{n,k}\}_{k=1}^\infty$ by 
\begin{equation}\label{def:tilde-v}
\tilde{v}_{n,k}(t,x) :=\lambda_n^{-\frac{d-2}{2}} \tilde{v}_k\left( \frac{t}{\lambda_n^2}, \frac{x-x_n}{\lambda_n} \right),
\end{equation}
and the error term $\tilde{e}_{n,k}$ by \eqref{error-tilde}. 
Since 
\[
v_n(t,x) = \lambda_n^{-\frac{d-2}{2}} v\left( \frac{t}{\lambda_n^2}, \frac{x-x_n}{\lambda_n} \right)
=
\lambda_n^{-\frac{d-2}{2}}(e^{\frac{t}{\lambda_n^2}\Delta} u_0)\left(\frac{x-x_n}{\lambda_n}\right),
\]
we have
\begin{equation}\label{equality1}
\|v_n - \tilde{v}_{n,k}\|_{\mathcal K^q(T)}
= \|v - \tilde{v}_{k}\|_{\mathcal K^q(\lambda_n^{-2}T)}
\le \|v - \tilde{v}_{k}\|_{\mathcal K^q(\infty)}
\end{equation}
for any $n\in\mathbb N$. 
Hence,
\[
\begin{split}
\bigg\|\int_0^t e^{(t-\tau)\Delta}&(e_n(\tau) - \tilde{e}_{n,k}(\tau))\,d\tau\bigg\|_{\mathcal K^q(T)} \\
& \le C 
\big(\|v_n\|_{\mathcal K^q(T)}^{2^*(\gamma)-2} + \|\tilde{v}_{n,k}\|_{\mathcal K^q(T)}^{2^*(\gamma)-2} \big)
\|v_n - \tilde{v}_{n,k}\|_{\mathcal K^q(T)}\\
& \le C 
\big(\|v\|_{\mathcal K^q(\infty)}^{2^*(\gamma)-2} + \|\tilde{v}_{k}\|_{\mathcal K^q(\infty)}^{2^*(\gamma)-2} \big)
\|v - \tilde{v}_{k}\|_{\mathcal K^q(\infty)}.
\end{split}
\]
By combining the above estimate and the convergence \eqref{sufficient1}, 
we obtain \eqref{appro-3.1}. Thus, we can construct approximation sequences $\{\tilde{v}_{n,k}\}_{k=1}^\infty$ of $v_n$ with compact support in $\mathbb R^d$ and \eqref{appro-3.1}. 

Finally, if we can prove that
\begin{equation}\label{reduction}
\left\|\int_0^t e^{(t-\tau)\Delta}(\tilde{e}_{n,k}(\tau))\,d\tau\right\|_{\mathcal K^q(T)} 
\to 0\quad \text{as }n\to\infty\text{ for each }k\in\mathbb N,
\end{equation}
then we conclude \eqref{error'}. 
Hence, we will prove \eqref{reduction}. 
Making the changes $\tau' = \tau/\lambda_n^2$ and $x' = (x-x_n)/\lambda_n$, and then putting $t_n := t/\lambda_n^2$ and $\tilde{x}_n := x_n/\lambda_n$, 
we have 
\[
\begin{split}
& \left\|\int_0^t e^{(t-\tau)\Delta}(\tilde{e}_{n,k}(\tau))\,d\tau\right\|_{L^q}\\
& \le C \lambda_n^{-d(\frac{1}{q_c}-\frac{1}{q})}
\int_0^{t_n} (t_n - \tau')^{-\frac{d}{2}(\frac{2^*(\gamma)-1}{q} - \frac{1}{q})} 
\||\cdot + \tilde{x}_n |^{-\gamma} |\tilde{v}_{k} (\tau', \cdot)|^{2^*(\gamma)-1}
\|_{L^\frac{q}{2^*(\gamma)-1}}
\, d\tau'.
\end{split}
\]
Since the support of $\tilde{v}_{k}$ is compact, we have 
$|x+\tilde{x}_n|^{-\gamma} \sim |\tilde{x}_n|^{-\gamma}$ for a sufficiently large $n$. 
Hence 
\[
\left\|\int_0^t e^{(t-\tau)\Delta}(\tilde{e}_{n,k}(\tau))\,d\tau\right\|_{L^q}
\le  C 
|x_n |^{-\gamma}
t^{-\frac{d}{2}(\frac{1}{q_c}-\frac{1}{q}) + \frac{\gamma}{2}}
\|\tilde{v}_{k}
\|_{\mathcal K^q(t_n)}^{2^*(\gamma)-1},
\]
which implies that 
\[
\left\|\int_0^t e^{(t-\tau)\Delta}(\tilde{e}_{n,k}(\tau))\,d\tau\right\|_{\mathcal K^q(T)}
\le 
C T^\frac{\gamma}{2} |x_n|^{-\gamma} \|\tilde{v}_k\|_{\mathcal K^q}^{2^*(\gamma)-1}.
\]
Since $|x_n|^{-\gamma} \to 0$ as $n\to\infty$, 
we obtain \eqref{reduction}. The proof of Lemma \ref{lem:key1} is complete.
\end{proof}

\begin{proof}[Proof of Lemma \ref{lem:key2}]
We may assume that $u^j_0 \in \dot H^1(\R^d) \cap L^\infty(\R^d)$ for $j=0,1$ without loss of generality, taking (iii) in Proposition \ref{prop:wellposed1} into account.  
Let $\{v_k^j\}_{k=1}^\infty$ and $\{v_{n,k}^j\}_{k=1}^\infty$ be approximation sequences of $u^j$ and $u_n^j$ defined by 
\[
v_k^j(t,x) := \eta_k(t)\chi_k(x)u^j(t,x)
\quad \text{and}\quad 
v_{n,k}^j(t,x) :=\lambda_n^{-\frac{d-2}{2}} v_k^j\left( \frac{t}{\lambda_n^2}, \frac{x-x_n}{\lambda_n} \right),
\]
respectively, where $\chi_k \in C^\infty_0(\mathbb R^d)$ is such that $\chi_k(x) \to 1$ as  $k\to\infty$ for each $x\in\mathbb R^d$, and 
$\eta_k \in C^\infty([0,\infty))$ is such that $\eta_k(t) = 1$ for $t\in[0,k]$ and $\eta_k(t)=0$ for $t\in [2k,\infty)$. 
By a similar argument to the proof of  Lemma \ref{lem:key1}, 
to prove \eqref{eq.key}, it suffices to show that 
\begin{equation}\label{eq,key_k}
\lim_{n\to \infty} \sup_{t\in (0,T)}t^{\frac{d(2^*(\gamma)-1)}{2}(\frac{1}{q_c}-\frac{1}{q})}
\||v_{n,k}^1(t)|^{2^*(\gamma)-2}|v_{n,,k}^2(t)|\|_{L^{\frac{q}{2^*(\gamma)-1}}} = 0
\end{equation}
for each $T>0$ and $k\in\mathbb N$. 
Taking \eqref{aym-orth} into account, we divide the proof of \eqref{eq,key_k} into two cases: 
$\lambda_n^1/\lambda_n^2 \to 0$ or $+\infty$ and $|x_n^1 - x_n^2|^2/(\lambda_n^1\lambda_n^2)\to + \infty$. 

First, we consider the case $\lambda_n^1/\lambda_n^2 \to 0$. 
By making the changes $t'=t/(\lambda_n^1)^2$ and $x'=x/\lambda_n^1$, we have
\[
\begin{split}
&
\sup_{t\in (0,T)}t^{\frac{d(2^*(\gamma)-1)}{2}(\frac{1}{q_c}-\frac{1}{q})}
\||v_{n,k}^1(t)|^{2^*(\gamma)-2}|v_{n,k}^2(t)|\|_{L^{\frac{q}{2^*(\gamma)-1}}}\\
&
= \left(\frac{\lambda_n^1}{\lambda_n^2}\right)^{\frac{d-2}{2}} 
\sup_{t'\in (0, \frac{T}{(\lambda_n^1)^2})}
\Bigg(\int_{\R^d} 
\bigg(
\bigg|
v_{k}^1\bigg(t', x' - \frac{x_n^1}{\lambda_n^1} 
\bigg)\bigg|^{2^*(\gamma)-2}\\
& \qquad\qquad\qquad\qquad\qquad\qquad \times
\bigg|
v_{k}^2
\bigg(
\bigg(
\frac{\lambda_n^1}{\lambda_n^2}
\bigg)^2
t', \frac{\lambda_n^1}{\lambda_n^2}x' - \frac{x_n^2}{\lambda_n^2}
\bigg)
\bigg|
\bigg)^{\frac{q}{2^*(\gamma)-1}}\, 
dx
\Bigg)^{\frac{2^*(\gamma)-1}{q}}\\
& \le \left(\frac{\lambda_n^1}{\lambda_n^2}\right)^{\frac{d-2}{2}} 
\sup_{t\in (0,\infty)} \|v_{k}^1(t)\|_{L^\infty}^{2^*(\gamma)-2}
\|v_{k}^2(t)\|_{L^\infty} |\supp \chi_{k}|,
\end{split}
\]
where $|\supp \chi_{k}|$ is the measure of $\supp \chi_{k}$. 
Here, we note that 
\[
v_{k}^j \in L^\infty([0,\infty) ; L^\infty(\R^d))
\]
for $j=0,1$. 
Therefore, we obtain \eqref{eq,key_k} in the case $\lambda_n^1/\lambda_n^2 \to 0$. 
As to the case $\lambda_n^1/\lambda_n^2 \to +\infty$, 
we only have to make the changes $t'=t/(\lambda_n^2)^2$ and $x'=x/\lambda_n^2$, and 
perform the same argument as above.

Next, we consider the case $|x_n^1 - x_n^2|^2/(\lambda_n^1\lambda_n^2)\to + \infty$, 
which implies that $|x_n^1-x_n^2|/\lambda_n^1\to +\infty$ or $|x_n^1-x_n^2|/\lambda_n^2\to +\infty$ as $n\to \infty$. 
It suffices to consider the case where 
$\sup_{n}\lambda_n^1/\lambda_n^2\in (0,\infty)$ and $|x_n^1-x_n^2|/\lambda_n^1\to +\infty$ as $n\to \infty$, 
as the other cases are similar. 
By making the changes $t'=t/(\lambda_n^1)^2$ and  $x'=(x-x_n^1)/\lambda_n^1$, we have
\[
\begin{split}
&
\sup_{t\in (0,T)}t^{\frac{d(2^*(\gamma)-1)}{2}(\frac{1}{q_c}-\frac{1}{q})}
\||v_{n,k}^1(t)|^{2^*(\gamma)-2}|v_{n,k}^2(t)|\|_{L^{\frac{q}{2^*(\gamma)-1}}}\\
&
= \left(\frac{\lambda_n^1}{\lambda_n^2}\right)^{\frac{d-2}{2}} 
\sup_{t'\in (0, \frac{T}{(\lambda_n^1)^2})}
\Bigg(\int_{\R^d} 
\bigg(
|
v_{k}^1(t', x')|^{2^*(\gamma)-2}\\
& \qquad\qquad\qquad\qquad\qquad\qquad \times
\bigg|
v_{k}^2
\bigg(
\bigg(
\frac{\lambda_n^1}{\lambda_n^2}
\bigg)^2
t', \frac{\lambda_n^1}{\lambda_n^2}x' + \frac{x_n^1-x_n^2}{\lambda_n^2}
\bigg)
\bigg|
\bigg)^{\frac{q}{2^*(\gamma)-1}}\, 
dx
\Bigg)^{\frac{2^*(\gamma)-1}{q}}.
\end{split}
\]
Then, we also obtain \eqref{eq,key_k} 
for each $k\in\mathbb N$, 
as the integrand is identically zero for a sufficiently large $n$.  
From the above, \eqref{eq,key_k} is proved in all cases.
Thus, we conclude Lemma \ref{lem:key2}. 
\end{proof}

\subsection{Proof of the blow-up part (ii)}\label{sub:3.2}
In this subsection, we shall prove only the former part of (ii) in Theorem \ref{thm:GD}. 
We omit the proof of the latter part, as 
it is almost the same as in that of Theorem 2.7 in \cite{IT-arxiv}. 

The main idea of proof of the former part is as follows. 
Our proof is based on Levine's concavity method for the spatially localized solution. 
More precisely, 
the usual concavity method uses the $L^2$-norm of the solution; however, our solutions $u$ do not necessarily belong to $L^2(\mathbb R^d)$. 
Hence, instead of $u$ itself, we apply the concavity method to the $L^2$-norm of $\chi_R u$ multiplied by the cut-off function $\chi_R$.
Then, the remainder term $((\chi_R^2-1)u(t),\partial_t u(t))_{L^2(\mathbb R^d)}$ appears (see Lemma \ref{nehari2}). 
To deal with this remainder term, we use Lemma \ref{lem:key-blowup}. 
This lemma is technical but crucial, and is proved by the advantage of the decay effect of $|x|^{-\gamma}$ at infinity.

\medskip

Let us define a cut-off function $\chi \in C^\infty_0(\R^d)$ by 
\[
0 \le \chi(x) \le 1\quad \text{for }x\in\mathbb R^d\quad \text{and}\quad 
\chi(x) =
\begin{cases}
1\quad &\text{if }|x|\le 1,\\
0 &\text{if }|x|\ge 2, 
\end{cases}
\]
and $\chi_R (x) := \chi(x/R)$ for $R>0$.\\

We prepare two lemmas as follows:

\begin{lem}
\label{nehari2}
Let $d\ge 3$, $0<\gamma<2$, $u_0\in \dot{H}^1(\R^d)$, and $u$ be a mild solution to \eqref{crtHS} on $[0,T_m)$. 
Then, $\partial_tu\in L^{\infty}([0,T];\dot{H}^{-1}(\R^d))$ for any $T\in (0,T_m)$. Moreover, the identity
\[
\frac12\frac{d}{dt}\|\chi_R u(t)\|_{L^2(\mathbb R^d)}^2 = - J_\gamma(u(t)) + ((\chi_R^2-1)u(t),\partial_t u(t))_{L^2(\mathbb R^d)}
\]
holds for any $t\in [0,T_m)$, where $J_\gamma:\dot{H}^1(\R^d)\rightarrow\R$ is the Nehari functional given by \eqref{Nehari}. 
\end{lem}

\begin{proof}
In the similar manner as the proof of the estimate \eqref{nonlin.est3} in Lemma \ref{l:crtHS.nonlin.est}, there exists $C=C(T)>0$ such that
\[
\| \partial_tu \|_{L^{\infty}([0,T];\dot H^{-1}(\R^d))}
\le C\big(\| u_0 \|_{\dot{H}^1(\R^d)}
+
\|u \|_{L^{\infty}([0,T];\dot{H}^1(\R^d))}^{2^*(\gamma)-1}
+
\|u\|_{\mathcal K^q(T)}^{2^*(\gamma)-1}\big)
 < \infty,
\]
indicating that $\partial_t u \in L^\infty([0,T]; \dot{H}^{-1}(\mathbb R^d))$. Moreover, it follows that
\[
\sup_{t\in[0,T]}| (u(t), \partial_t u(t))_{L^{2}(\R^d)} | 
\le \| u \|_{L^{\infty}([0,T];\dot H^1(\R^d))} \| \partial_t u \|_{L^{\infty}([0,T];\dot H^{-1}(\R^d))} <\infty.
\]
Then, the identities
\[
\begin{split}
\frac{d}{dt}\|\chi_R u(t)\|_{L^2(\mathbb R^d)}^2
& = 2 (\chi_R u(t), \chi_R \partial_t u(t))_{L^{2}(\R^d)}\\
& = 2 (u(t), \partial_t u(t))_{L^{2}(\R^d)} + 2((\chi_R^2-1)u(t), \partial_t u(t))_{L^{2}(\R^d)}\\
& = -2 J_\gamma(u(t)) + 2((\chi_R^2-1)u(t), \partial_t u(t))_{L^{2}(\R^d)}
\end{split}
\]
hold for any $t\in (0,T_m)$, which completes the proof of Lemma \ref{nehari2}.
\end{proof}

\begin{lem}
\label{lem:key-blowup}
Let $d\ge 3$, $0<\gamma<2$, $u_0\in \dot{H}^1(\R^d)$, and $u$ be a mild solution to \eqref{crtHS} on $[0,T_m)$. 
Suppose that $T_m = +\infty$ and 
\begin{equation}\label{unfirom-bdd}
\limsup_{t\to \infty}\|u(t)\|_{\dot H^1(\R^d)} < \infty.
\end{equation}
Let $T_R := R^2 \|u_0\|_{L^{q_c}(|x|\ge R)}$ for $R>0$. 
Then,
\[
\lim_{R\to\infty}\sup_{t \in (0,T_R)} \left|((\chi_R^2-1)u(t), \partial_t u(t))_{L^{2}(\R^d)}\right| = 0.
\]
\end{lem}
\begin{proof}
By the assumption \eqref{unfirom-bdd}, we have
\begin{equation}\label{unfirom-bdd2}
\sup_{t\in (0,\infty)}\|u(t)\|_{L^{q_c}(\mathbb R^d)} \le \sup_{t\in (0,\infty)}\|u(t)\|_{\dot H^1(\mathbb R^d)} <\infty.
\end{equation}
Noting from Remark \ref{rem:classical} that $u$ satisfies the differential equation 
\[
\partial_t u = \Delta u + |x|^{-\gamma} |u|^{2^*(\gamma)-2}u\quad  \text{in } (0,\infty) \times \{|x|\ge R\}, 
\]
we write
\[
\left|((\chi_R^2-1)u(t), \partial_t u(t))_{L^{2}(\R^d)}\right| = I_R+II_R,
\]
where 
\[
I_R := \left|((\chi_R^2-1)u(t), \Delta u(t))_{L^{2}(\R^d)} \right|,
\]
\[
II_R := \left|((\chi_R^2-1)u(t), |x|^{-\gamma} |u(t)|^{2^*(\gamma)-2}u(t))_{L^{2}(\R^d)}\right|.
\]
By H\"older's inequality and the critical Sobolev embedding, we estimate
\[
\begin{split}
I_R 
& = 
\left|\int_{\R^d} \left(\nabla (\chi_R^2(x)-1) \cdot u(t,x) \nabla u(t,x) + (\chi_R^2(x)-1) |\nabla u(t,x)|^2\right)\, dx \right|\\
& \le C
\left( \|\nabla (\chi_R^2-1)\|_{L^d(\mathbb R^d)} \|u(t)\|_{L^{q_c}(R \le|x|\le 2R)} \|u(t)\|_{\dot H^1(|x|\ge R)}
+
\|u(t)\|_{\dot H^1(|x|\ge R)}^2
\right)\\
& \le C 
\|u(t)\|_{\dot H^1(|x|\ge R)}^2.
\end{split}
\]
By the Hardy-Sobolev inequality, we have
\[
\begin{split}
II_R 
& =
\int_{\mathbb R^d} \frac{|(1-\chi_R^2(x))^{\frac{1}{2^*(\gamma)}}u(t,x)|^{2^*(\gamma)}}{|x|^\gamma}\, dx\\
& \le C
\|(1-\chi_R^2)^{\frac{1}{2^*(\gamma)}}u(t)\|_{\dot H^1(\mathbb R^d)}^{2^*(\gamma)}\\
& \le C 
\|u(t)\|_{\dot H^1(|x|\ge R)}^{2^*(\gamma)}.
\end{split}
\]
Combining the just obtained estimates 
and the uniform bound \eqref{unfirom-bdd2}, we have
\[
\sup_{t \in (0,T_R)} \left|((\chi_R^2-1)u(t), \partial_t u(t))_{L^{2}(\R^d)}\right|
\le C \sup_{t \in (0,T_R)} \|u(t)\|_{\dot H^1(|x|\ge R)}^2.
\]
Hence, all we have to do is to prove that 
\begin{equation}\label{lem3.6-goal}
\lim_{R\to\infty}\sup_{t \in (0,T_R)} \|u(t)\|_{\dot H^1(|x|\ge R)} =0.
\end{equation}
We proceed to estimate
\begin{multline}\label{sec3.2-1}
\|u(t)\|_{\dot H^1(|x|\ge R) }
\le \|e^{t\Delta} u_0\|_{\dot H^1(|x|\ge R)} \\
+ 
\left\|
\int_0^t 
e^{(t-\tau)\Delta} (|x|^{-\gamma} |u(\tau)|^{2^*(\gamma)-2}u(\tau))\, d\tau
\right\|_{\dot H^1(|x|\ge R)}.
\end{multline}
For the first term, the stronger assertion
\begin{equation}\label{sec3.2-2}
\lim_{R\to\infty}\sup_{t\in [0,\infty)}  \|e^{t\Delta} u_0\|_{\dot H^1(|x|\ge R)}=0
\end{equation}
holds. In fact, as $e^{t\Delta} u_0$ is dissipative, 
for any $R>0$, there exists $t_R \in [0, \infty)$ such that 
it attains the supremum of $\|e^{t\Delta} u_0\|_{\dot H^1(|x|\ge R)}$ over $t\in [0,\infty)$. 
Moreover, after possibly passing to a subsequence (in which case, we rename it $t_R$), 
there exists $t_\infty \in [0,\infty]$ such that $t_R \to t_\infty$ as $R\to\infty$. 
If $t_\infty <\infty$, then 
\[
\begin{split}
\sup_{t\in [0,T_R]}  \|e^{t\Delta} u_0\|_{\dot H^1(|x|\ge R)}
& = \|e^{t_R\Delta} u_0\|_{\dot H^1(|x|\ge R)}\\
& \le \|e^{t_R\Delta} u_0 - e^{t_\infty\Delta} u_0\|_{\dot H^1(|x|\ge R)} + \|e^{t_\infty\Delta} u_0\|_{\dot H^1(|x|\ge R)}.
\end{split}
\]
Because $e^{t\Delta}u_0 \in C([0,T];\dot H^1(\R^d))$ for any $T>0$, 
the right-hand side in the above converges to $0$ as $R\to \infty$.
If $t_\infty=\infty$, then it also converges to $0$ as $R\to\infty$, as $e^{t\Delta}u_0$ is dissipative. 
Hence, \eqref{sec3.2-2} is obtained.

For the second term, 
we write 
\begin{equation}\label{sec3.2-3}
\big(e^{(t-\tau)\Delta} (|x|^{-\gamma} |u(\tau)|^{2^*(\gamma)-2}u(\tau)\big)(x) =  A_R(t-\tau, x) + B_R(t-\tau, x),
\end{equation}
where
\[
A_R(t-\tau, x) := \int_{|y|\le \frac{R}{2}}
G(t-\tau, x-y) |y|^{-\gamma}|u(\tau,y)|^{2^*(\gamma)-2}u(\tau,y)\, dy,
\]
\[
B_R(t-\tau, x) := \int_{|y|> \frac{R}{2}}
G(t-\tau, x-y) |y|^{-\gamma}|u(\tau,y)|^{2^*(\gamma)-2}u(\tau,y)\, dy.
\]
Here, we recall that $G$ is the heat kernel given in \eqref{gaussian-kernel}. 
First, we estimate the term $A_R(t-\tau, x)$.
If $|x|\ge R$ and $|y|\le R/2$, then 
\[
|x-y|^2 \ge \frac{1}{2}|x-y|^2 + \frac{R^2}{8};
\]
hence, 
\[
\begin{split}
|\nabla_x G(t-\tau, x-y) |
& = (4\pi (t-\tau))^{-\frac{d}{2}} \frac{|x-y|}{2(t-\tau)} e^{-\frac{|x-y|^2}{4(t-\tau)}}\\
& \le 
2e^{-\frac{R^2}{32(t-\tau)}}(4\pi (t-\tau))^{-\frac{d}{2}} \frac{|x-y|}{2\cdot 2(t-\tau)} e^{-\frac{|x-y|^2}{4 \cdot 2(t-\tau)}}
\end{split}
\]
for any $\tau\in (0,t)$ and $|x|\ge R$ and $|y|\le R/2$.
Then, by (ii) in Lemma \ref{l:lin.est.HS} and the uniform bound \eqref{unfirom-bdd2}, we have
\[
\begin{split}
&\left\|
\int_0^t 
A_R(t-\tau, x)\, d\tau
\right\|_{\dot H^1(|x|\ge R)}\\
& \le 
C\int_0^t 
e^{-\frac{R^2}{32(t-\tau)}} (t-\tau)^{-1} 
\|
|u(\tau)|^{2^*(\gamma)-2}u(\tau)
\|_{L^\frac{q_c}{2^*(\gamma)-1}(|x|< \frac{R}{2})}\, d\tau\\
& \le 
C\int_0^t 
e^{-\frac{R^2}{32(t-\tau)}} (t-\tau)^{-1} \, d\tau
\cdot 
\|u\|_{L^\infty([0,\infty); L^{q_c}(\R^d))}^{2^*(\gamma)-1}\\
& \le 
C \int_0^{\|u_0\|_{L^{q_c}(|x|\ge R)}} 
e^{-\frac{1}{32\tau'}} \tau'^{-1} \, d\tau'
\end{split}
\]
for any $t\in (0,T_R)$. Hence
\begin{equation}\label{sec3.2-4}
\lim_{R\to\infty} \sup_{t\in [0,T_R]}\left\|
\int_0^t 
A_R(t-\tau, x)\, d\tau
\right\|_{\dot H^1(|x|\ge R)} =0.
\end{equation}
Next, we estimate the term $B_R(t-\tau, x)$.
Again using (ii) in Lemma \ref{l:lin.est.HS}, we estimate
\[
\begin{split}
& \left\|
\int_0^t 
B_R(t-\tau, x)\, d\tau
\right\|_{\dot H^1(|x|\ge R)}\\
& \le 
C \int_0^t (t-\tau)^{-1 + \frac{\gamma_0}{2}} 
\|
|x|^{-\gamma_0}|u(\tau)|^{2^*(\gamma)-2}u(\tau)
\|_{L^\frac{q_c}{2^*(\gamma)-1}(|x|\ge R)}\, d\tau\\
& \le C R^{-\gamma_0}
\int_0^t (t-\tau)^{-1 + \frac{\gamma_0}{2}} \, d\tau
\cdot \|u\|_{L^\infty([0,\infty); L^{q_c}(\R^d))}^{2^*(\gamma)-1}\\
& \le C R^{-\gamma_0} T_R^{\frac{\gamma_0}{2}}\\
& = C \|u_0\|_{L^{q_c}(|x|\ge R)}^{\frac{\gamma_0}{2}}
\end{split}
\]
for any $t\in (0,T_R)$, 
where $\gamma_0 \in (0, \gamma]$ satisfies 
\[
\frac12 < \frac{\gamma-\gamma_0}{d} + \frac{2^*(\gamma)-1}{q_c}\quad \text{i.e.,}\quad \gamma_0 <1.
\]
Hence,
\begin{equation}\label{sec3.2-5}
\lim_{R\to\infty} \sup_{t\in [0,T_R]}\left\|
\int_0^t 
B_R(t-\tau, x)\, d\tau
\right\|_{\dot H^1(|x|\ge R)} =0.
\end{equation}
Therefore, by summarizing \eqref{sec3.2-1}--\eqref{sec3.2-5}, we obtain \eqref{lem3.6-goal}. 
The proof of Lemma \ref{lem:key-blowup} is thus complete.
\end{proof}

With the proof of Lemma \ref{lem:key-blowup} complete,
we are now in a position to prove the former part of (ii) in Theorem \ref{thm:GD}.

\begin{proof}[Proof of the former part of {\rm (ii)} in Theorem \ref{thm:GD}]
Let $u_0 \in \mathcal M^-$ with $E_\gamma(u_0) < l_{HS}$. 
For contradiction, we suppose that $T_m(u_0) = +\infty$ and 
\[
\limsup_{t\to \infty}\|u(t)\|_{\dot H^1(\R^d)} < \infty.
\]
Given $R>0$ and $A>0$, we define a function $I_R:[0,\infty)\rightarrow [0,\infty)$ given by
\[
I_R(t) := \int_0^t \| \chi_R u(s) \|_{L^2(\R^d)}^2 \,ds + A.
\]
Here, by H\"older's inequality and the critical Sobolev embedding $\dot{H}^1(\R^d)\hookrightarrow L^{q_c}(\R^d)$, there exists a constant $C>0$ depending only on $d$ such that 
\begin{equation}\label{ineq-R}
\| \chi_R u(t) \|_{L^2(\R^d)} \le \|\chi_R\|_{L^{d}(\R^d)} \|u(t) \|_{L^{q_c}(|x|\ge R)} \le CR\| u(t) \|_{\dot H^1(\R^d)}
\end{equation}
for any $t\ge 0$, which implies that $\chi_R u (t) \in L^2(\R^d)$ for any $t\ge 0$. 
Then, by a direct computation and Lemma \ref{nehari2}, the identities
\begin{align}
I_R'(t)&=\| \chi_R u(t) \|_{L^2(\R^d)}^2,\notag\\
I_R''(t)&= -2 J_\gamma (u(t)) + 2((\chi_R^2-1)u(t), \partial_t u(t))_{L^{2}(\R^d)}
\label{twice deriI}
\end{align}
hold for any $t>0$. Set $T_R := R^2\|u_0\|_{L^{q_c}(|x|\ge R)}$. 
Then, we can prove that the estimate
\begin{equation}\label{eq.keypoint2}
I_R''(t)I_R(t) - (1+\alpha)I'_R(t)^2 > 0
\end{equation}
holds for any $t\in (0,T_R]$, if $R$ and $A$ are sufficiently large and $\alpha>0$ is sufficiently small. 
This fact is verified as follows. 
For any $\varepsilon>0$, by a fundamental calculous and Schwartz's inequality, the estimates
\begin{align*}
I_R'(t)^2
& = \left(\| \chi_Ru_0\|_{L^2(\R^d)}^2 + 2 \Re \int_0^t (\chi_Ru,\chi_R\partial_t u)_{L^2(\R^d)}\,ds\right)^2\\
& \le (1+\varepsilon^{-1}) \| \chi_Ru_0\|_{L^2(\R^d)}^4 + 4(1+\varepsilon)\left(\int_0^t (\chi_Ru,\chi_R\partial_t u)_{L^2(\R^d)}\,ds\right)^2\\
& \le (1+\varepsilon^{-1}) \| \chi_Ru_0\|_{L^2(\R^d)}^4 \\
& \qquad \qquad \qquad + 4(1+\varepsilon)
\left(\int_0^t \| \chi_Ru(s)\|_{L^2(\R^d)}^2\,ds\right) \left(\int_0^t \| \chi_R\partial_t u (s)\|_{L^2(\R^d)}^2\,ds\right)
\end{align*}
hold for any $t>0$. Next, we estimate $I_R''(t)$ from below. By the identity \eqref{twice deriI}, statement (iii) in Lemma \ref{lem:invariant} and the energy identity \eqref{energy-id2}, the estimates
\begin{align*}
I_R''(t)
&\ge 2\cdot 2^*(\gamma) \{ l_{HS} - E(u(t))\}+ 2((\chi_R^2-1)u(t), \partial_t u(t))_{L^{2}(\R^d)}\\
&\ge
2\cdot2^*(\gamma) \left( l_{HS} - E(u_0) + \int_0^t \| \chi_R \partial_t u (s)\|_{L^2(\R^d)}^2\,ds \right)\\
& \qquad \qquad \qquad \qquad \qquad \qquad \qquad \qquad - 2|((\chi_R^2-1)u(t), \partial_t u(t))_{L^{2}(\R^d)}|
\end{align*}
hold for any $t\in (0,T_m)$. By Lemma \ref{lem:key-blowup}, there exists $R_1>0$ such that
\[
\sup_{t\in [0, T_R]} 2|((\chi_{R_1}^2-1)u(t), \partial_t u(t))_{L^{2}(\R^d)}| \le 2^*(\gamma) \{l_{HS} - E(u_0)\}
\]
for any $R\ge R_1$.
Hence, 
\[
I_{R}''(t) 
\ge 
2\cdot2^*(\gamma) \left( \frac{l_{HS} - E(u_0)}{2} + \int_0^t \| \chi_{R} \partial_t u (s)\|_{L^2(\R^d)}^2\,ds \right)
\]
for any $t\in (0,T_R)$ and $R\ge R_1$. 
Let $\alpha>0$. By summarizing the above estimates,
we have
\[
\begin{split}
&I_{R}''(t)I_{R}(t) - (1+\alpha)I_{R}'(t)^2 \\
\ge &\,
2\cdot2^*(\gamma) \left( \frac{l_{HS} - E(u_0)}{2} + \int_0^t \| \chi_{R}\partial_t u (s)\|_{L^2(\R^d)}^2\,ds \right)
\left(  \int_0^t \|\chi_{R} u(s)\|_{L^2(\R^d)}^2\,ds + A \right)\\
- &\, 4(1+\alpha)(1+\varepsilon) \left(\int_0^t \| \chi_{R} u(s)\|_{L^2(\R^d)}^2\,ds\right)
\left(\int_0^t \| \chi_{R} \partial_t u (s)\|_{L^2(\R^d)}^2\,ds\right) \\
- &\, (1+\alpha)(1+\varepsilon^{-1}) \| \chi_{R}u_0\|_{L^2(\R^d)}^4
\end{split}
\]
for any $t\in (0,T_R)$ and $R\ge R_1$. 
Hence, 
specifying a sufficiently large $A$ as well as sufficiently small $\varepsilon$ and $\alpha$ 
so that
\[
2\cdot2^*(\gamma) > 4(1+\alpha)(1+\varepsilon)
\quad \text{and}\quad 
2^*(\gamma)\{l_{HS} - E(u_0)\} A > (1+\alpha)(1+\varepsilon^{-1}) \| \chi_{R}u_0\|_{L^2(\R^d)}^4
\]
for any $t\in (0,T_R)$ and $R\ge R_1$, 
we obtain \eqref{eq.keypoint2}. Here, we note that $2^*(\gamma)>2$. 
We take 
\[
A = A(R) = \frac{2(1+\alpha)(1+\varepsilon^{-1}) \| \chi_{R}u_0\|_{L^2(\R^d)}^4}{2^*(\gamma)\{l_{HS} - E(u_0)\}}.
\]
The inequality \eqref{eq.keypoint2} is equivalent to 
\[
\frac{d}{dt}\left(
\frac{I_R'(t)}{I_R(t)^{\alpha+1}} \right)
>0.
\]
This implies that 
\[
\frac{I_R'(t)}{I_R(t)^{\alpha+1}} > \frac{I_R'(0)}{I_R(0)^{\alpha+1}} = \frac{\| \chi_Ru_0\|_{L^2}^2}{A^{\alpha+1}}=: a.
\]
By integrating the above inequality over $t \in [0,t]$, we have
\[
\frac{1}{\alpha} \left(
\frac{1}{I_R(0)^\alpha }- \frac{1}{I_R(t)^\alpha}
\right)
> at;
\]
hence, 
\[
I_R(t)^\alpha >
\frac{I(0)^\alpha}{1-I_R(0)^\alpha\alpha a t} \to +\infty
\]
as $t \to 1/(I_R(0)^\alpha\alpha a )= A/(\alpha \|\chi_Ru_0\|_{L^2}^2 ) =:\tilde{t}_R$. 
Now, there exists $R_2>0$ such that $\tilde{t}_R \le T_R$ for any $R\ge R_2$, 
as 
\[
\tilde{t}_R = \frac{A}{\alpha \|\chi_Ru_0\|_{L^2}^2} = C \|\chi_Ru_0\|_{L^2}^2 \le CR^2 \|u_0\|_{L^{q_c}(|x|\ge R)}^2 
\quad \text{and}\quad \lim_{R\to\infty}\|u_0\|_{L^{q_c}(|x|\ge R)}=0.
\]
Set $R_0 := \max(R_1,R_2)$. 
Then, we obtain 
\[
\limsup_{t\to \tilde{t}_{R_0}-} \|\chi_{R_0}u(t)\|_{L^2} = +\infty.
\]
Therefore, we find from \eqref{ineq-R} that
\begin{equation*}\label{eq.blowup2}
\limsup_{t\to \tilde{t}_{R_0}-} \| u(t)\|_{\dot H^1(\R^d)} = +\infty.
\end{equation*}
However, this contradicts $T_m=T_m(u_0) = +\infty$; that is,
\[
u \in C([0,T]; \dot H^1(\R^d))\quad \text{for any $T>0$}
\]
by the result on local well-posedness in Proposition \ref{prop:wellposed1}. 
Thus, we conclude that $T_m < + \infty$ or 
\[
\limsup_{t\to \infty}\|u(t)\|_{\dot H^1(\R^d)} = \infty.
\]
The proof of the former part of (ii) in Theorem~\ref{thm:GD} is complete.
\end{proof}

\section{The absorbing case and Dirichlet problem}\label{sec:4}

\subsection{The absorbing case}\label{sub:4.1}
In this subsection, we mention the absorbing case:
\begin{equation}\label{crtHS-a}
\begin{cases}
\partial_t u - \Delta u = -|x|^{-\gamma} |u|^{2^*(\gamma)-2} u,&(t,x)\in (0,T)\times\R^d, \\
u(0) = u_0.
\end{cases}
\end{equation}
The problem \eqref{crtHS-a} is locally well-posed in $L^{q_c}(\R^d)$ and $\dot H^1(\R^d)$ (see Proposition \ref{prop:wellposed-A} below). 
Moreover, we have the following.

\begin{thm}[Large data dissipation in the absorbing case]\label{thm:absor}
Let $d\ge 3$ and $0\le \gamma <2$. Then, the solution to \eqref{crtHS-a} with initial data in $L^{q_c}(\mathbb R^d)$ is dissipative. 
\end{thm}
\begin{proof}
First, we consider the case in which initial data $u_0 \in \dot H^1(\R^d)$. 
The strategy of proof is almost the same as in (i) in Theorem \ref{thm:GD}. 
In fact, we suppose that $E^c<+\infty$, where $E^c$ is given in \eqref{def:Ec}, i.e.,
\[
\begin{split}
E^c 
:= \sup\big\{ &E \in \R : \ \text{$T_m(u_0)=+\infty$ and $\|u\|_{\mathcal K^q}<\infty$}\\
& \text{ for any solution $u$ to \eqref{crtHS-a} with $u_0\in \dot H^1(\R^d)$ and $E_\gamma (u_0) < E$}
\big\}.
\end{split}
\]
Here, we note that $J_\gamma(\phi)$ is always non-negative for any $\phi \in \dot H^1(\R^d)$ in the absorbing case. 
We can apply the same argument as in Subsection \ref{sub:3.1} to obtain $E^c=0$. 
This contradicts $E^c>0$ by small-data global existence. 
Hence, we obtain $E^c=+\infty$ by contradiction. 
$E^c=+\infty$ means that all solutions $u$ to \eqref{crtHS-a} with initial data $u_0\in \dot H^1(\R^d)$ are global in time and dissipative 
by Proposition \ref{prop:diss-A} below. 
Finally, by the density argument and Proposition~\ref{prop:perturbation-A}, we can prove that 
all solutions $u$ to \eqref{crtHS-a} with initial data $u_0\in L^{q_c}(\R^d)$ are also global in time and dissipative in $L^{q_c}(\R^d)$. 
Thus, we conclude Theorem \ref{thm:absor}.
\end{proof}

\subsection{Dirichlet problem}\label{sub:4.2}
Let $\Omega$ be a domain of $\R^d$ that contains the origin $0$. 
We consider the Dirichlet problem of the critical Hardy-Sobolev parabolic equation:
\begin{equation}\label{crtHS-D}
\begin{cases}
\partial_t u - \Delta u = |x|^{-\gamma} |u|^{2^*(\gamma)-2} u,&(t,x)\in (0,T)\times\Omega, \\
u|_{\partial \Omega} = 0,\\
u(0) = u_0.
\end{cases}
\end{equation}
For simplicity, we take initial data $u_0$ as a function in the inhomogeneous space $H^1_0(\Omega)$. 
The problem \eqref{crtHS-D} is locally well-posed in $H^1_0(\Omega)$ (see Proposition \ref{prop:wellposed-A} below). 
To state our result, we introduce the energy functional $E_{\gamma,\Omega} : H^1_0(\Omega) \to \R$ and the Nehari functional $J_{\gamma,\Omega} : H^1_0(\Omega) \to \R$
associated with \eqref{crtHS-D} as follows:
\[
E_{\gamma,\Omega}(\phi) := \frac{1}{2} \|\phi\|_{\dot H^1(\Omega)}^2 
	- \frac{1}{2^*(\gamma)} 
	\int_{\Omega} \frac{ |\phi(x)|^{2^*(\gamma)} }{|x|^\gamma} dx,
\]
\[
J_{\gamma,\Omega}(\phi) := \frac{d}{d\lambda} E_{\gamma,\Omega}(\lambda \phi)|_{\lambda=1} = \|\phi\|_{\dot H^1(\Omega)}^2 - \int_{\Omega} \frac{ |\phi(x)|^{2^*(\gamma)} }{|x|^\gamma} dx.
\]
Moreover, we define the mountain pass energy $l_{HS}(\Omega)$ by 
\[
l_{HS}(\Omega)
:= 
\inf_{\phi\in H^1_0(\Omega) \setminus \{0\}} \max_{\lambda\ge0} E_{\gamma,\Omega}(\lambda \phi).
\]
Then, it is represented by the best constant $C_{HS}(\Omega)$ of the Hardy-Sobolev inequality \eqref{HS-ineq} as follows:
\[
l_{HS}(\Omega)=\frac{2-\gamma}{2(d-\gamma)} C_{HS}(\Omega)^{\frac{2(d-\gamma)}{2-\gamma}}
\]
(see, e.g., Appendix C in \cite{IT-arxiv}). 
Hence, $l_{HS}(\Omega)$ is independent of $\Omega$ by Lemma \ref{t:HrdySob2},
and we simply write $l_{HS}=l_{HS}(\Omega)$.

\medskip

 As a corollary of Theorem \ref{thm:GD} and Lemma \ref{lem:comparison} below, 
 we have a dichotomy between dissipation and blow-up for solutions to the problem \eqref{crtHS-D}. More precisely, we have the following:

\begin{thm}\label{thm:GD-D}
Let $d\ge 3$ and $0\le \gamma <2$, and let $u=u(t)$ be a solution to \eqref{crtHS-D} with initial data $u_0 \in H^1_0(\Omega)$ with $E_{\gamma,\Omega}(u_0)\le l_{HS}$. Then, 
the following statements hold{\rm :}
\begin{itemize}
\item[(i)] If $J_{\gamma,\Omega}(u_0)>0$, then $u$ is dissipative.
\item[(ii)] If $J_{\gamma,\Omega}(u_0)<0$, then $u$ blows up in finite time.
\end{itemize}
\end{thm}

\begin{rem}
Theorem \ref{thm:GD-D} is a generalization of the results by Tan \cite{Tan2001}, which deals with 
the case of $\gamma=0$ and bounded domains $\Omega$. 
\end{rem}

\begin{proof}
Let $u_0\in H^1_0(\Omega)$ with $E_{\gamma,\Omega}(u_0) \le l_{HS}$ and $J_{\gamma,\Omega}(u_0)>0$ 
and $\tilde{u_0}$ be the zero extension of $u_0$ to $\R^d$.  
Then, it is clear that $|\tilde{u_0}| \in H^1(\R^d)$, $E_{\gamma,\Omega}(u_0) = E_{\gamma}(|\tilde{u_0}|)$, and $J_{\gamma,\Omega}(u_0) = J_{\gamma}(|\tilde{u_0}|)$.
Denoting by $\tilde{u}$ a mild solution to the Cauchy problem \eqref{crtHS} with initial data $|\tilde{u}_0|$, 
we have $\|\tilde{u}\|_{\mathcal K^q} < \infty$ by (i) in Theorem \ref{thm:GD}. 
By Lemma~\ref{lem:comparison}, we see that 
$\|u\|_{\mathcal K^q(\Omega)} \le 
\|\tilde{u}\|_{\mathcal K^q} < \infty$. 
Hence (i) in Theorem \ref{thm:GD-D} is proved by Proposition~\ref{prop:diss-A}. 
The statement (ii) can be proved similarly as in Subsection \ref{sub:3.2}, as the same variational results also hold in this case (see, e.g., \cite{IT-arxiv}).
Thus, we omit the proof.
\end{proof}

\appendix

\section{Local theory and dissipation of global solutions}\label{app:A}
Let $\Omega$ be a domain of $\R^d$.  
We study the Dirichlet problem of nonlinear heat equation
\begin{equation}\label{crtHS-Dg}
\begin{cases}
\partial_t u - \Delta u = F(x,u),&(t,x)\in (0,T)\times\Omega, \\
u|_{\partial \Omega} =0, \\
u(0) = u_0,
\end{cases}
\end{equation}
where $F : \Omega\times \mathbb C\to\mathbb C$. 
We rewrite the problem \eqref{crtHS-Dg} as the integral form
\begin{equation}\label{integral-eq}
u(t,x) = (e^{t\Delta_\Omega}u_0)(x) + \int_{0}^t e^{(t-\tau)\Delta_\Omega}F(\cdot,u(\tau,\cdot))(x)\, d\tau
\end{equation}
for any $t \in [0,T)$ and almost everywhere $x\in\Omega$, 
where $-\Delta_\Omega$ is the Dirichlet Laplacian on $L^2(\Omega)$ and 
$\{e^{t\Delta_\Omega}\}_{t>0}$ is the semigroup generated by $-\Delta_\Omega$. 
We regard $\mathbb C$ as the two-dimensional vector space $\R^2$, and assume that 
$F(x,\cdot) \in C^1 (\mathbb R^2; \mathbb R^2)$ with $F(x,0) = 0$ and 
\begin{equation}\label{assum-F}
	|F(x,z_1)-F(x,z_2)| \le C|x|^{-\gamma}(|z_1| + |z_2|)^{2^*(\gamma)-2}|z_1-z_2|
\end{equation}
for almost everywhere $x\in\Omega$ and any $z_1,z_2\in \mathbb C$. 
In this appendix, we discuss the local well-posedness, small-data global existence, and dissipation of global solutions for \eqref{crtHS-Dg} in the critical spaces $L^{q_c}(\Omega)$, $H^1(\Delta_\Omega)$ and $\dot H^1(\Delta_\Omega)$. 
Here, $H^1(\Delta_\Omega)$ and $\dot H^1(\Delta_\Omega)$ are Sobolev spaces associated with $-\Delta_\Omega$ and their norms are given by 
\[
\|f\|_{H^1(\Delta_\Omega)} := \|(I-\Delta_\Omega)^\frac12 f\|_{L^2(\Omega)}\quad\text{and}\quad 
\|f\|_{\dot H^1(\Delta_\Omega)} := \|(-\Delta_\Omega)^\frac12 f\|_{L^2(\Omega)}
\]
via the spectral decomposition of $-\Delta_\Omega$, 
respectively, where $I$ is the identity operator on $L^2(\Omega)$. For these precise definitions, we refer to Definition 1.1 in \cite{IT-arxiv} for instance. 
Note that $H^1(\Delta_\Omega) = H^1_0(\Omega)$, $H^1(\Delta_{\R^d}) = H^1(\R^d)$, and $\dot H^1(\Delta_{\R^d}) = \dot H^1(\R^d)$. 
For convenience, we set 
\[
X = L^{q_c}(\Omega), H^1(\Delta_\Omega)\text{ or } \dot H^1(\Delta_\Omega). 
\]
To state the result on well-posedness, let us introduce the notion of a mild solution.
\begin{definition}\label{def:sol-A}
Let $T \in (0,\infty]$ and $u_0 \in X$. 
A function $u : [0,T) \times \R^d \to \mathbb C$ is called an $X$-mild solution to \eqref{crtHS-Dg} with initial data $u(0)=u_0$ if it satisfies $u\in C([0,T); X)$ and the integral equation \eqref{integral-eq} 
for any $t \in [0,T)$ and almost everywhere $x \in \R^d$.
The time $T$ is said to be the maximal existence time, which is denoted by $T_m=T_m(u_0)$, if the solution cannot be extended beyond $[0,T).$  
We say that $u$ is global in time if $T_m = + \infty$ 
and that $u$ blows up in finite time otherwise. 
Moreover, we say that $u$ is dissipative if $T_m = + \infty$ and 
\[
\lim_{t\to\infty} \|u(t)\|_{X} = 0,
\]
and that $u$ is stationary if $u(t,x) = \phi(x)$ on $[0,\infty)\times\Omega$, where $\phi$ is a solution of the elliptic equation $ - \Delta \phi = F(x,\phi)$. 
\end{definition}

We introduce an auxiliary space of the following type.

\begin{definition}\label{def:Kato-m}
Let $T \in (0,\infty]$ and $q,r\in [1,\infty]$. The space $\mathcal{K}^{q}_r(T,\Omega)$ is defined by 
\[
   \mathcal{K}^{q}_r(T,\Omega):=\left\{u\in \mathscr{D}'([0,T)\times\Omega)\ ;\ \|u\|_{\mathcal{K}^{q}_r(T',\Omega)}
   <\infty\ \text{for any }T' \in (0,T)\right\}
\]
endowed with 
\[
	\|u\|_{\mathcal K^{q}_r(T,\Omega)}
	:= 
	\left(
		\int_0^T (t^{\kappa} \|u(t)\|_{L^q(\Omega)})^r \, dt
	\right)^\frac{1}{r}
\]
if $r<\infty$, and 
\[
\|u\|_{\mathcal K^{q}_{\infty}(T,\Omega)} = \|u\|_{\mathcal K^{q}(T,\Omega)}
	:=\sup_{0\le t\le T}t^{\kappa}\|u\|_{L^q(\Omega)},
\]
where $\kappa$ is given by 
\begin{equation}\label{kappa}
\kappa = \kappa(q,r) := \frac{d}{2} \left( \frac{1}{q_c} - \frac{1}{q}\right) - \frac{1}{r},
\end{equation}
and 
$ \mathscr{D}'([0,T)\times\Omega)$ is the space of distributions on $[0,T)\times\Omega$. 
For simplicity, we may omit 
the symbols $T$ and $\Omega$ in $\mathcal K^{q,\alpha}_r(T,\Omega)$ when $T=\infty$ and $\Omega=\mathbb R^d$, respectively, if they do not cause a confusion.
For example, 
$\mathcal K^{q,\alpha}_r = \mathcal K^{q,\alpha}_r(\infty,\mathbb R^d)$. 
\end{definition}

The auxiliary norm is invariant under the scaling transformation \eqref{crt.HSscale} when $T = \infty$ and $\Omega = \mathbb R^d$. 
Indeed, $\|u_\lambda\|_{\mathcal K^q_r} = \|u\|_{\mathcal K^q_r}$ for any $\lambda>0$ 
if and only if \eqref{kappa} holds.

\medskip

Throughout this appendix, we make the following assumptions on $q$ and $r$. 
We assume that $q\in (1,\infty)$ satisfies 
\begin{equation}\label{l:crtHS.nonlin.est:c}
	\frac{1}{q_c} - \frac{2}{d(2^*(\gamma)-1)}
	< \frac{1}{q} < \frac{1}{q_c} \quad \text{if $X=L^{q_c}(\Omega)$},
\end{equation} 
and
\begin{equation}\label{l:crtHS.nonlin.est:c2'}
	\frac{1}{q_c} - \frac{1}{d(2^*(\gamma)-1)}
	< \frac{1}{q} < \frac{1}{q_c}
	\quad \text{if $X=H^1(\Delta_\Omega)$ or $\dot H^1(\Delta_\Omega)$}
\end{equation} 
In addition, we assume $r \in [1,\infty]$ satisfies 
\begin{equation}\label{t:HS.mK.WP.H1:cr}
	0\le  \frac1{r}  < \frac{d}2 \left(\frac1{q_c} - \frac1{q} \right). 
\end{equation}
Here, we note that \eqref{l:crtHS.nonlin.est:c2'} and \eqref{t:HS.mK.WP.H1:cr} are the same as \eqref{l:crtHS.nonlin.est:c2} \eqref{condi-new} in Proposition \ref{prop:wellposed1}, respectively.

\medskip

Then we have the following:

\begin{prop}\label{prop:wellposed-A}
Let $d\ge3$ and $0\le \gamma< 2$. Then the following statements hold{\rm :}
\begin{itemize}
\item[(i)] {\rm (}Existence{\rm )} 
For any $u_0 \in X$, there exists a maximal existence time $T_m=T_m(u_0)\in (0,\infty]$ such that 
there exists a unique mild solution 
\[
u\in C([0,T_m); X)\cap \mathcal K^q_r(T_m,\Omega)
\]
to \eqref{crtHS-Dg} with $u(0)=u_0$.
Moreover, we also have 
\[
u \in K^{\tilde{q}}_{\tilde{r}}(T_m,\Omega)
\]
for any $(\tilde{q},\tilde{r}) \in [1,\infty]^2$ satisfying \eqref{l:crtHS.nonlin.est:c}, \eqref{l:crtHS.nonlin.est:c2'} and \eqref{t:HS.mK.WP.H1:cr}.

\item[(ii)] 
{\rm (}Uniqueness in $\mathcal{K}^{q}_r(T,\Omega)${\rm )} 
Let $T>0.$ If $u_1, u_2 \in \mathcal{K}^{q}_r(T,\Omega)$ 
satisfy the integral equation \eqref{integral-eq} with $u_1(0)=u_2(0)=u_0$,  
then $u_1=u_2$ on $[0,T].$ 

\item[(iii)] {\rm (}Continuous dependence on initial data{\rm )} 
The map $T_m : X \to (0,\infty]$ is 
lower semicontinuous. Furthermore, for any $u_0, v_0 \in X$ and for any $T < \min\{T_m(u_0),T_m(v_0)\}$, there exists a constant $C>0$, depending on $\|u_0\|_{X}$, $\|v_0\|_{X}$, and $T$, such that
\[
\sup_{t \in [0,T]} \|u(t) - v(t)\|_{X} 
+
\|u-v\|_{\mathcal K^q_r(T,\Omega)} \le C \|u_0 - v_0\|_{X}.
\]

\item[(iv)] {\rm (}Blow-up criterion{\rm )} 
If $T_m < + \infty,$ then $\|u\|_{\mathcal{K}^q_r(T_m,\Omega)} = \infty.$ 

\item[(v)] {\rm (}Small-data global existence and dissipation{\rm )} 
There exists $\rho >0$ such that 
if $u_0 \in X$ satisfies
\begin{equation}\nonumber
\|e^{t\Delta_\Omega} u_0\|_{\mathcal{K}^{q}_r(\Omega)}
	\le \rho,
\end{equation}
then $T_m=+\infty$ and 
\[
\|u\|_{\mathcal{K}^{q}_r(\Omega)} \le 2\rho\quad \text{and}\quad \lim_{t\to\infty}\|u(t)\|_{X} = 0.
\]

\item[(vi)] Let $d=3$, $r=\infty$ and $X=H^1(\Delta_\Omega)$ or $\dot H^1(\Delta_\Omega)$. Suppose additionally that $q$ satisfies \eqref{l:crtHS.nonlin.est:c3}, i.e.,
\[
	\frac{1}{q_c} - \frac{1}{12(2-\gamma)}
	< \frac{1}{q}. 
\]
Then, for any $u_0 \in X$, there exists a maximal existence time 
$T_m=T_m(u_0)\in (0,\infty]$ such that there exists a unique mild solution 
\[
	u \in C([0,T_m); X)\cap \mathcal K^q(T_m,\Omega)\quad\text{and}\quad 
	\partial_t u \in \mathcal K^{3,1}(T_m,\Omega)
\]
to \eqref{crtHS-Dg} with $u(0)=u_0$. Furthermore, the solution $u$ satisfies 
\[
\partial_t u \in \mathcal K^{2,1}(T_m,\Omega).
\]
Here, the space $\mathcal{K}^{q,\alpha}(T,\Omega)$ is defined by 
\[
   \mathcal{K}^{q,\alpha}(T,\Omega):=\left\{u\in \mathscr{D}'([0,T)\times\Omega)\ ;\, \|u\|_{\mathcal{K}^{q,\alpha}(T',\Omega)}
   <\infty\ \text{for any }T' \in (0,T)\right\}
\]
endowed with 
\[
\|u\|_{\mathcal K^{q,\alpha}(T,\Omega)}
	:=\sup_{0\le t\le T}t^{\frac{d}{2}(\frac{1}{q_c}-\frac{1}{q})+\alpha}\|u\|_{L^q(\Omega)}.
\]
\end{itemize}
\end{prop}

Our task is now reduced to applying the fixed point theorem 
to the map defined by 
\begin{equation}\label{map}
	\Phi_{u_0}[u] = \Phi_{u_0}[u](t) := e^{t\Delta_\Omega} u_0 + N(u),
\end{equation}
where 
\begin{equation}	\label{mapN}
\begin{aligned}
	 N(u)(t) := \int_0^t e^{(t-\tau)\Delta_\Omega} 
		\{F(\cdot, u(\tau)) \}\, d\tau.
\end{aligned}
\end{equation}
For this purpose, let us prepare some estimates for $e^{t\Delta_\Omega}$ and $N(u)$. We recall the pointwise estimates for its integral kernel $G_\Omega(t,x,y)$:
\begin{equation}\label{Gaussian}
0\le G_\Omega(t,x,y) \le (4\pi t)^{-\frac{d}{2}} \exp\left(-\frac{|x-y|^2}{4t}\right),\quad t>0,\ \text{a.e.}\, x,y\in \Omega
\end{equation}
(see, e.g., Sections 2.2, 2.3 and 6.3 in Ouhabaz \cite{Ouh2005}). 
By combining this pointwise estimate with Lemma \ref{l:lin.est.HS}, 
we have linear estimates for $\{e^{t\Delta_\Omega}\}_{t>0}$. 
\begin{lem}\label{lem:smooth-dirichlet}
Let $d\ge1$, $0 < \gamma < d$ and $s\ge0$. Then, the following statements hold{\rm :}
\begin{itemize}
\item[(i)] For any $1\le p_1 \le p_2 \le \infty$,  
there exists $C>0$ such that 
\[
\| (-\Delta_\Omega)^\frac{s}{2} e^{t\Delta_\Omega} f\|_{L^{p_2}(\Omega)} 
\le C  t^{-\frac{d}{2}(\frac{1}{p_1}-\frac{1}{p_2})-\frac{s}{2}} 
	\| f\|_{L^{p_1}(\Omega)}
\]
for any $t>0$ and $f \in L^{p_1}(\Omega)$. 
\item[(ii)] 
Suppose 
\[
0 \le \frac{1}{p_2} < \frac{\gamma}{d} + \frac{1}{p_1} < 1.
\]
Then, 
there exists $C>0$ such that 
\[
\|(-\Delta_\Omega)^\frac{s}{2} e^{t\Delta_\Omega} (|\cdot|^{-\gamma} f) \|_{L^{p_2}(\Omega)} 
    \le C t^{-\frac{d}{2}(\frac{1}{p_1}-\frac{1}{p_2})-\frac{s+\gamma}{2}} 
	\|f\|_{L^{p_1}(\Omega)}
\]
for any $t>0$ and $f \in L^{p_1}(\Omega)$.
\end{itemize}
\end{lem}

\begin{proof}
We prove only the statement (ii), as the proof of (i) is simpler. 
By the property of semigroup and \eqref{Gaussian}, 
we have 
\[
\begin{split}
\|(-\Delta_\Omega)^\frac{s}{2} e^{t\Delta_\Omega} (|\cdot|^{-\gamma} f) \|_{L^{p_2}(\Omega)} 
& \le C t^{-\frac{s}{2}} \|e^{\frac{t}{2}\Delta_\Omega} (|\cdot|^{-\gamma} f) \|_{L^{p_2}(\Omega)} \\
& \le C t^{-\frac{s}{2}} \|e^{\frac{t}{2}\Delta} (|\cdot|^{-\gamma} |\tilde{f}|) \|_{L^{p_2}(\R^d)},
\end{split}
\]
where $\tilde{f}$ is the zero extension of $f$ to $\R^d$. 
By \eqref{decayest2} in Lemma \ref{l:lin.est.HS}, we estimate
\[
\begin{split}
\|e^{\frac{t}{2}\Delta} (|\cdot|^{-\gamma} |\tilde{f}|) \|_{L^{p_2}(\R^d)}
& \le 
C t^{-\frac{d}{2}(\frac{1}{p_1}-\frac{1}{p_2})-\frac{\gamma}{2}} \|\tilde{f}\|_{L^{p_1}(\R^d)}\\
& =
C t^{-\frac{d}{2}(\frac{1}{p_1}-\frac{1}{p_2})-\frac{\gamma}{2}} \|f\|_{L^{p_1}(\Omega)}.
\end{split}
\]
Summarizing the above two estimates, we obtain (ii). 
\end{proof}

\begin{lem}
\label{l:crtHS.nonlin.est}
Let $d\ge3$, $0\le\gamma< 2$, and $T>0$. 
Then, the following statements hold{\rm :}
\begin{itemize}
\item[(i)] 
Assume $(q,r)$ satisfies \eqref{l:crtHS.nonlin.est:c} and \eqref{t:HS.mK.WP.H1:cr}. Then, 
there exists a positive constant $C_1$ depending only on 
$d,$ $\gamma$, and $q$ such that 
\begin{equation}\label{lemA.5-1}
\begin{split}
&\| N(u) - N(v)\|_{\mathcal{K}^{q}_r(T,\Omega)\cap L^{\infty}([0,T];L^{q_c}(\Omega))} \\
&\le C_1 \max\{\|u\|_{\mathcal{K}^{q}_r(T,\Omega)},\|v\|_{\mathcal{K}^{q}_r(T,\Omega)} \}^{2^*(\gamma)-2} 
 \|u-v\|_{\mathcal{K}^{q}_r(T,\Omega)}
\end{split}
\end{equation}
holds for any $u, v\in \mathcal{K}^{q}_r(T,\Omega).$ 
\item[(ii)] 
Assume $(q,r)$ satisfies \eqref{l:crtHS.nonlin.est:c2'} and \eqref{t:HS.mK.WP.H1:cr}. 
Then, there exists a positive constant $C_2$ depending only on 
$d,$ $\gamma$, and $q$ such that 
\begin{equation}\label{lemA.5-2}
\begin{split}
& \| N(u) - N(v) \|_{L^{\infty}([0,T];\dot{H}^1(-\Delta_\Omega)) \cap \mathcal{K}^{2}(T,\Omega)} \\
&\le C_2 \max\{\|u\|_{\mathcal{K}^{q}_r(T,\Omega)},\|v\|_{\mathcal{K}^{q}_r(T,\Omega)} \}^{2^*(\gamma)-2} 
 \|u-v\|_{\mathcal{K}^{q}_r(T,\Omega)}
\end{split}
\end{equation}
holds for any $u, v\in \mathcal{K}^{q}_r(T,\Omega).$ 
\item[(iii)] 
Assume $q$ satisfies \eqref{l:crtHS.nonlin.est:c}, $r=\infty$, and 
\begin{equation}\label{l:crtHS.nonlin.est:c4}
\frac{1}{q_c} - \frac{4-d}{2d(2^*(\gamma)-2)}
< \frac{1}{q}.
\end{equation}
Then, there exists a positive constant $C_3$ depending only on 
$d,$ $\gamma$, and $q$ such that 
\begin{equation}\label{nonlin.est3}
\| \partial_tN(u)\|_{\mathcal K^{d,1}(T,\Omega)\cap \mathcal K^{2,1}(T,\Omega)} 
\le C_3 \left(\|u_0\|_{L^{q_c}(\Omega)}^{2^*(\gamma)-1} + \|u\|_{\mathcal{K}^{q}(T,\Omega)}^{2^*(\gamma)-2}
 \|\partial_tu\|_{\mathcal{K}^{d,1}(T,\Omega)} \right)
\end{equation}
holds for any $u_0 \in L^{q_c}(\Omega)$ and for any $u\in \mathcal{K}^{q}(T,\Omega)$ satisfying the integral equation \eqref{integral-eq} and $\partial_tu\in \mathcal{K}^{d,1}(T,\Omega)$.
\end{itemize}
\end{lem}

\begin{rem}\label{rem:restriction-q}
Note that the statement (iii) in Lemma \ref{l:crtHS.nonlin.est} holds only if $d=3$,
as it is possible to take $q$ satisfying both \eqref{l:crtHS.nonlin.est:c} and \eqref{l:crtHS.nonlin.est:c4} only if $d=3$.
This statement (iii) is a key tool in the proof of (vi) in Proposition \ref{prop:wellposed-A}, and \eqref{l:crtHS.nonlin.est:c4} yields the additional assumption \eqref{l:crtHS.nonlin.est:c3} of (vii) in Proposition \ref{prop:wellposed1}.
\end{rem}

To show Lemma \ref{l:crtHS.nonlin.est}, we use the 
Hardy-Littlewood-Sobolev inequality. 
In \cite{SteWei1958}, Stein and Weiss extend the doubled-weighted 
Hardy-Littlewood-Sobolev inequality to general dimension. 
The endpoint cases $(p,q) = \{1,\infty\}$ are due to \cite{Str1969}. 
Their result reads as follows. 
\begin{lem}[\cite{SteWei1958}, \cite{Str1969}]
\label{l:SW}
Let $0<\lambda<d,$ $1\le p\le q \le \infty,$ $\alpha < d/p',$ $\beta<d/q,$ 
$\alpha+\beta\ge 0$ and $1/q = 1/p + (\lambda + \alpha + \beta)/d -1,$ 
except if $\alpha+\beta=0$ we require $1< p\le q < \infty.$ 
Define the operator $S_{\lambda, \alpha, \beta}(f)$ by 
\begin{equation}\nonumber
	S_{\lambda, \alpha, \beta} (f) 
		= \int_{\R^d} \frac{f(y)}{|x|^{\beta} |x-y|^{\lambda} |y|^{\alpha}} dy.
\end{equation}
Then there exists a constant $C>0$ such that 
\begin{equation}\nonumber
	\|S_{\lambda, \alpha, \beta} (f) \|_{L^q} \le C \|f\|_{L^p}
\end{equation}
for any $f\in L^p(\R^d).$
\end{lem}

\begin{proof}[Proof of Lemma \ref{l:crtHS.nonlin.est}]
First, we prove the assertion (i). We assume $v=0$ for simplicity. 
The proof of (i) with $r=\infty$ can be found in the proof of Theorem 1.1 in \cite{BenTayWei2017}.
Let $r<\infty$. Then, by (ii) in Lemma \ref{lem:smooth-dirichlet}, we estimate
\[
\begin{split}
	\|N(u)(t)\|_{L^q} 
	& \le \int_0^t \|e^{(t-\tau)\Delta_\Omega} \{ |\cdot|^{-\gamma} |u(\tau)|^{2^*(\gamma)-1} \} \|_{L^q} \, d\tau \\
	&\le C \int_0^t (t-\tau)^{-\frac{d(2^*(\gamma)-2)}{2q}- \frac{\gamma}2} 
		\|  |u(\tau)|^{2^*(\gamma)-1} \|_{L^\frac{q}{2^*(\gamma)-1}} \, d\tau \\
	&= C  \int_0^t (t-\tau)^{-\frac{d(2^*(\gamma)-2)}{2q} - \frac{\gamma}2} 
		\tau^{-(2^*(\gamma)-1) \kappa}
		\left( \tau^{(2^*(\gamma)-1)\kappa}\| u(\tau) \|_{L^q}^{2^*(\gamma)-1}
			 \right) \, d\tau,
\end{split}
\]
provided that 
\begin{equation}\nonumber
	0 \le \frac{1}{q} < \frac{\gamma}{d} + \frac{2^*(\gamma) - 1}{q} < 1,
\end{equation}
i.e., 
\[
0 \le \frac1{q} < \frac{d-\gamma}{d(2^*(\gamma) - 1)}.
\]
Let $\tilde f(\tau) := (  \tau^{(2^*(\gamma)-1)\kappa}\| u(\tau) \|_{L^q}^{2^*(\gamma)-1} ) \chi_{[0,T]}(\tau)$ 
with $\chi_{[0,T]}$ the indicator function on $[0,T],$ so that 
\begin{equation}\nonumber
	\|N(u)\|_{\mathcal{K}^q_r(T)} 
	\le C \left\| \int_{\R} t^{\kappa} 
		(t-\tau)^{-\frac{d(2^*(\gamma)-2)}{2q} - \frac{\gamma}2} 
		\tau^{-(2^*(\gamma)-1) \kappa}
		\tilde f(\tau) \, d\tau \right\|_{L^r(\R)}.
\end{equation}
Hence, Lemma \ref{l:SW} with 
\[
\text{$d= 1,$ 
$\lambda = \frac{d(2^*(\gamma)-2)}{2q} +\frac{\gamma}2,$ 
$\beta = - \kappa,$ $\alpha = (2^*(\gamma) - 1)\kappa,$ 
$q = r,$ $p = \frac{r}{2^*(\gamma)-1}$}
\] 
yields 
\[
\|N(u)\|_{\mathcal{K}^q_r(T)} 
\le C \|\tilde f\|_{L^{\frac{r}{2^*(\gamma)-1}}(\mathbb R)}
= C\|u\|_{\mathcal{K}^q_r(T)}^{2^*(\gamma)-1},
\]
provided that 
\begin{equation}\nonumber
\begin{aligned}
&0 < \frac{d(2^*(\gamma)-2)}{2q} + \frac{\gamma}{2} < 1, \quad
	1\le \frac{r}{2^*(\gamma)-1} \le r \le \infty, \\
&(2^*(\gamma)-1)\kappa < 1 - \frac{2^*(\gamma)-1}{r}, \quad
	- \kappa < \frac1{r}, \quad 
	(2^*(\gamma)-2) \kappa \ge 0, \\
&\frac{1}{r} = \frac{2^*(\gamma)-1}{r} 
	+ \frac{d(2^*(\gamma)-2)}{2q} + \frac{\gamma}{2} 
		+ (2^*(\gamma)-2) \kappa - 1,
\end{aligned}
\end{equation}
where 
\begin{equation}\nonumber
	1< \frac{r}{2^*(\gamma)-1} \le r < \infty
		\quad\text{if }
	(2^*(\gamma)-2) \kappa = 0.
\end{equation}
The above conditions amount to \eqref{l:crtHS.nonlin.est:c} and \eqref{t:HS.mK.WP.H1:cr}. 
Thus, the assertion (i) is proved.

The proof of (ii) is similar to that of (i), so we may omit the proof. 
Finally, we give a proof of (iii).
By making the change $\tau'=t-\tau$, 
we write 
\[
\begin{split}
\partial_t \int_0^t e^{(t-\tau)\Delta_\Omega} 
	 F(x, u(\tau))\, d\tau 
=
e^{t\Delta_\Omega}F(x, u_0)
+
\int_0^t e^{\tau'\Delta_\Omega} 
	  \partial_t F(x, u(t-\tau'))\, d\tau'.
\end{split}
\]
Note from the assumption \eqref{assum-F} on $F$ and (ii) in Lemma \ref{lem:smooth-dirichlet} that 
\begin{equation}\label{decayest-gene}
\|(-\Delta_\Omega)^\frac{s}{2} e^{t\Delta_\Omega} (|\cdot|^{-\gamma} F(x,u)) \|_{L^{p_2}(\Omega)} 
    \le C t^{-\frac{d}{2}(\frac{1}{p_1}-\frac{1}{p_2})-\frac{s+\gamma}{2}} 
	\||u|^{2^*(\gamma)-1}\|_{L^{p_1}(\Omega)}
\end{equation}
for any $p_1,p_2,\gamma$ satisfying
\[
0 < \gamma < d,\quad 0 \le \frac{1}{p_2} < \frac{\gamma}{d} + \frac{1}{p_1} < 1,
\]
as it follows from the positivity in \eqref{Gaussian} and the assumption \eqref{assum-F} on $F$ that 
\[
|e^{t\Delta_\Omega} F(x,u)| \le e^{t\Delta_\Omega} (|x|^{-\gamma}|u|^{2^*(\gamma)-1}).
\]
From \eqref{decayest-gene}, the first term is estimated as follows:
\[
\begin{split}
\left\|e^{t\Delta_\Omega}F(\cdot, u_0)\right\|_{L^d(\Omega)}
& \le C
t^{-\frac{d}{2} (\frac{2^*(\gamma)-1}{q_c} - \frac{1}{d})- \frac{\gamma}{2}} \||u_0|^{2^*(\gamma)-1}\|_{L^{\frac{q_c}{2^*(\gamma)-1}}(\Omega)}\\
& = C
t^{-\frac{d}{2} (\frac{1}{q_c} - \frac{1}{d})-1}
\|u_0\|_{L^{q_c}(\Omega)}^{2^*(\gamma)-1}.
\end{split}
\]
As to the second term, again using \eqref{decayest-gene}, noting from the assumptions $F(x,\cdot) \in C^1 (\mathbb R^2; \mathbb R^2)$ and \eqref{assum-F} on $F$ that 
\[
\begin{split}
|\partial_t F(x, u(t-\tau'))| & = |\partial_z F(x, u(t-\tau')) \partial_t u(t-\tau')|\\
& \le C|x|^{-\gamma}|u(t-\tau')|^{2^*(\gamma)-2} |\partial_t u(t-\tau')|,
\end{split}
\]
and then applying H\"older's inequality, 
we estimate
\begin{equation}\label{eq1-second}
\begin{split}
\Big\|
\int_0^t e^{\tau'\Delta_\Omega} &
	  \partial_t F(x, u(t-\tau'))\, d\tau'
\Big\|_{L^d(\Omega)}\\
& \le C \int_0^t 
\tau'^{-\frac{d}{2} (\frac{1}{q_0} - \frac{1}{d}) - \frac{\gamma}{2}} 
\big\||u(t-\tau')|^{2^*(\gamma)-2} \partial_t u(t-\tau')\big\|_{L^{q_0}(\Omega)}\, d\tau'\\
& \le C \int_0^t 
\tau'^{-\frac{d}{2} (\frac{1}{q_0} - \frac{1}{d}) - \frac{\gamma}{2}} 
\|u(t-\tau')\|_{L^q}^{2^*(\gamma)-2}\|\partial_t u(t-\tau')\|_{L^d(\Omega)}\, d\tau',
\end{split}
\end{equation}
where the exponent $q_0$ satisfies  
\[
\frac{1}{d} < \frac{\gamma}{d} + \frac{1}{q_0}<1,\quad \frac{1}{r} = \frac{2^*(\gamma)-2}{q} + \frac{1}{d}.
\]
By the definitions of $\|\cdot\|_{\mathcal K^q(T,\Omega)}$ and $\|\cdot\|_{\mathcal K^{d,1}(T,\Omega)}$, 
the right-hand side of \eqref{eq1-second} is estimated from above as 
\[
C 
\left(
\int_0^t 
\tau'^{-\frac{d}{2} (\frac{1}{q_0} - \frac{1}{d}) - \frac{\gamma}{2}} 
(t-\tau')^{-\frac{d(2^*(\gamma)-2)}{2} (\frac{1}{q_c} - \frac{1}{q} ) -\frac{d}{2} ( \frac{1}{q_c} - \frac{1}{d} ) -1}\, d\tau'
\right)
\|u\|_{\mathcal K^q(t,\Omega)}^{2^*(\gamma)-2}
\|\partial_t u\|_{\mathcal K^{d,1}(t,\Omega)},
\]
where we require that $q$ and $r$ satisfy 
\[
-\frac{d}{2} \left(\frac{1}{q_0} - \frac{1}{d}\right) - \frac{\gamma}{2} > -1,\quad 
-\frac{d(2^*(\gamma)-2)}{2} \left(\frac{1}{q_c} - \frac{1}{q} \right) -\frac{d}{2} \left( \frac{1}{q_c} - \frac{1}{d} \right) -1 > -1
\]
for convergence of the above integral with respect to $\tau'$.
Here, the above four conditions on $q$ and $r$ amount to 
\begin{equation}\label{condi3}
\frac{1}{q_c} - \frac{4-d}{2d(2^*(\gamma)-2)}
< \frac{1}{q} < \frac{1}{q_c},
\end{equation}
and the above integral is calculated as follows:
\[
\int_0^t 
\tau'^{-\frac{d}{2} (\frac{1}{q_0} - \frac{1}{d}) - \frac{\gamma}{2}} 
(t-\tau')^{-\frac{d(2^*(\gamma)-2)}{2} (\frac{1}{q_c} - \frac{1}{q} ) -\frac{d}{2} ( \frac{1}{q_c} - \frac{1}{d} ) -1}\, d\tau' = 
C t^{-\frac{d}{2} (\frac{1}{q_c} - \frac{1}{d})-1}. 
\]
Hence, by combining what has been obtained so far, we obtain 
\[
\Big\| \partial_t \int_0^t e^{(t-\tau)\Delta_\Omega} 
	 F(\cdot, u(\tau)) \,d\tau 
	\Big\|_{\mathcal K^{d,1}(T,\Omega)} 
\le C \left(\|u_0\|_{L^{q_c}(\Omega)}^{2^*(\gamma)-1} + \|u\|_{\mathcal{K}^{q}(T,\Omega)}^{2^*(\gamma)-2}
 \|\partial_tu\|_{\mathcal{K}^{d,1}(T,\Omega)} \right).
\]
Similarly, we can also prove another estimate with $\mathcal K^{2,1}(T,\Omega)$-norm in \eqref{nonlin.est3} under the condition \eqref{condi3}. 
Therefore, we conclude the statement (iii). The proof of Lemma~\ref{l:crtHS.nonlin.est} is finished.
\end{proof}

\begin{proof}[Proof of Proposition \ref{prop:wellposed-A}]
The proofs of (i)--(v) are obtained by combining Lemma \ref{l:crtHS.nonlin.est} and the standard fixed-point argument.
Thus, we may omit the proofs. We give only a sketch of proof of (vi) when $X = \dot H^1(\Delta_\Omega)$. 
Take $\rho >0$ and $M>0$ such that
\begin{equation}\label{rho-M}
\rho\, +\, C_1 M^{2^*(\gamma)-1} \le M\quad \text{and}\quad \max\{C_1,C_3\} M^{2^*(\gamma)-2} \le \frac12, 
\end{equation}
where $C_1$ and $C_3$ are the same constants as those in (i) and (iii) of Lemma \ref{l:crtHS.nonlin.est}, respectively.
Let $A>0$. 
Suppose that $u_0 \in \dot H^1(\Delta_\Omega)$ and $T>0$ satisfy
\begin{equation}\label{t:HS.WP.Lbsg:c2}
\|u_0\|_{\dot H^1(\Delta_\Omega)}\le A\quad \text{and}\quad \|e^{t\Delta_\Omega} u_0\|_{\mathcal{K}^{q}(T,\Omega)}
\le \rho. 
\end{equation}
Define the map $\Phi_{u_0}$ by 
\[
\Phi_{u_0}[u](t) := e^{t\Delta_\Omega}u_0 + \int_{0}^t e^{(t-\tau)\Delta_\Omega} F(x, u(\tau))\, d\tau
\]
for $t\in [0,T]$. Given $B>0$, we define 
\[
Y := \{ u  \ ;\,  \|u\|_{\mathcal K^q(T,\Omega)} \le M, \|\partial_t u\|_{\mathcal K^{3,1}(T,\Omega)} \le B \},
\]
equipped with the metric $\mathrm{d}(u,v) :=  \|u-v\|_{\mathcal K^q(T,\Omega)}$. 
Then, $(Y, \mathrm{d})$ is a complete metric space. By (i) in Lemma \ref{l:crtHS.nonlin.est}, \eqref{rho-M} and \eqref{t:HS.WP.Lbsg:c2}, 
we have
\[
\| \Phi_{u_0}[u] \|_{\mathcal K^q(T,\Omega)}
 \le \| e^{t\Delta_\Omega} u_0  \|_{\mathcal K^q(T,\Omega)} + C_1 \|u\|_{\mathcal K^q(T,\Omega)}^{2^*(\gamma)-1}\le \rho + C_1 M^{2^*(\gamma)-1} \le M
\]
for any $u \in Y$, and 
\[
\begin{split}
\| \Phi_{u_0}[u] - \Phi_{u_0}[v] \|_{\mathcal K^q(T,\Omega)}
& \le C_1 \max\{\|u\|_{\mathcal{K}^{q}(T,\Omega)},\|v\|_{\mathcal{K}^{q}(T,\Omega)} \}^{2^*(\gamma)-2} 
 \|u-v\|_{\mathcal{K}^{q}(T,\Omega)}\\
& \le  C_1 M^{2^*(\gamma)-2} 
 \|u-v\|_{\mathcal{K}^{q}(T,\Omega)} \\
& \le \frac12 \|u-v\|_{\mathcal{K}^{q}(T,\Omega)}
\end{split}
\]
for any $u, v \in Y$.
On the other hand, by (iii) in Lemma \ref{l:crtHS.nonlin.est}, \eqref{rho-M}, and \eqref{t:HS.WP.Lbsg:c2}, we estimate
\[
\begin{split}
\|\partial_t \Phi_{u_0}[u]
\|_{\mathcal K^{3,1}(T,\Omega)}
&\le 
\|e^{t\Delta_\Omega}u_0
\|_{\mathcal K^{3,1}(T,\Omega)} + 
 C_3 \left(\|u_0\|_{L^{q_c}(\Omega)}^{2^*(\gamma)-1} + \|u\|_{\mathcal{K}^{q}(T,\Omega)}^{2^*(\gamma)-2}
 \|\partial_tu\|_{\mathcal{K}^{d,1}(T,\Omega)} \right)\\
& \le 
C_4 \left( \|u_0\|_{\dot H^1(\Delta_\Omega)} + \|u_0\|_{\dot H^1(\Delta_\Omega)}^{2^*(\gamma)-1}\right) +  C_3 \|u\|_{\mathcal{K}^{q}(T,\Omega)}^{2^*(\gamma)-2}
 \|\partial_tu\|_{\mathcal{K}^{d,1}(T,\Omega)}\\
&\le C_4 (A + A^{2^*(\gamma)-1}) + C_3 M^{2^*(\gamma)-2}B\\
&\le \frac{B}{2}+\frac{B}{2} = B
\end{split}
\]
for any $u \in Y$, where we take $B = 2C_4 (A + A^{2^*(\gamma)-1})$. 
Summarizing the estimates obtained so far, we see that $\Phi_{u_0}$ is contractive from $Y$ into itself. 
Therefore, Banach's fixed-point theorem allows us to prove that there exists a function $u\in Y$ such that 
$u=\Phi_{u_0}[u]$. 
Finally, it follows from (ii) and (iii) in Lemma \ref{l:crtHS.nonlin.est} that $u \in C([0,T); \dot H^1(\Delta_\Omega))$ and $\partial_t u \in \mathcal K^{2,1}(T,\Omega)$. 
Thus, we conclude Proposition~\ref{prop:wellposed1}.
\end{proof}

Moreover, we have the following stability result for \eqref{crtHS}.

\begin{prop}\label{prop:perturbation-A}
Let $d\ge3$ and $0\le \gamma <2$. Assume $q\in [1,\infty]$ satisfies \eqref{l:crtHS.nonlin.est:c} or \eqref{l:crtHS.nonlin.est:c2'}, and $r \in [1,\infty)$ satisfies \eqref{t:HS.mK.WP.H1:cr}. 
Let $T \in (0,\infty]$ and 
$v \in C([0,T); X) \cap \mathcal K^q_r(T,\Omega)$ satisfy the equation
\[
	\partial_t v - \Delta_\Omega v = F(x,v) + e
\]
with initial data $v(0) = v_0 \in X$, where $e=e(t,x)$ is a function on $(0,T) \times \Omega$. 
Assume that $v$ satisfies
\begin{equation}\label{propA:ass1}
	\|v\|_{L^\infty([0,T); X)} \le M \quad \text{and}\quad 
	\|v\|_{\mathcal K^q_r(T,\Omega)} \le M.
\end{equation}
Then the following assertions hold:
\begin{itemize}
\item[(i)] 
There exist constants $\delta_0=\delta_0(M)>0$ and $C=C(M)>0$ such that 
the following assertion holds{\rm :}
If the error term $e$ and a function $u_0 \in X$ satisfy
\[
	\delta := \| u_0 - v_0\|_{X} + \left\|\int_0^t e^{(t-s)\Delta_\Omega}(e(s))\,ds\right\|_{\mathcal K^q_r(T,\Omega)}
	 \le \delta_0,
\]
then there exists a unique solution $u$ to \eqref{crtHS-Dg} on $(0,T)\times \Omega$ with $u(0)=u_0$ satisfying 
\[
	\|u-v \|_{L^\infty([0,T); X)\cap \mathcal K^q_r(T,\Omega)}\le C \delta.
\]

\item[(ii)] Let $\tilde{q} \in [1,\infty]$ satisfy \eqref{l:crtHS.nonlin.est:c} or \eqref{l:crtHS.nonlin.est:c2'}. 
Assume further that $(q,r)$ satisfies 
\begin{equation}\label{new1}
\frac{d-2-2\gamma}{2d} - \frac{1}{\tilde{q}} < \frac{1}{q} <\frac{\gamma}{d(2^*(\gamma) - 2)}
\end{equation}
and 
\begin{equation}\label{new2}
\frac{1}{r} < \min \left\{ \frac{d-2}{2} - \frac{d}{2q}, \frac{1}{2^*(\gamma)-2} \right\}.
\end{equation}
Then there exist constants $\delta_0=\delta_0(M)>0$ and $C=C(M)>0$ such that 
the following assertion holds{\rm :}
If the error term $e$ and a function $u_0 \in X$ satisfy
\[
	\delta := \| u_0 - v_0\|_{X} + \left\|\int_0^t e^{(t-s)\Delta_\Omega}(e(s))\,ds\right\|_{\mathcal K^{\tilde{q}}(T,\Omega)}\le \delta_0
\]
then there exists a unique solution $u$ to \eqref{crtHS-Dg} on $(0,T)\times \Omega$ with $u(0)=u_0$ satisfying 
\[
	\|u-v \|_{L^\infty([0,T); X)\cap \mathcal K^{\tilde{q}}(T,\Omega)}\le C \delta.
\]

\end{itemize}
\end{prop}

To prove (ii) in Proposition \ref{prop:perturbation-A}, we need the following estimates.

\begin{lem}\label{lem:Kato-est}
Let $d\ge3$ and $0\le \gamma <2$. Assume $q,\tilde{q}\in [1,\infty]$ satisfies \eqref{l:crtHS.nonlin.est:c} or \eqref{l:crtHS.nonlin.est:c2'}, and $r \in [1,\infty)$ satisfies \eqref{t:HS.mK.WP.H1:cr}.
In addition, we assume that $(q,r)$ satisfies \eqref{new1} and \eqref{new2}.  
Then, there exists a constant $C>0$ such that for any $T>0$,
\[
\left\| 
\int_0^t e^{(t-\tau)\Delta_\Omega} ( |\cdot|^{-\gamma} |v(\tau)|^{2^*(\gamma)-2} |u(\tau)|)\, d\tau
\right\|_{\mathcal K^{\tilde{q}}(T,\Omega)}
\le C \|v\|_{\mathcal K^q_r(T,\Omega)}^{2^*(\gamma)-2} \|u\|_{\mathcal K^{\tilde{q}}(T,\Omega)}
\]
for any $u \in \mathcal K^{\tilde{q}}(T,\Omega)$ and $v \in \mathcal K^q_r(T,\Omega)$.
\end{lem}

\begin{proof}
By combining (ii) in Lemma \ref{lem:smooth-dirichlet} and H\"older's inequality, we estimate
\[
\begin{split}
& \left\| 
\int_0^t e^{(t-\tau)\Delta_\Omega} ( |\cdot|^{-\gamma} |v(\tau)|^{2^*(\gamma)-2} |u(\tau)|)\, d\tau
\right\|_{L^{\tilde{q}}(\Omega)}\\
& \le C
\int_0^t (t-\tau)^{-\frac{d}{2}(\frac{1}{q_0} - \frac{1}{\tilde{q}}) -\frac{\gamma}{2}} 
\| |v(\tau)|^{2^*(\gamma)-2} |u(\tau)| \|_{L^{q_0}}\, d\tau\\
& \le C
\int_0^t (t-\tau)^{-\frac{d(2^*(\gamma)-2)}{2p} -\frac{\gamma}{2}} 
\| v(\tau)\|_{L^q}^{2^*(\gamma)-2} \|u(\tau) \|_{L^{\tilde{q}}}\, d\tau\\
& \le C
\int_0^t (t-\tau)^{-\frac{d(2^*(\gamma)-2)}{2\tilde{q}} -\frac{\gamma}{2}} 
\tau^{-(2^*(\gamma)-2) \kappa - \frac{d}{2}(\frac{1}{q_c} - \frac{1}{\tilde{q}})}
(\tau^\kappa\| v(\tau)\|_{L^q})^{2^*(\gamma)-2}\, d\tau
\times \|u\|_{K^{\tilde{q}}(T,\Omega)},
\end{split}
\]
where $p_0$ satisfies $1/q_0 = (2^*(\gamma)-2)/q + 1/\tilde{q}$, $\kappa = \kappa(q,r)$ is given by \eqref{kappa} and we require that
\[
0 \le \frac{1}{\tilde{q}} < \frac{\gamma}{d} + \frac{1}{q_0} < 1.
\]
Again using H\"older's inequality, 
we have
\[
\begin{split}
& \int_0^t (t-\tau)^{-\frac{d(2^*(\gamma)-2)}{2\tilde{q}} -\frac{\gamma}{2}} 
\tau^{-(2^*(\gamma)-2) \kappa - \frac{d}{2}(\frac{1}{q_c} - \frac{1}{\tilde{q}})}
(\tau^\kappa\| v(\tau)\|_{L^q})^{2^*(\gamma)-2}\, d\tau\\
& \le C
\left\{
\int_0^t \left( 
(t-\tau)^{-\frac{d(2^*(\gamma)-2)}{2\tilde{q}} -\frac{\gamma}{2}} 
\tau^{-(2^*(\gamma)-2) \kappa - \frac{d}{2}(\frac{1}{q_c} - \frac{1}{\tilde{q}})}
\right)^\sigma\, d\tau 
\right\}^{\frac{1}{\sigma}}\times \|v\|_{\mathcal K^q_r(T,\Omega)}^{2^*(\gamma)-2}\\
& \le C t^{-\frac{d}{2}(\frac{1}{q_c} - \frac{1}{\tilde{q}})} \|v\|_{\mathcal K^q_r(T,\Omega)}^{2^*(\gamma)-2}
\end{split}
\]
where $\sigma \in [1,\infty]$ satisfies $1 = (2^*(\gamma)-2)/r + 1/\sigma$ and we require that 
\[
1 - \left(
\frac{d(2^*(\gamma) -2)}{2q} + \frac{\gamma}{2}
\right)\sigma>0
\]
and 
\[
1 - \left\{
(2^*(\gamma) -2 )\kappa + \frac{d}{2}\left( \frac{1}{q_c} - \frac{1}{q} \right)
\right\} \sigma >0.
\]
Here, the above conditions amount to \eqref{new1} and \eqref{new2}.
The proof of Lemma \ref{lem:Kato-est} is finished.
\end{proof}

\begin{proof}[Proof of Proposition \ref{prop:perturbation-A}]
The proof of (i) is the same as that of the stability result with the so-called Strichartz space as an auxiliary space, instead of $\mathcal K^q_r(T)$ (see, e.g., 
Theorem 2.15 in \cite{KM-2006}, Proposition 4.1 in \cite{Mas2017}, and Proposition 2.1 in \cite{GR-2018}).
Moreover, we can also prove (ii) in almost the same way as (i) by use of Lemma~\ref{lem:Kato-est}. 
Let us give only a sketch of proof of (ii). 

To prove (ii), it is sufficient to find a solution $w=w(t)$ to the problem
\begin{equation}\label{pertubation-pro}
\begin{cases}
\partial_t w - \Delta_\Omega w = F(x,v+w) - F(x,v) -e,&(t,x)\in I\times\Omega, \\
w(0) = w_0 := u_0 - v_0 \in X,
\end{cases}
\end{equation}
where $I := (0,T)$. Define 
\[
\Phi (w)(t) := 
e^{t\Delta_\Omega}w_0 +
\int_0^t e^{(t-\tau)\Delta_\Omega} (F(x,v+w) - F(x,v) -e)(\tau)\, d\tau,
\]
Let $m>0$ be chosen later.
We can divide the interval $I$ into a finite number of small intervals $[t_{j-1}, t_{j})$ such that
$0=: t_0 < t_1 < t_2 < \ldots < t_{J-1} < t_J := \sup\, I$ and  
\[
\| \chi_{[t_{j-1}, t_{j})} v\|_{\mathcal K^q_r(T,\Omega)} \le m,\quad j= 1,2,\cdots, J,
\]
where $J \in \mathbb N$ satisfies
\[
\left( \frac{M}{m} \right)^{2^*(\gamma)}
\le J < \left( \frac{M}{m} \right)^{2^*(\gamma)} + 1,
\]
and $\chi_{[t_{j-1}, t_{j})}=\chi_{[t_{j-1}, t_{j})}(t)$ is the characteristic function of $[t_{j-1}, t_{j})$. We write $I_j := [0, t_j)$. 
To construct a solution $w$ to \eqref{pertubation-pro} on $I_1$
we define a complete metric space $X_1$ by 
\[
X_1 := \{ w  \in \mathcal K^{\tilde{q}}(t_1,\Omega) \, ;\, \|w\|_{\mathcal K^{\tilde{q}}(t_1,\Omega)} \le a_1 \},
\]
\[
\mathrm{d}_{X_1}(w_1,w_2) := \|w_1 -w_2\|_{\mathcal K^{\tilde{q}}(t_1,\Omega)}.
\]
Then, by Lemma \ref{lem:Kato-est}, we have
\[
\begin{split}
\|\Phi(w)\|_{\mathcal K^{\tilde{q}}(I_1,\Omega)}
& \le \|e^{t\Delta_\Omega}w_0\|_{\mathcal K^{\tilde{q}}(I_1,\Omega)}
+ C \left( \|v\|_{\mathcal K^q_r(I_1,\Omega)}^{2^*(\gamma)-2} + \|w\|_{\mathcal K^{\tilde{q}}(I_1,\Omega)}^{2^*(\gamma)-2} \right)\|w\|_{\mathcal K^{\tilde{q}}(I_1,\Omega)}\\
& \le \delta + C(m^{2^*(\gamma)-2} + a_1^{2^*(\gamma)-2}) a_1,
\end{split}
\]
\[
\begin{split}
& \mathrm{d}_{X_1} (\Phi_1(w_1), \Phi_1(w_2))\\
& \le C \left( \|v\|_{\mathcal K^q_r(I_1,\Omega)}^{2^*(\gamma)-2} + \|w_1\|_{\mathcal K^{\tilde{q}}(I_1,\Omega)}^{2^*(\gamma)-2} + \|w_2\|_{\mathcal K^{\tilde{q}}(I_1,\Omega)}^{2^*(\gamma)-2} \right)\|w_1-w_2\|_{\mathcal K^{\tilde{q}}(I_1,\Omega)}\\
& \le C(m^{2^*(\gamma)-2} + 2a_1^{2^*(\gamma)-2}) \mathrm{d}_{X_1} (w_1, w_2)
\end{split}
\]
for any $w,w_1,w_2 \in X_1$. Then, choosing $a_1, m, \delta$ as 
\begin{equation}\label{A:condi1}
2Ca_1^{2^*(\gamma)-2} \le \frac14,\quad Cm^{2^*(\gamma)-2} \le \frac14,\quad \delta \le \frac{a_1}{2},
\end{equation}
we obtain
\[
\|\Phi_1(w)\|_{\mathcal K^{\tilde{q}}(I_1,\Omega)} \le a_1,\quad 
\mathrm{d}_{X_1} (\Phi(w_1), \Phi(w_2)) \le \frac12 \mathrm{d}_{X_1} (w_1, w_2),
\]
and hence, $\Phi$ is contractive on $X_1$ and there exists a unique solution $w$ to \eqref{pertubation-pro} on $I_1$ by Banach's fixed point theorem.

For $2 \le j \le J$, we suppose that there exists a unique solution $w$ to \eqref{pertubation-pro} on $I_{j-1}$ such that 
\[
\|w\|_{\mathcal K^{\tilde{q}}(t_{j-1},\Omega)} \le a_{j-1}.
\]
To construct a solution $w$ to \eqref{pertubation-pro} on $I_{j}$, we define a complete metric space $X_j$ by
\[
X_j := \{ w  \in \mathcal K^{\tilde{q}}(t_{j},\Omega) \, ;\, \|w\|_{\mathcal K^{\tilde{q}}(t_j,\Omega)} \le a_j \},
\]
\[
\mathrm{d}_{X_j}(w_1,w_2) := \|w_1 -w_2\|_{\mathcal K^{\tilde{q}}(t_{j-1},\Omega)}.
\]
By Lemma \ref{lem:Kato-est}, we have
\[
\begin{split}
\|\Phi(w)\|_{\mathcal K^{\tilde{q}}(t_j,\Omega)}
& \le \|e^{t\Delta_\Omega}w_0\|_{\mathcal K^{\tilde{q}}(t_j,\Omega)}
+ C \left( \|v\|_{\mathcal K^q_r(t_{j-1},\Omega)}^{2^*(\gamma)-2} + \|w\|_{\mathcal K^{\tilde{q}}(t_{j-1},\Omega)}^{2^*(\gamma)-2} \right)\|w\|_{\mathcal K^{\tilde{q}}(t_{j-1},\Omega)}\\
& + C \left( \| \chi_{[t_{j-1}, t_{j})}v\|_{\mathcal K^q_r(t_j,\Omega)}^{2^*(\gamma)-2} + \| \chi_{[t_{j-1}, t_{j})}w\|_{\mathcal K^{\tilde{q}}(t_j,\Omega)}^{2^*(\gamma)-2} \right)\| \chi_{[t_{j-1}, t_{j})}w\|_{\mathcal K^{\tilde{q}}(t_j,\Omega)}\\
& \le \delta + C(M^{2^*(\gamma)-2} + a_{j-1}^{2^*(\gamma)-2}) a_{j-1} + C(m^{2^*(\gamma)-2} + a_j^{2^*(\gamma)-2}) a_j
\end{split}
\]
for any $w\in X_j$.
Moreover, let $w_1,w_2\in X_j$. Then, noting that $\Phi(w_1) = w = \Phi(w_2)$ by the uniqueness, we have
\[
\begin{split}
\mathrm{d}_{X_j} (\Phi(w_1), \Phi(w_2))
& \le C \Big( \|\chi_{[t_{j-1}, t_{j})}v\|_{\mathcal K^q_r(t_j,\Omega)}^{2^*(\gamma)-2} + \|\chi_{[t_{j-1}, t_{j})}w_1\|_{\mathcal K^{\tilde{q}}(t_j,\Omega)}^{2^*(\gamma)-2} \\
& + \|\chi_{[t_{j-1}, t_{j})}w_2\|_{\mathcal K^{\tilde{q}}(t_j,\Omega)}^{2^*(\gamma)-2} \Big)\|\chi_{[t_{j-1}, t_{j})}(w_1-w_2)\|_{\mathcal K^{\tilde{q}}(t_j,\Omega)}\\
& \le C(m^{2^*(\gamma)-2} + 2a_j^{2^*(\gamma)-2}) \mathrm{d}_{X_j} (w_1, w_2)
\end{split}
\]
Hence, choosing $a_{j-1}, a_j, m, \delta$ so that 
\begin{equation}\label{A:condi2}
2Ca_j^{2^*(\gamma)-2} \le \frac14,\quad Cm^{2^*(\gamma)-2} \le \frac14,\quad \delta \le \frac{a_j}{2},
\quad C(M^{2^*(\gamma)-2} + a_{j-1}^{2^*(\gamma)-2}) a_{j-1} \le \frac{a_j}{8}.
\end{equation}
Then, $\Phi$ is contractive on $X_j$ and there exists a unique solution $w$ to \eqref{pertubation-pro} on $I_j$ by Banach's fixed point theorem.
By induction, we construct a unique solution $w$ to \eqref{pertubation-pro} on $[0,T)$. 
Here, to justify the above argument, we appropriately choose the parameters $\delta_0, m, a_j$ as follows. 
Let $m>0$ be such that $Cm^{2^*(\gamma)-2} = 1/4$.
We define
\[
a_j := a_1 (1+ 8CM^{2^*(\gamma)-2})^{j-1},
\]
and set $a_1 = 2\delta$. Moreover, 
we take $a_1>0$ so that $2Ca_J^{2^*(\gamma)-2} \le 1/4$ (i.e., we take $\delta_0$ sufficiently small). 
If we choose as above, the conditions \eqref{A:condi1} and \eqref{A:condi2} are satisfied. 
Finally, we define $u := v + w$, and then $u$ is a unique solution to \eqref{crtHS-Dg} on $(0,T)\times \Omega$ with $u(0)=u_0$ satisfying 
\[
\|u-v \|_{\mathcal K^{\tilde{q}}(T,\Omega)}\le a_J = C \delta.
\]
Similarly, we also have
\[
\|u-v \|_{L^\infty([0,T); X)}\le C \delta.
\]
Thus, we conclude Proposition \ref{prop:perturbation-A}.
\end{proof}

\medskip

Finally, in this appendix, we state the result on dissipation of global solutions $u$ to \eqref{crtHS-Dg} with $\|u\|_{\mathcal K^q_r(\Omega)} < \infty$.
\begin{prop}
\label{prop:diss-A}
Let $d\ge 3$ and $0\le \gamma <2$. 
Let $u_0\in X$ and $u=u(t)$ be a mild solution to \eqref{crtHS-Dg} with $u(0)=u_0$. 
Then the following assertions hold:
\begin{itemize}
\item[(i)] 
When $\gamma < \infty$, 
the following statements are equivalent{\rm :}
\begin{itemize}
\item[(a)] $T_m=+\infty$ and $\|u\|_{\mathcal K^q_r(\Omega)} < \infty$.
\item[(b)] $\lim_{t\to T_m} \|u(t)\|_{X}=0$.
\end{itemize}

\item[(ii)] When $\gamma = \infty$, 
the following statements are equivalent{\rm :}
\begin{itemize}
\item[(a)] $T_m=+\infty$ and $\|u\|_{\mathcal K^q(\Omega)} < \infty$.
\item[(b)] $\lim_{t\to T_m} \|u(t)\|_{X}=0$.
\item[(c)] $\lim_{t\to T_m} t^{\frac{d}{2} (\frac{1}{q_c}-\frac{1}{q})} \|u(t)\|_{L^{q}(\Omega)}=0$.
\end{itemize}
\end{itemize}
\end{prop}

\begin{proof}
We consider only the assertion (ii) in the case of $X=L^{q_c}(\Omega)$, as the other cases are almost the same.
By small-data dissipation (iv) in Proposition \ref{prop:wellposed-A}, the statement (a) follows from (b). 
By the blow-up criterion (iii) in Proposition \ref{prop:wellposed-A}, it is clear that (c) implies (a). 
We prove only that (a) implies (c), as the proof of the remaining case is similar. 

We suppose (a). 
The solution $u$ to \eqref{crtHS-Dg} is written as
\[
\begin{split}
u(t)  & =
e^{t\Delta_\Omega} u_0 + \int_{0}^{t'} e^{(t-\tau)\Delta_\Omega}F(x,u(\tau))\, d\tau +
\int_{t'}^t e^{(t-\tau)\Delta_\Omega}F(x, u(\tau))\, d\tau\\
& =:  I(t) + I\hspace{-1pt}I(t) + I\hspace{-1pt}I\hspace{-1pt}I(t)
\end{split}
\]
for $t'\in (0,t)$. 
Let $\varepsilon>0$ be fixed and $\alpha := d(1/q_c -1/q)/2$. 
Since $L^{q_c}(\Omega)$ is dense in $C^\infty_0(\Omega)$, there exists $v_\varepsilon \in C^\infty_0(\Omega)$ such that 
$C_5\| u_0 - v_\varepsilon\|_{L^{q_c}(\Omega)} < \varepsilon/2$, where $C_5$ is the constant in \eqref{c_5} below. 
By (i) in Lemma \ref{l:lin.est.HS}, there exist a constant $C_5>0$, independent of $\varepsilon$, and a time $t_1 = t_1(\varepsilon)>0$ such that 
\begin{equation}\label{c_5}
t^{\alpha}\| I(t) \|_{L^q(\Omega)} \le 
C_5 \| u_0 - v_\varepsilon\|_{L^{q_c}(\Omega)} + C_5t^{-\beta} \|v_\varepsilon\|_{L^1(\Omega)} \le \varepsilon\quad \text{for any $t>t_1$,}
\end{equation}
where $\beta$ is a positive real number given by $\beta= d(1-1/q)/2 -\alpha$. 
Next, we consider the second term $I\hspace{-1pt}I(t)$. We write
$
I\hspace{-1pt}I(t) 
= e^{(t-t')\Delta}w(t')
$
for $t'\in (0,t)$, where $w(t')$ is given by
\[
w(t'):= \int_{0}^{t'} e^{(t'-\tau)\Delta_\Omega}F(x,u(\tau))\, d\tau.
\]
We take $t' = t - t^\delta$ with $\delta\in (0,1)$ such that $d\delta (1 - 1/q)/2 - \alpha >0$.
Since $w(t')\in L^{q_c}(\Omega)$ for any $t'\in (0,t)$, we can apply the same argument as the proof of $I(t)$ to $I\hspace{-1pt}I(t)$, and hence, there exists a positive time $t_2 = t_2(\varepsilon)>0$ such that 
\[
t^{\alpha}\| I\hspace{-1pt}I(t) \|_{L^q(\Omega)} 
\le \varepsilon\quad \text{for any $t>t_2$.}
\]
Finally, we estimate the third term $I\hspace{-1pt}I\hspace{-1pt}I(t)$ as 
\[
t^\alpha\| I\hspace{-1pt}I\hspace{-1pt}I(t)\|_{L^{q}(\Omega)} 
\le 
C \|u\|_{\mathcal K^q(\Omega)}^{2^*(\gamma)-1}\int_{\frac{t'}{t}}^1 A(\tau)\,d\tau
\]
for any $t'\in (0,t)$, where 
\[
A(\tau) := (1-\tau)^{-\frac{d}{2}(\frac{2^*(\gamma)-1}{q}-\frac{1}{2})
					-\frac{1+\gamma}{2}} 
	\tau^{-\frac{d(2^*(\gamma)-1)}{2}(\frac{1}{q_c}-\frac{1}{q})}.
\]
It is seen from the assumption \eqref{l:crtHS.nonlin.est:c} on $q$ that 
\[
\int_0^1 A(\tau)\,d\tau<\infty.
\]
Hence we have 
\[
\int_{\frac{t'}{t}}^1 A(\tau)\,d\tau \to 0\quad \text{as }\frac{t'}{t} \to 1,
\]
as $t' = t-t^\delta$. Then 
there exists $t_3 = t_3(\varepsilon) >0$ such that
\[
t^\alpha\| I\hspace{-1pt}I\hspace{-1pt}I(t) \|_{L^{q}(\Omega)} < \varepsilon \quad \text{for any $t>t_3$}.
\]
Combining the estimates obtained so far, we obtain (c). 
Thus, we conclude Proposition~\ref{prop:diss-A}.
\end{proof}

\section{Comparison principle}\label{app:B}
In this appendix, we assume that $\Omega_1$ and $\Omega_2$ are domains of $\R^d$ such that $\Omega_1\subset\Omega_2$. 
Then, it is known that the pointwise estimates
\[
0\le G_{\Omega_1}(t,x,y) \le G_{\Omega_2}(t,x,y),\quad t>0,\ \text{a.e.}\, x,y\in \Omega_1
\]
hold (see, e.g., Ouhabaz \cite{Ouh2005}). 
These yield the comparison
\begin{equation}\label{comparison-l}
|e^{t\Delta_{\Omega_1}} (u_{0,1})(x)|
\le e^{t\Delta_{\Omega_2}} (u_{0, 2})(x)
\end{equation}
for any $t>0$ and almost everywhere $x\in\Omega_1$, where $u_{0,j} \in L^{q_c}(\Omega_j)$ are such that $|u_{0,1}| \le u_{0,2}$ for almost everywhere $x\in\Omega_1$, which means that $u_{0,2}$ is real and positive on $\Omega_1$. 
This comparison \eqref{comparison-l} is inhibited by the Dirichlet problems \eqref{crtHS-Dg} of nonlinear heat equations with source terms $F$.

\begin{lem}\label{lem:comparison}
Let $j=1,2$ and $T\in (0,\infty]$. 
Suppose that $F_j : \Omega_j \times \mathbb C \to \mathbb C$ satisfies \eqref{assum-F} and  
\[
F_j(x,z) \in \mathbb R\quad \text{for }z\in\mathbb R,\quad 
F_j(x,z)\ge 0 \quad \text{for }z\ge0.
\]
Let $u_j$ be a solution to \eqref{crtHS-Dg} on $[0,T)\times \Omega_j$ with the nonlinear term $F_j$ and initial data $u_{0,j} \in L^{q_c}(\Omega_j)$. 
In addition, assume that 
\begin{equation}\label{assum:F_1F_2}
|F_1(x,z_1)| \le F_2(x,z_2),\quad x\in \Omega_1,\quad z_1,z_2\in\mathbb C\text{ with }|z_1|\le z_2,
\end{equation}
\begin{equation}\label{assum:initialdata}
|u_{0,1}(x)| \le u_{0,2}(x),\quad \text{a.e.}\, x\in \Omega_1.
\end{equation}
Then 
\[
|u_1(t,x)|\le u_2(t,x),\quad t\in [0,T),\ \text{a.e.}\, x\in \Omega_1.
\]
\end{lem}

\begin{proof}
Let $j=1,2$ and $\{v_{j,n}\}_{n=1}^\infty$ be an iterative sequence of $u_j$ defined by 
\[
v_{j,1} (t) := e^{t\Delta_{\Omega_j}} u_{0.j},
\]
\begin{equation}\label{iteration}
v_{j,n} (t) := e^{t\Delta_{\Omega_j}}v_{j,n-1} + \int_{0}^t e^{(t-\tau)\Delta_{\Omega_j}} F_j(x,v_{j,n-1}(\tau)) \, d\tau,\quad n\ge2.
\end{equation}
By Proposition \ref{prop:wellposed-A}, after possibly passing to a subsequence, 
\[
v_{j,n}(t,x) \to u_j(t,x)\quad \text{as }n\to\infty
\]
for any $t\in [0,T)$ and almost everywhere $x\in\Omega_1$.
Hence, to prove Lemma \ref{lem:comparison}, it suffices to show that 
\begin{equation}\label{goal-B}
|v_{1,n} (t,x)|\le v_{2,n} (t,x),\quad t\in [0,T],\ \text{a.e.}\, x\in \Omega_1
\end{equation}
for any $n\in\mathbb N$. 
When $n=1$, it follows from \eqref{comparison-l} and \eqref{assum:initialdata} that $|v_{1,1} (t,x)|\le v_{2,1} (t,x)$ for any $t\ge 0$ and almost everywhere $x\in\Omega_1$. 
Suppose that \eqref{goal-B}
holds for a fixed $n\ge2$. 
Then, by combining the recurrence formula \eqref{iteration} with \eqref{comparison-l}, 
\eqref{assum:F_1F_2} and \eqref{goal-B} with $n$, we have \eqref{goal-B} with $n+1$, i.e.,
\[
|v_{1,n+1} (t,x)|\le v_{2,n+1} (t,x),\quad t\in [0,T),\ \text{a.e.}\, x\in \Omega_1.
\]
Hence, by induction, we obtain \eqref{goal-B}
for any $n\in\mathbb N$. 
Thus, we conclude Lemma~\ref{lem:comparison}.
\end{proof}

\section*{Acknowledgement}
\par
The first author is supported by Grant-in-Aid for Young Scientists (B) 
(No. 17K14216) and Challenging Research (Pioneering) (No.17H06199), 
Japan Society for the Promotion of Science. 
The second author is supported by JST CREST (No. JPMJCR1913), Japan and 
the Grant-in-Aid for Scientific Research (B) (No.18H01132) and 
Young Scientists Research (No.19K14581), 
JSPS.
The third author is supported by Grant-in-Aid for JSPS Fellows 
(No.19J00206), JSPS.


\end{document}